\documentclass[11pt, oneside]{amsart}
\usepackage[utf8]{inputenc}

\usepackage[margin=1in]{geometry}  
\usepackage{graphicx}

\usepackage{indentfirst,csquotes}
\usepackage[round,authoryear]{natbib}
\usepackage{float}

\usepackage{amssymb,amsbsy,amsmath,amscd,amsthm,amsfonts,MnSymbol,mathrsfs}
\usepackage{xcolor,paralist,titlesec,fancyhdr,etoolbox}
\usepackage{booktabs}
\usepackage{secdot}
\newtheorem{theorem}{Theorem}[]
\newtheorem{lemma}{Lemma}
\newtheorem{proposition}{Proposition}

\newtheorem{assumption}{Assumption}
\newtheorem{remark}{Remark}

\theoremstyle{definition}   
\newtheorem{example}{Example}
\newtheorem{definition}{Definition}

\let\P\relax
\newcommand{\P}{\mathbb{P}}

\newcommand{\E}{\mathbb{E}}
\newcommand{\R}{\mathbb{R}}
\newcommand{\Var}{\operatorname{Var}}
\newcommand{\Cov}{\operatorname{Cov}}
\newcommand{\PIF}{\operatorname{PIF}}
\newcommand{\IF}{\operatorname{IF}}
\newcommand{\loss}{\rho}
\newcommand{\score}{\psi}
\newcommand{\closs}{\bar{\loss}}
\newcommand{\mposterior}{\pi_n^{\loss}}
\newcommand{\centeredmposterior}{\pi_n^{\closs}}
\newcommand{\corruptedsample}{X^{(n,m)}}
\newcommand{\wempirical}{F_n^{\boldsymbol{\alpha}}}
\newcommand{\mestimator}{\hat{\theta}^{\boldsymbol{\alpha}}_{\loss }}

\newcommand{\mathbbm}[1]{\text{\usefont{U}{bbm}{m}{n}#1}} 

\usepackage[colorlinks=true,
            citecolor=blue,
            linkcolor=black,
            urlcolor=blue]{hyperref}

\usepackage{cleveref}

\usepackage[
    textsize=scriptsize,
    figwidth=0.99\linewidth,
]{todonotes}

\makeatletter
  \@mparswitchfalse
\makeatother
\normalmarginpar  

\titleformat{\section}[display]{\normalfont\huge\bfseries\centering}{\centering\chaptertitlename\thechapter}{10pt}{\Large}
\titlespacing*{\section}{0pt}{0ex}{0ex}

\usepackage{titlesec}

\titleformat{\section}[block]
  {\normalfont\centering}   
  {\thesection.}                           
  {1ex}                                    
  {\MakeUppercase}                         
  []                                       

\titlespacing*{\section}{0pt}{1em}{0em}    

\titleformat{\subsection}[block]            
  {\normalfont\bfseries}                    
  {\thesubsection.}                         
  {0.5em}                                   
  {}                                        

\titlespacing*{\subsection}{0pt}{1em}{0em}

\hypersetup{ colorlinks=true, linkcolor=black, filecolor=black, urlcolor=black }

\usepackage{lipsum}

\begin{document}
\title{A theoretical framework for M-posteriors: frequentist guarantees and robustness properties} 
\author[J.\ Marusic]{Juraj Marusic}
\address{%
  Department of Statistics\\
  Columbia University\\
  New York, NY 10027, USA
}
\email{juraj.marusic@columbia.edu}
\author[M. Avella Medina]{Marco Avella Medina}
\address{%
  Department of Statistics\\
  Columbia University\\
  New York, NY 10027, USA
}
\email{marco.avella@columbia.edu}
\author[C. Rush]{Cynthia Rush}
\address{%
  Department of Statistics\\
  Columbia University\\
  New York, NY 10027, USA
}
\email{cynthia.rush@columbia.edu}

\thanks{This research was partially
supported by  NSF grants DMS-2310973 (Avella Medina) and DMS-2413828 (Rush). The authors are grateful to Heather Battey and Elvezio Ronchetti for insightful feedback on a preliminary version of this paper. }   
 
\date{\today}

\begin{abstract}
We provide a theoretical framework for a wide class of generalized posteriors that can be viewed as the natural Bayesian posterior counterpart of the class of M-estimators in the frequentist world. We call the members of this class M-posteriors and show that they are asymptotically normally distributed under mild conditions on the M-estimation loss and the prior. In particular, an M-posterior contracts in probability around a normal distribution centered at an M-estimator, showing frequentist consistency and suggesting some degree of robustness depending on the reference M-estimator. We formalize the robustness properties of the M-posteriors by a new characterization of the posterior influence function and a novel definition of  breakdown point adapted for posterior distributions. We illustrate the wide applicability of our theory in various popular models and illustrate their empirical relevance in some numerical examples.

\end{abstract} 

\maketitle

\section{Introduction}
Modern Bayesian methods provide a rich set of data analysis tools that are very popular in statistics \citep{gelmanetal2013} and machine learning \citep{bishop2006, murphy::mlbook} across many disciplines, such as natural language processing \citep{blei2003}, genomics \citep{larget1999}, and epidemiology \citep{best2005}.
However, in the presence of outliers or under model misspecification, classical Bayes estimators constructed using the conventional posterior distribution may be fragile. To date, there have been several approaches studied in the literature for creating more robust Bayesian procedures. 
The classical Bayesian way of handling data suspected to be contaminated with outliers either constructs posteriors using heavy-tailed models, employs mixture models where the contamination appears explicitly as a mixture component, or uses priors that penalize large parameter values \citep{berger1994,andrade:ohagan2006}. Despite these attempts however, according to Huber  \cite[Chapter 15]{huber:ronchetti2009},  a   robustness theory for Bayesian statistics has remained elusive, or at least philosophically distant from the  foundational principles of robust statistics.

More recent approaches for managing outliers have tried to reconcile the Bayesian paradigm with some traditional robust statistics concepts that are rooted in the frequentist paradigm \citep{huber1981,huber:ronchetti2009,hampeletal1986,maronnaetal2019}. These efforts have led to defining notions of qualitative robustness, influence functions and breakdown points for some non-standard Bayesian methods including the disparity-based posteriors of \citep{hooker:vidyashankar2014, ghosh:basu2016,matsubara:knoblauch:briol:oates2022}, the coarsened posterior of \cite{miller:dunson2019} and Gaussian processes methods \citep{altamirano:briol:knoblauch2023}. All these papers emphasize the role of carefully designed robust losses over building  more complex models such as mixtures or choosing carefully constructed priors.
Our work follows this path and systematically connects robustness properties of a class of generalized posterior distributions to the standard M-estimation theory in (frequentist) robust statistics. While in the Bayesian paradigm the parameters are viewed as random and one seeks to quantify uncertainty about them using the data, in frequentist statistics, these parameters are considered fixed and the goal is to estimate them with point estimates. With this correspondence in mind, the posterior distributions that we study can be viewed as the natural Bayesian counterparts to M-estimators in the frequentist world; hence, we call them M-posteriors. In more detail, M-posteriors are obtained by combining a prior distribution on the parameters with a Gibbs measure that is constructed using an empirical loss function that defines an M-estimator.

We study the robustness properties of Bayesian M-posteriors in a number of ways. First, we establish a frequentist  asymptotic theory describing the contraction of the M-posterior distribution by means of a Bernstein-von Mises (BvM) theorem. This result mirrors the standard asymptotic normality theory for M-estimators. Furthermore, we obtain general robustness assessments of M-posteriors by virtue of a new, but natural, characterization of \emph{the posterior influence function} and a novel conception of \emph{posterior breakdown point}. 
In more detail, our posterior influence function measures how much the M-posterior distribution  changes under infinitesimal contamination of the data while the posterior breakdown point measures how many arbitrarily bad observations are needed before the M-posterior gives arbitrarily bad results. We introduce these ideas more formally in Section~\ref{sec:robust_statistics}.
As we will show, it turns out that all of the asymptotic and robustness properties of the Bayesian M-posteriors we consider in this work are connected to a fundamental quantity that is also of interest for M-estimators: the score function. Beyond just the score, our analysis also highlights the role played by the choice of the prior in the properties and behavior  of M-posteriors.

While robust statistics is rooted in frequentist ideas,  the tools we introduce in this work are completely model and paradigm agnostic:  they do not assume an underlying Bayesian or frequentist data generating process. Indeed,  both of our notions of robustness, namely, the posterior influence function and the posterior breakdown point, are characterized mathematically as  functionals of the empirical distribution induced by the observed data. We  believe this makes them
natural approaches for quantifying the robustness of the posterior distribution  to outliers.

Our asymptotic theory builds on a long tradition of BvM results \citep{lecam1953,freedman1963,Vaart_1998}; in particular, on recent work by \cite{chernozhukov:hong2003,kleijn:vandervaart2012,wang:blei2019-JASA, miller2021,avellamedinaetal2021}. Various forms of generalized posteriors have appeared over the years \citep{zhang1999,chernozhukov:hong2003,bissiri:holmes:walker2016}, including some interesting work on Bayesian quantile regression \citep{yu:moyeed2001,yang:wang:he2016}, and with some increased  interest in recent years on power/fractional/tempered posteriors \citep{grunwald2012, grunwald:ommen2017, holmes:walker2017,higginsetal2017,  miller:dunson2019, avellamedinaetal2021,ruchira,mclatchie:fong:frazier:knoblauch2025} and divergence-based posteriors motivated by their robustness to outliers \citep{hooker:vidyashankar2014,ghosh:basu2016,nakagawa:hashimoto2020,matsubara:knoblauch:briol:oates2022, altamirano:briol:knoblauch2023}. 

Our work was inspired  by some core asymptotic ideas and initial robustness assessments for Bayesian methods existing in the literature \cite{chernozhukov:hong2003, kleijn:vandervaart2012,hooker:vidyashankar2014,ghosh:basu2016,wang:kucukelbir:blei2017,matsubara:knoblauch:briol:oates2022}. We hope to contribute to this emerging literature by providing a general framework for analyzing robustness in Bayesian procedures through the lens of our Bayesian M-posteriors. We highlight the following main aspects of our contributions:
\begin{enumerate}[(a)]
    \item \textbf{\emph{Frequentist guarantees:}} we show that M-posteriors are consistent and asymptotically normally distributed under a standard frequentist data generating process and some minimal regularity conditions on the M-estimation loss and prior. This result is formally stated as a BvM theorem for a class of weighted M-posteriors where, in addition to an arbitrary loss function, we also introduce weights for each observation.  Special cases of this analysis include the BvM-type results for  alpha-posteriors from \cite{avellamedinaetal2021} and generalized posteriors from \cite{chernozhukov:hong2003}.  Interestingly, introducing multiple weights affects both the location and the variance of the limiting distribution, contrary to the result with just a single tempering weight, where only the limiting variance is affected.  We also show how robustifying a posterior can lead to contraction around the wrong parameter value and propose a simple bias correction that is inspired by a well-known Fisher consistency correction introduced by \cite{huber1964} in the context of M-estimation.
    \item 
\textbf{\emph{Posterior influence function:}} we characterize the infinitesimal robustness to outliers of M-posteriors by deriving their influence function.  Our characterization of the posterior influence function, inspired by ideas first considered in \cite{ghosh:basu2016, matsubara:knoblauch:briol:oates2022}, is completely model agnostic and serves as a tool to assess the sensitivity of the posterior distribution to infinitesimal perturbations. Our results for M-posterior connect the boundedness of the posterior influence function to the boundedness of the score function of the corresponding loss, which is analogous to known influence function results for M-estimators in the frequentist setting \citep{hampeletal1986}.  We show that a bounded score function is also a necessary condition for an M-posterior to have a bounded influence function and emphasize the importance of the prior in the case where the reference M-estimator is not defined by a convex loss.  We also study  the influence functions of important posterior functionals such as the posterior mean and posterior quantiles. 
    \item 
\textbf{\emph{Posterior breakdown point:}} to the best of our knowledge, this is the first work to define a Bayesian counterpart of the finite sample breakdown point, which we call the \emph{posterior breakdown point}.  This global  measure of robustness quantifies how many data points can be arbitrarily perturbed before the posterior density itself is moved arbitrarily.  Our approach leverages ideas derived in frequentist settings.  Namely, we build on concepts introduced in the work of \cite{huber1984}, which derives results for estimators arising from both convex losses and losses with redescending score functions.  
Once again, our analysis demonstrates the importance of the prior in M-posterior robustness, which leads to different conclusions from those  corresponding to M-estimators in the frequentist setting.  In the case of uninformative priors, we retrieve similar results to those in \cite{donoho:huber1983} for convex M-estimators.  Namely, if the convex loss defining the M-posterior has a bounded score, the posterior breakdown point is $1/2$.   Interestingly, priors with lighter than exponential tails lead to a strange phenomenon when combined with robust convex losses: the posterior breakdown point does not exist, in the sense that by moving all of our data points, we cannot make the posterior arbitrarily bad. Similarly, M-posteriors associated to bounded loss functions like Tukey's loss or the Huber skip loss cannot be broken. This is an undesirable property that suggests that in the context of M-posteriors one should only consider robust losses that lead to Gibbs measures that can be viewed as likelihoods. 
We extend our posterior breakdown point results to posterior functionals, namely posterior means and posterior quantiles, showing that these functionals inherit the breakdown point of the M-posterior.
\end{enumerate}

\section{Preliminaries and motivation}
\label{sec:robust_statistics}

Robust statistics is a mature field of mathematical statistics that was pioneered by the groundbreaking work of \cite{huber1964,hampel1968}. Book-length expositions on the topic include \citep{huber1981, huber:ronchetti2009, hampeletal1986, maronnaetal2019}. See \cite{avellamedina:ronchetti2015} for a short overview that covers all the key concepts introduced in this section.  The primary goal in  robust statistics is to develop methods that give stable results even in settings where deviations from the stochastic assumptions of the model occur.
The field of robust statistics provides  a mathematical framework both to account for data corruptions and analyze the effect of such corruptions on statistical methods.  In this work, we study two classical tools for quantifying robustness in the robust statistics literature: the influence function and the breakdown point.

\emph{Notation and statistical framework.} 
Let  $\mathcal{F}_n = \{f_n(\cdot \mid \theta): \theta \in\Theta\subset \mathbb{R}^p \}$, where $\Theta\subset\mathbb R^p$ is the parameter space, be a parametric family used as a statistical model for the i.i.d.\ random sample $X^n = (X_1,\dots,X_n) \in \mathcal{X}^n$, where $\mathcal{X}\subset\mathbb{R}^d$ denotes the sample space. This model will be assumed to be well specified as the true density $f_n(\cdot \mid \theta^*)$ of the random sample $X^n$ belongs to $\mathcal{F}_n$. We will be particularly interested in estimating the true parameter $\theta^*$. We let $F_n$ denote the empirical distribution function induced by the random sample $X_1,\dots,X_n$. Sometimes we will use $\mathcal{F}$ to denote a generic space of distributions, and we are often interested in studying functionals of the form $T:\mathcal{F}\to\Theta$. We occasionally also consider the corresponding statistics $T:\mathcal{X}^n\to\Theta$, which slightly overloads the notation but keeps the presentation simple.

\emph{The influence function.} A fundamental idea  in robust statistics is to study a statistic of interest as a functional of an underlying data-generating distribution, and the influence function is a tool used to gauge the robustness of such a statistical functional in an infinitesimal sense.

\begin{definition}
 The influence function of a functional $T$ at a point $x\in\mathcal{X}$ for a distribution $F$ is the G{\^a}teaux derivative  
\begin{equation*}
    \mathrm{IF}(x;T,F):=\lim_{\epsilon\to 0+}\frac{T(F_\epsilon)-T(F)}{\epsilon},
\end{equation*} 
where $F_\epsilon=(1-\epsilon)F+\epsilon\delta_x$ and $\delta_x$ is a mass point at $x$. 
\end{definition}
An appealing feature of the  influence function is that it can be interpreted as describing the effect of an infinitesimal contamination at the point $x$ on a statistical functional. Indeed, if a functional $T(F)$ is sufficiently regular, a von Mises expansion \citep{vonmises1947,hampel1974, hampeletal1986} yields
\begin{equation}
\label{vonMises}
T(G) = T(F)+\int \IF(x; T,F)\mathrm{d}(G-F)(x) +o\big(\|F-G\|_\infty\big).
\end{equation}
Considering the neighborhood $\mathcal{F}_\epsilon=\{F^{(\epsilon)}| F^{(\epsilon)}=(1-\epsilon)F+ \epsilon G, ~$G$ \mbox{ an arbitrary distribution}\}$ and the  approximation in \eqref{vonMises},   we see that the influence function can be used to linearize the ``bias'' of $T(F)$ in the neighborhood $\mathcal{F}_\epsilon$. Hence, a statistical functional with a bounded influence function will have a bounded approximate bias in a neighborhood of $F$ and statistical functionals with this property are called B-robust in the literature \citep{hampeletal1986}. 

\emph{The breakdown point.}  The breakdown point, another fundamental tool for quantifying robustness, was introduced by \cite{hampel1968,hampel1971} in what is now called the asymptotic or population form. The perhaps more popular finite sample version of the breakdown point, introduced later in \cite{donoho:huber1983}, answers the following general question: given a sample $X^n = (X_1,\dots,X_n) \in \mathcal{X}^n$, how many arbitrarily bad observations can a statistic $T(X^n)$ tolerate before it gives arbitrarily bad results?

\begin{definition}
    \label{def:BP}
    The \emph{finite sample breakdown point} of a statistic $T:\mathcal{X}^n\to\mathbb{R}^p$ at a given sample $X^n$ is the fraction
\[
    \varepsilon^*(T,X^n)
    :=
    \min\biggl\{\frac{m}{n}:\ 
    \sup_{X^{(n,m)} \, \in  \, \mathcal{B}_H(X^n,m)} \bigl\|\,T(X^{(n,m)})-T(X^n)\bigr\|_2 = \infty\biggr\},
\]
where $\mathcal{B}_H(X^n,m)=\{\tilde{X}^{n}\in \mathcal{X}^n: \sum_{i=1}^n\mathbbm{1}\{\tilde{X}_i\neq X_i\}\leq m\}$ is the collection of datasets of size $n$ such that $m$ or fewer data points are different from the given sample $X^n$.
\end{definition}

Intuitively, a breakdown point of $1/2$ is the maximal value one can expect. For instance, it is well known that the breakdown point of any translation-equivariant location estimator is at most $1/2$ \citep{donoho:huber1983}.

\subsection{M-estimators} 
M-estimators are a broad class of estimators that generalize the usual maximum likelihood estimators. They are naturally appealing  for robust statistics \citep{huber1964,huber:ronchetti2009} and will serve as motivation for our robust Bayesian posteriors. In particular, we are interested in M-estimators $\hat\theta=T(F_n)$ defined as minimizers of the form
\begin{equation}
\label{eq:M-estimator}
    \hat\theta
    =
    \mbox{argmin}_{\theta\in \Theta}\frac{1}{n}\sum_{i=1}^n\rho(X_i,\theta)
    =
    \mbox{argmin}_{\theta\in \Theta} \mathbb{E}_{F_n}[\rho(X,\theta)]
\end{equation}
where $\rho:\mathcal{X}\times\Theta\to \mathbb{R}_{\geq 0}$ is a loss function. 
Under mild conditions, $\rho$ is differentiable and convex,  we can also see $\hat\theta$ as the solution to the estimating equation $
\frac{1}{n}\sum_{i=1}^n\psi(X_i,\hat\theta)=0$, where $\psi(x,\theta)=\frac{\partial}{\partial\theta}\rho(x,\theta)$ is called a score function.

Assuming an i.i.d.\ random sample from distribution $F$, under some standard and mild conditions \citep[Ch.\ 6]{huber:ronchetti2009}, including $\theta^*=\mbox{argmin}_\theta\mathbb{E}_F[\rho(X,\theta)]$ and $\mathbb{E}_F[\psi(X,\theta^*)]=0$, we have that $\hat\theta$ is asymptotically normally distributed as $n\to\infty$. More precisely,
\begin{equation*}
\sqrt{n}(\hat\theta-\theta^*)\to_{d} N(0,V(T,F)),
\end{equation*}
where $V(T,F)=\mathbb{E}_F[\IF(X;T,F)\IF(X;T,F)^\top]$ and the influence function is shown to be equal to 
 \begin{equation}
 \label{IF}
 \IF(x;T,F)= \Big(M(T,F)\Big)^{-1}\psi(x,T(F)),
 \end{equation}
 where $M(T,F)=-\frac{\partial}{\partial \theta}\mathbb{E}_F[\psi(X,\theta)]\big|_{\theta=\theta^*}$. Consequently,  M-estimators defined by bounded score functions  $\psi$ are said to be B-robust. In the case of one-dimensional location models where $\psi(x,\theta)=\psi(x-\theta)$, \cite{donoho:huber1983, huber1984} also showed that a bounded $\psi$ also guarantees a finite sample breakdown point of $1/2$. In general dimension the results are not as simple and a bounded $\psi$ is in general not enough to guarantee a breakdown point of $1/2$ \citep{maronna1976, rousseeuw:yohai1984,rousseeuw1984,davies1987,yohai1987,lopuhaa:rousseeuw1991}.

 \subsection{M-posteriors}
The main statistical objects of interest in this work are generalized posteriors of the form
\begin{equation}
\label{eq:M-posterior}
    \pi^\loss_{n}(\theta \mid F_n) \equiv \frac{\exp \left ( -n\mathbb{E}_{F_n}[\loss(X, \theta)]  \right ) \pi(\theta)}{\int_{\Theta} \exp \left ( -n \mathbb{E}_{F_n}[\loss(X, \theta')] \right )  \pi(\theta') \, d\theta'}=\frac{\exp \left ( -\sum_{i=1}^n \loss(X_i, \theta)  \right ) \pi(\theta)}{\int_{\Theta} \exp \left ( -\sum_{i=1}^n \loss(X_i, \theta') \right )  \pi(\theta') \, d\theta'},
\end{equation}
where, as before, $\rho:\mathcal{X}\times\Theta\to \mathbb{R}_{\geq 0}$ is a loss function.
We call distributions of the form in \eqref{eq:M-posterior} \emph{M-posteriors} given their intuitive connection to M-estimators of the form in \eqref{eq:M-estimator}. Our notation aims to highlight the fact that these posteriors can be viewed as functionals of the empirical distribution $F_n$, and this notation will also be convenient when we seek to study their robustness properties in what follows.
We will see throughout this paper that the connections between M-posteriors and M-estimators are quite deep and are illuminated by both the asymptotic  and  robustness properties of the M-posterior we study. In particular, the Bernstein-von Mises theorem that we establish is analogous to asymptotic normality of the M-estimator. Furthermore, the sufficient conditions guaranteeing that the M-posterior has a bounded influence function and high breakdown point will be very similar to those required by standard M-estimators. However, our work also demonstrates the role that the choice of prior plays in the robustness properties of M-posteriors and how the interplay between the score and the prior tends to be the characterizing property of the M-posterior.

We note that \cite{minsker:srivastava:lin:dunson2017} used the term M-posterior to refer to a different robustification of the standard posterior based on calculating the median of subset posteriors. M-posteriors of the form \eqref{eq:M-posterior} studied in this work have appeared in the literature under various names including quasi-posteriors \citep{chernozhukov:hong2003}, general belief updates \citep{bissiri:holmes:walker2016}, and generalized posteriors \citep{miller2021}. The robustness properties of special cases of the M-posterior have also been studied in \cite{hooker:vidyashankar2014,ghosh:basu2016,ghosh:majumder:basu2022,matsubara:knoblauch:briol:oates2022,altamirano:briol:knoblauch2023}. Our general framework covers most of these settings that have previously been considered in the literature, and we will discuss connections to previous work more carefully when presenting our main results.

\subsection{Motivating examples }

The following three models will be running examples throughout the paper. They motivated our work and will serve to demonstrate the usefulness of our theoretical findings throughout the paper.

\subsubsection{Huber location posterior.}
\label{subsection::huber-location-posterior}
The Huber loss, introduced by \cite{huber1964}, is a robust alternative to the squared error that interpolates between quadratic and linear penalization of residuals. 
The loss is defined as 
\[
    \loss_c(x) =\
    \begin{cases}
        \tfrac{1}{2}x^2, & |x|\leq c, \\[6pt]
        c |x| - \tfrac{1}{2}c^2, & |x| > c,
    \end{cases}
\]
where the tuning parameter \( c > 0\) controls the threshold at which the function transitions from quadratic to linear. 
In the same fashion, for a given prior \( \pi(\theta) \), we define a \emph{Huber location posterior} as an M-posterior corresponding to the Huber loss \(\loss_c(x)\):
\begin{equation} \label{eqn::huber-location-posterior}
    \pi_n^{\loss_c}(\theta \mid F_n) \propto
    \exp \left (- n \E_{F_n} \bigl [ \loss_c (X - \theta) \bigr ] \right ) \pi(\theta) =
    \exp \Big ( - \sum_{i=1}^n \loss_c ( X_i - \theta )  \Big ) \pi(\theta) .
\end{equation}

\subsubsection{Bayesian quantile regression.}
\label{subsetion::bayesian-quantile-regression}
Quantile regression provides a flexible alternative to mean regression by targeting conditional quantiles of the response distribution rather than its expectation \citep{koenker:bassett1978,koenker2005}. The central idea is to model the conditional $\tau$-quantile of the responses as a linear function of the covariates. For a design matrix \( X^n=(X_{1},\dots,X_{n})^\top \in\R^{n\times d} \) and responses \( Y^n=(y_{1},\dots,y_{n})\in\R^n \), estimating conditional quantiles boils down to finding the slope parameter that solves the M-estimation problem
\begin{equation*}
    \hat\beta_{\tau}=\mbox{argmin}_{\beta\in\mathbb{R}^d}\sum_{i=1}^n\rho_\tau(y_i-X_i^\top\beta),
\end{equation*}
where $\rho_\tau$ is the check loss defined as 
\(
    \loss_\tau(x) = x ( \tau - \mathbf{1}\{ x < 0 \} ),
\)
which penalizes positive and negative values asymmetrically.  We will call 
   \emph{Bayesian quantile regression} the natural M-posterior corresponding to the check loss \( \loss_\tau\), i.e.,
\begin{equation*}
    \pi_n^{\loss_\tau} (\beta \mid X^n, Y^n) \propto
    \exp \left (- n \E_{(X, y) \sim F_n} \bigl [ \loss_\tau (y - X^\top \beta) \bigr ] \right ) \pi(\beta)  =
    \exp \Big ( - \sum_{i=1}^n \loss_\tau (y_i - X_i^\top \beta)  \Big ) \pi(\beta) .
\end{equation*}
This M-posterior corresponds to the asymmetric Laplace likelihood introduced in \citet{yu:moyeed2001}. %
There have been several developments extending this approach, addressing both theoretical and computational challenges. For example, \cite{yang:wang:he2016} studied the posterior inference properties under the asymmetric Laplace model, while \citet{li:he2024} proposed a pseudo-Bayesian approach for sparse quantile regression. 

\subsubsection{Bayesian data reweighting.}
\label{subsection::Bayesian-Data-Reweighting}
Here we introduce the Bayesian data reweighting procedure studied in \cite{wang:kucukelbir:blei2017}. 
Starting with sample \(X^n = (X_1, \ldots, X_n)\), prior \( \pi(\theta)\) and likelihood \( f(\cdot \mid \theta)\), we define the procedure as follows:
\begin{enumerate}
    \item Define a probabilistic model \( \pi(\theta) \prod_{i=1}^n f(X_i \mid \theta)\).
    \item Raise each likelihood to a positive latent weight \( \alpha_i\), where each of the weights \( \alpha_i \) is sampled independently from the prior distribution  \( \pi_\alpha(\alpha)\). 
    We define the joint distribution:
    \[
        \pi(X^n, \theta, \alpha^n) = \frac{1}{Z} \pi(\theta) \prod_{i=1}^n \pi_\alpha (\alpha_i) f(X_i \mid \theta) ^ {\alpha_i},
    \]
    where \( Z \) is the normalizing constant.
    \item Infer the posterior of both the latent variables \(\theta\) and the weights \( \alpha^n\); namely, \( \pi(\theta, \alpha^n \mid X^n). \)
\end{enumerate}

The idea behind this approach is following: if the data point \( X_i\) is unlikely under the likelihood, the value \( \alpha_i\) should downweight the influence of this point on the posterior of \( \theta \mid X^n\). 
In their work, \cite{wang:kucukelbir:blei2017} demonstrated robustness properties of the posterior mean. To be more precise, they show the boundedness of the influence function of the posterior mean of $\theta\mid X^n$ under certain choices of priors on weights ~\( \alpha^n \).

As the main object of interest in this setting, we define the reweighted posterior to be equal to the marginal distribution over the latent variable \( \theta \), given the observed data \( X^n\); that is, we integrate out the weights from the joint posterior from step (3) of the procedure:
\begin{align}
\label{def::alpha-posterior}
 \nonumber   \pi_\alpha(\theta \mid X^n) &:= \int_{\R^n} \pi(\theta, \alpha^n \mid X^n) \, d\alpha^n= \frac{\int_{\R^n} \pi(\theta, \alpha^n, X^n) \, d\alpha^n }{\int_{\Theta} \int_{\R^n} \pi(\theta, \alpha^n, X^n) \, d\alpha^n \, d\theta}\\
 \nonumber &= \frac{\pi(\theta) \exp \left ( \sum_{i=1}^n \log \int_\R \pi_\alpha(\alpha_i) f(X_i \mid \theta)^{\alpha_i} \, d\alpha_i \right ) }{\int_{\Theta} \int_{\R^n} f(\theta, \alpha^n, X^n) \, d\alpha^n \, d\theta}\\
 &  \propto  \exp \left ( - n\mathbb{E}_{F_n}[\loss(X, \theta) ]\right )\pi(\theta),
\end{align}
where  we defined 
\(
    \loss (x, \theta) := -\log \int_\R \pi_\alpha(\alpha) f(x \mid \theta)^\alpha \, d\alpha,
\). We conclude that $\pi_\alpha(\theta \mid X^n)$ is an M-posterior with a loss defined by the likelihood and $\pi_\alpha$. This connection will enable us to complement the work of \cite{wang:kucukelbir:blei2017} by establishing frequentist asymptotic guarantees and deriving robustness properties.

\section{Asymptotic frequentist guarantees}
\label{Sec::Asymptotics}

In this section, we study the asymptotic properties of \emph{Weighted M-posteriors}, a combination of M-posteriors with reweighted posteriors.
Weighted M-posteriors arise from a simple but powerful idea: by allowing each observation to contribute to the overall loss with its own nonnegative weight, we gain a flexible mechanism for addressing a variety of practical and theoretical challenges in modern Bayesian inference. In Section \ref{subsection::Bayesian-Data-Reweighting} we saw that reweighting can be motivated by robustness considerations \cite{wang:kucukelbir:blei2017}. In a frequentist setting, this idea is intuitively connected to that of weighted M-estimators of \cite{field:smith1994,markatou:basu:lindsay1997, markatou:basu:lendsay1998, markatou2000}, the robust filter of \cite{calvet:czellar:ronchetti2015} or the robust  Kalman filter of \cite{duran:altamirano:shestopaloff:sanchez:knoblauch:jones:briol:murphy2024} which applies a so-called \textit{weighted observation likelihood filter}. Weighting schemes are also natural for multilevel data and post-stratification in survey sampling \citep{gelman:hill2007}. They can also be used in the context of severe class imbalance  as often seen in rare-event classification tasks, assigning larger weights to under-represented examples mitigates the tendency of the posterior to be dominated by majority-class losses \citep{rosenblattetal2025}.

\subsection{Framework}

We introduce \emph{Weighted M-posteriors}, which will allow us to state the BvM result in full generality. We then demonstrate how some known results \citep{kleijn:vandervaart2012, avellamedinaetal2021, chernozhukov:hong2003} follow from this statement, along with some new observations. For this, we first need to define the weighted empirical measure.

\begin{definition}[Weighted empirical distribution function]
    Let \( (X_i)_{i=1}^n \) be observations from the statistical model \( \mathcal{F}_n \), and let \( \boldsymbol{\alpha} = (\alpha_i)_{i=1}^n \) be non-negative weights.  The \emph{weighted empirical distribution function} is defined by
    \begin{equation}
        \label{def::weighted-empirical}
        \wempirical(x) \;=\; 
        \frac{1}{n}
        \sum_{i=1}^n \alpha_i \, \mathbf{1}\{X_i \leq x\},
        \quad x \in \mathbb{R}.
    \end{equation}
\end{definition}

Note that the way we defined the weighted empirical distribution function \( \wempirical \), it is not necessarily a true probabilitt distribution function since \( \wempirical(+ \infty) = \tfrac{1}{n} \sum_{i=1}^n \alpha_i\). 
For a given sequence of positive weights \( \boldsymbol{\alpha} \equiv (\alpha_i)_{i = 1}^{\infty}\), we define the weighted M-estimator by \( \hat{\theta}^{\boldsymbol{\alpha}}_{\loss} \) as a solution to
\[
\mathbb{E}_{\wempirical}[\psi(X,\hat{\theta}^{\boldsymbol{\alpha}}_{\loss} )]=\frac{1}{n}\sum_{i=1}^n \alpha_i \psi(X_i, \hat{\theta}^{\boldsymbol{\alpha}}_{\loss} ) = 0.
\]
where \( \psi \) is a score function corresponding to the loss \( \loss \).
We define the weighted M-posterior analogously.
\begin{definition}[Weighted M-posterior]
    Starting from the statistical model \( \mathcal{F}_n \), a prior density \( \pi \) for \( \theta \) over \( \Theta \), and a non-negative sequence of weights \( \boldsymbol{\alpha} \equiv (\alpha_i)_{i = 1}^{\infty}\), the weighted M-posterior is defined as the distribution having density:
    \begin{equation}
        \label{def::weighted-M-posterior}
        \pi^\loss_{n}(\theta \mid \wempirical) \equiv \frac{\exp \left ( -n\E_{\wempirical} \bigl [ \loss(X, \theta) \bigr ] \right ) \pi(\theta)}{\int_{\Theta} \exp \left ( -n\E_{\wempirical} \bigl [ \loss(X, \theta') \bigr ] \right )  \pi(\theta') \, d\theta'}.
    \end{equation}
\end{definition}
Clearly, taking $\alpha_i=1$ for all $i\in\mathbb{N}$, we recover the M-posteriors \eqref{eq:M-posterior}. Keeping general weights $(\alpha_i)_{i=1}^\infty$ and taking a negative log likelihood for the loss \( \loss = -\log f \), the above definition retrieves the standard definition of reweighted likelihood of \citep{wang:kucukelbir:blei2017} in \Cref{subsection::Bayesian-Data-Reweighting} with prior mass points at $(\alpha_i)_{i=1}^n$:
\[
\pi_{n}(\theta \mid \wempirical) \propto \exp \Big ( \sum_{i=1}^n \alpha_i \log f(X_i \mid \theta) \Big ) \pi(\theta) = \pi(\theta) \prod_{i=1}^n f(X_i \mid \theta)^{\alpha_i} .
\]
If $\alpha_i=\alpha$ for all $i\in\mathbb{N}$ we get again the $\alpha$-posterior.

We will study the asymptotic properties of the weighted M-posterior using a condition that is similar to the stochastic LAN assumption, but modified to take both weights and different loss functions into account.  We denote by $P_0$ the distribution of the i.i.d.\ random sample $X_1,\dots,X_n$ and remember that we assumed a well-specified model with true parameter $\theta^*$.
\begin{assumption}[Weighted M-LAN]
\label{assump::WMLAN}
    For any sequence of positive weights with a finite second moment, denoted \( \boldsymbol{\alpha} \equiv (\alpha_i)_{i = 1}^{\infty}\), let \(\overline{\alpha}_n := n^{-1}\sum_{i=1}^n \alpha_i\) be their average. Then there exists a positive definite matrix \( V_{\theta^*} \), such that
    \begin{equation*}
        R_{n, \mathrm{\boldsymbol{\alpha}}}(h) 
        :=
        \sum_{i=1}^n \alpha_i \left( \loss(X_i, \theta^* ) 
        -
        \loss(X_i, \theta^* + h/\sqrt{n}) \right) - h^\top \overline{\alpha}_n V_{\theta^*} \frac{1}{\sqrt{n}} (\mestimator - \theta^*)
        +
        \frac{1}{2} h^\top \overline{\alpha}_n V_{\theta^*} h,
    \end{equation*}
    satisfies
    $\sup_{h \in K} \, \lvert R_{n, \mathrm{\boldsymbol{\alpha}}}(h) \rvert \rightarrow 0$ in $P_0$-probability for any compact set $K \subseteq \mathbb{R}^p$.
\end{assumption}
We show that the weighted M-LAN assumption follows from the same regularity conditions used for the stochastic LAN property in i.i.d.\ models \citep{kleijn:vandervaart2012}. Specifically, assume the per–observation loss is differentiable in probability at the true parameter, is locally Lipschitz on a neighborhood, and the population risk admits a second–order (quadratic) expansion around the true parameter. Under these assumptions, the weighted M-LAN condition holds (see \Cref{lemma::WMLAN_conditions} in \Cref{appendix::supporting-lemmas}).

The following assumption controls the rate of concentration of the Weighted M-posterior around $\theta^*$ and, combined with the weighted M-LAN assumption introduced before, will allow one to derive BvM-type statements.
\begin{assumption}
\label{assump::concentration}
    We say that the weighted M-posterior \( \pi^\loss_{n}(\theta \mid \wempirical) \) defined in \eqref{def::weighted-M-posterior}, concentrates at rate $\sqrt{n}$ around $\theta^*$ if for every sequence of constants $r_n \to \infty$,
    \begin{equation}
    \label{assump:concentration_2}
        \mathbb{E}_{P_0} \left[ \int_\Theta \mathbf{1} \left\{ \|\sqrt{n} (\theta - \theta^*) \| > r_n \right\} \pi^\loss_{n}(\theta \mid \wempirical) d\theta \right] \to 0.
    \end{equation}
\end{assumption}
\subsection{Bernstein-von Mises theorems for weighted M-posteriors}
\begin{theorem}
\label{thm::BvM}
    Let \(\boldsymbol{\alpha} \equiv (\alpha_i)_{i = 1}^{\infty}\) be a sequence of positive (constant) weights with finite second moment. Suppose that the prior density $\pi$ is continuous and positive on a neighborhood around the true parameter $\theta^*$. Letting $d_{\mathrm{TV}}(\cdot, \cdot)$ denote the total variation distance, if Assumptions \ref{assump::WMLAN} and \ref{assump::concentration} hold,
    \begin{equation}
    \label{eq:Thm1}
        d_{\mathrm{TV}} \left( \pi^\loss_{n}(\cdot \mid \wempirical), \phi(\cdot \mid \mestimator, V_{\theta^*}^{-1} / (\overline{\alpha}_n n)) \right) \to 0,
    \end{equation}
    in $P_0$-probability, where $V_{\theta^*}$ is the positive definite matrix satisfying Assumption \ref{assump::WMLAN}.
\end{theorem} 
\begin{remark}
    While the above statement is formulated for a fixed sequence of weights, an analogous result holds when the weights are drawn independently at random. A more detailed discussion of this extension, together with its connections to Bayesian data reweighting (\Cref{subsection::Bayesian-Data-Reweighting}), is given in \Cref{asec::additional-BvM-results}.
\end{remark}

\Cref{thm::BvM} states that the weighted M-posterior behaves asymptotically as a multivariate normal distribution centered at the weighted M-estimator \( \mestimator \).
Furthermore, the result shows that the asymptotic covariance of the weighted M-posterior is given by  \( V_{\theta^*}^{-1} / (\overline{\alpha}_n n) \).
The weights influence the result through their mean; the asymptotic variance is inflated when \( \overline{\alpha}_n < 1 \), and deflated otherwise.
 \Cref{thm::BvM} is related to at least three types of similar results in the literature. 
First, by taking all weights to be equal to one, i.e.\ \( \alpha_n = 1 \) for all \( n\), and taking the loss to be negative log-likelihood, i.e.\ \( \rho = -\log f\), we obtain a standard BvM-type result. While we assume the well-specified case for simplicity, all arguments can be extended if we assume that \( \theta^*\) is the pseudo-true parameter, and hence we retrieve the result of \cite{kleijn:vandervaart2012}.
Second, by again considering the negative log-likelihood, and taking all weights to be equal to some constant, i.e.\ \( \alpha_n = \alpha \) for all \( n\), we derive the BvM-type result for the \( \alpha \)-posteriors of \cite{avellamedinaetal2021}. It is worth noting that by having only one weight parameter, the weight affects the limiting normal distribution only through the variance, and the limiting mean is equal to the standard MLE, unaffected by the choice of parameter \( \alpha \).
Third, by taking all weights to be equal to one, and considering an arbitrary loss function \( \rho \), we retrieve the BvM result of \cite{chernozhukov:hong2003}. Their expansion assumption is very similar to our weighted M-LAN condition (Assumption \ref{assump::WMLAN}).

\subsection{Examples}

\begin{example}[Huber location posterior]
    Consider the location model \(X_i \mid \theta \overset{i.i.d.}{\sim} N(\theta,1)\) and a prior \( \pi(\theta)=N(\mu_0,\sigma_0^2)\).
    Recall the setup of \Cref{subsection::huber-location-posterior} and let \( \loss_c \) be the Huber loss. 
    We proceed by showing that the Huber location posterior defined in \eqref{eqn::huber-location-posterior} concentrates around the true parameter \( \theta^* \). Let \( \psi_c(x):=\loss_c'(x) \) denote the Huber score. 
    By \Cref{thm::BvM}, we know that the M-posterior will concentrate around the M-estimator \( \hat{\theta}_\loss\), which solves the estimating equation
    $\sum_{i=1}^n \psi_c(X_i - \hat{\theta}_\loss) = 0.$
    Let \( \theta^* \) be the true model parameter. We have that
    \(
        \E_{X \sim N(\theta^*, 1)} [\psi_c  ( X-\theta^*  ) ]=0
    \)
    by the symmetry and the oddness of \( \psi_c \),
    so the loss is Fisher consistent at \( \theta^* \). Therefore, the M-posterior \( \pi^{\loss_c}_{n} (\cdot \mid F_n)\) will concentrate around the true model parameter \( \theta^* \).
\end{example}

We now turn our attention to the reweighted posteriors defined in \Cref{subsection::Bayesian-Data-Reweighting}.  We can show that robustifying the normal location model with weights drawn from a Gamma prior, the resulting reweighted posterior still concentrates around the true model parameter (see \Cref{example:gamma_weights} in \Cref{appendix::additional-results}). However, this need not be the case; data reweighting can actually lead to inconsistency. To that end, consider a similar setup to the one from the above example:
\begin{example}[Reweighted posterior: Exponential model]
\label{example::BvM-reweighted-posterior-2}
   Consider the setup of \Cref{subsection::Bayesian-Data-Reweighting} with the model \( X\mid\theta \overset{i.i.d.}{\sim} \mathrm{Exp}(\theta) \), and priors \( \pi(\theta) = N(\mu_0,\sigma_0^2) \) and \( \pi_\alpha(\alpha)= \Gamma(\kappa, \lambda) \).
    Again, a direct calculation (see \Cref{lemma::calculation-reweighted-posterior-exponential}) reveals that
    \begin{equation*}
        \rho(x, \theta) 
        = 
        \kappa\bigl[ \log(\lambda+\theta x-\log\theta)-\log\lambda \bigr],
        \quad
        \text{and}
        \quad
        \psi(x, \theta)
        =
        \frac{\kappa(x-1/\theta)}{\lambda+\theta x-\log\theta}.
    \end{equation*}
    Assume that the data is generated as \( X_i \overset{i.i.d.}{\sim} \mathrm{Exp}(1) \).
    To assess consistency, we evaluate the expectation of the score at \( \theta = 1 \):
    \[
    \begin{split}
        \E_{X \sim \mathrm{Exp}(1)} \bigl [ \psi(X, 1) \bigr ] 
        &=
        \E \left [ \frac{\kappa(X-1)}{\lambda+X} \right ]
        = 
        \kappa \left ( 1 - (\lambda + 1) \E \left [ \frac{1}{\lambda + X} \right ] \right ).
    \end{split}
    \]
    Now, since the function \( x \mapsto 1/(\lambda + x) \) is strictly convex, by Jensen's inequality,
    \[
        \E \left[ \frac{1}{\lambda + X} \right] > \frac{1}{\lambda + \E[X]} = \frac{1}{\lambda + 1}.
    \]
    Hence,
    \[
        \E [\psi(X, 1)] < \kappa \Big( 1 - (\lambda + 1) \cdot \frac{1}{\lambda + 1} \Big) = \kappa(1 - 1) = 0.
    \]
    This implies that the estimating equation has an asymptotic bias, since its expectation under the true model is negative at \( \theta = 1 \).
    In particular, this means the M-estimator \( \hat{\theta}_\rho \) will not converge to the true value \( \theta^* = 1 \).
    As a result, the M-posterior, which concentrates around this biased M-estimator as in \Cref{thm::BvM}, will also fail to concentrate around the true parameter.
\end{example}

\subsection{Bias Correction for M-posteriors}
A standard procedure for removing the asymptotic bias from an M-estimator proceeds by adjusting the estimating equation rather than the estimator itself.  
If \(\E_{P_{\theta^*}}[\psi(X, \theta^*)] =: B \neq 0\), the estimating equation \(\sum_{i=1}^n \psi(X_i, \theta) = 0\) will  have a solution \(\hat{\theta}_\loss\) that is asymptotically biased.  A standard bias–correction idea going back to \cite{huber1964} replaces \(\psi\) with the modified score  
\[
    \psi_{\mathrm{corr}}(x, \theta) := \psi(x, \theta) - B,
\]
so that \(\E_{P_{\theta^*}}[\psi_{\mathrm{corr}}(X, \theta^*)] = 0\). In other words, this correction restores Fisher consistency and ensures that the M–estimator is centered at \(\theta^*\) in the limit.  
We adapt this Fisher consistency adjustment idea for  M-posteriors and hence ensure their concentration around the true model parameter.
We define the bias–corrected loss  
\[
\rho_{\mathrm{corr}}(x, \theta) := \rho(x, \theta) - B\theta,
\]
which has a corresponding estimating equation that is equivalent to using \(\psi_{\mathrm{corr}}\) above, and hence yields an M-estimator \(\hat{\theta}_{n,\mathrm{corr}}\) that is Fisher consistent.  
Since the M-posterior is constructed from the bias–corrected loss, it inherits this property and concentrates at \(\theta^*\), eliminating the systematic shift in the posterior mode observed when using the uncorrected loss.

\begin{example}[(continued) Reweighted posterior: Exponential model] 

We will adopt a similar setup to that of \Cref{example::BvM-reweighted-posterior-2}, but now the goal is to construct a robust de-biased loss for the exponential model. 
 Consider the estimating equation for finding the maximum likelihood estimator of the exponential model:
\[
    \sum_{i=1}^n \hat\theta X_i = n  \Longleftrightarrow  \sum_{i=1}^n (\hat\theta X_i - 1) = 0.
\]

A simple way to make this estimation robust is to apply the Huber score to the summands, thereby changing the estimating equation to
\( 
    \sum_{i=1}^n \psi_c (\hat\theta X_i - 1) = 0.
\)
As shown in the left panel of \Cref{fig::weighted-posterior-not-consistent}, this results in the inconsistency of the corresponding M-estimator since \( \E_{X \sim \text{Exp}(\theta^*)}[\psi_c( \theta^* X - 1 )] \neq 0\).  To fix this, we can define
\(
    \widetilde \psi_c(x) := \psi(x) - B,
\)
where 
\[
    B := \E_{X \sim \text{Exp}(\theta^*)}[\psi_c( \theta^* X - 1 )] = \E_{Y \sim \text{Exp}(1)}[\psi_c( Y - 1 )],
\]
does not depend on the unknown \( \theta^* \).  By integrating the estimating equation from above, we derive that the corresponding loss is equal to
\(
    \rho(x, \theta) = \frac{1}{x} \rho_c (\theta x - 1),
\)
where \( \rho_c \) is the Huber loss.  Accordingly, the bias-corrected loss is equal to
\[
    \rho_{\text{corr}}(x, \theta) = \frac{1}{x} \rho_c (\theta x - 1) - \bar B \theta,
\]
where \( \bar B\) is a Monte Carlo estimate of \( B\).
The results, displayed in Figure~\ref{fig::weighted-posterior-not-consistent}, show that the original M-posterior is sharply concentrated around a mode above the true value, while the bias–corrected M-posterior centers tightly on \(\theta^* = 1\), confirming that the correction restores posterior consistency.

The bias-corrected loss that we constructed can be viewed as a special case of the robust quasilikelihood of \cite{cantoni:ronchetti2001} which was introduced in the more complex setting of generalized linear models and has been successfully used in the construction of robust generalized additive models \citep{alimadad:salibian-barrera2011,croux:gijbels:prosdocimi2012} and high dimensional generalized linear models \citep{avellamedina:ronchetti2018}. The alternative robust loss construction of \cite{bianco:yohai1996,bianco:boente:rodriguez2013} could also be used for M-posteriors  for exponential families.

\begin{figure}[htb]
  \centering
  \includegraphics[width=\textwidth]{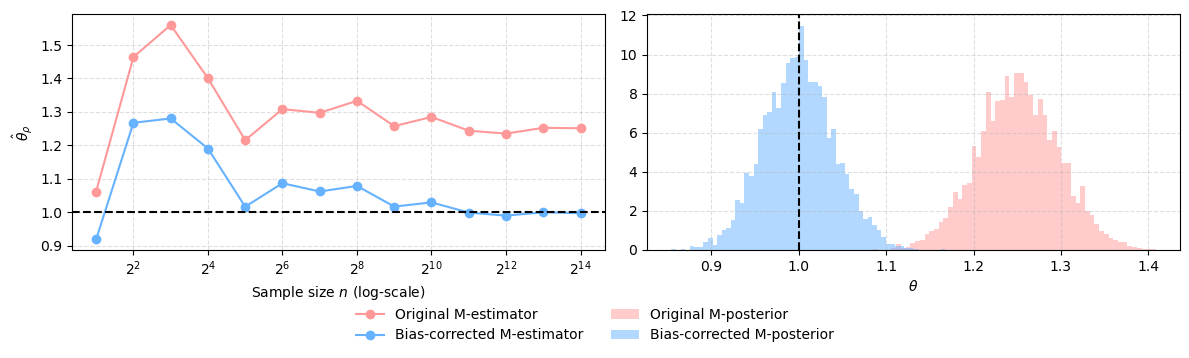}
  \caption{Comparison of original vs.\ bias–corrected M-estimators and M-posteriors. The left panel traces the M–estimator~\(\hat \theta_\loss\) as the sample size \(n\) increases, showing that under the uncorrected loss the estimator converges to a value well above the true rate \(\theta^*=1\), whereas the bias–corrected estimator rapidly stabilizes at the correct value. The right panel displays Metropolis–Hastings draws from the corresponding M-posteriors at \(n=1000\): the original M-posterior is concentrated around the same incorrect mode, while the bias–corrected M-posterior centers on \(\theta=1\). Taken together, these plots demonstrate that removing the asymptotic bias from the estimating equations restores posterior consistency in the Bayesian framework.
    }
  \label{fig::weighted-posterior-not-consistent}
\end{figure}

\end{example}
 
\section{Posterior Influence Function}
\label{sec::PIF}
The posterior influence function describes how sensitive the posterior distribution is to an infinitesimal contamination of the data distribution and is the Bayesian analogue to the classical influence function in robust estimation theory. In this section, we revisit the problem of deriving the influence function for generalized Bayesian posteriors. An early influence function derivation in the context of Bayesian estimators was given in \cite{hooker:vidyashankar2014} where the authors considered posterior mean estimators computed from disparity-based Bayesian posteriors and the first derivation of a \emph{posterior} influence function was obtained by  \cite{ghosh:basu2016} in the context of power divergence posteriors.   Recent extensions of this idea were considered in the work of \cite{matsubara:knoblauch:briol:oates2022,altamirano:briol:knoblauch2023} for Stein-discrepancy based posteriors. These derivations are particular instances of the general form of the posterior influence function for M-posteriors we give in this section. 
\subsection{Uniformly bounded-influence M-posteriors}
Consider the following slight generalization of the M-posterior \eqref{eq:M-posterior},
\[
\pi_n^\rho(\theta \mid G) \propto \pi(\theta)\exp\left(-n \E_{G} \left[\rho\left(X, \theta\right)\right]\right),
\]
where we retrieve the original definition by taking \( G = F_n\). We can now define a pointwise posterior influence function\footnote{We use the accronym PIF for posterior influence function, but we note that it has also been used in the context of robust inference to denote the \emph{power influence function} \cite{hampeletal1986,heritier:ronchetti1994}}, similar to the one introduced in \cite{ghosh:basu2016}.

\begin{definition}[Posterior influence function]
Consider the mixture \( F_{n,\epsilon,x_0} = (1 - \epsilon) F_n + \epsilon \delta_{x_0} \), where $\delta_{x_0}$ is a masspoint at $x_0$ for  \( x_0 \in \mathcal{X}\) and \( \epsilon \in [0, 1]\). The influence function of $\mposterior$ at a point $x_0$, for $\theta\in\Theta$ and the distribution $F_n$ is 
\begin{equation*}
    \PIF(x_0; \theta, \loss, F_n) 
    := \frac{\mathrm{d}}{\mathrm{d} \epsilon} \mposterior (\theta \mid F_{n,\epsilon,x_0}) \Big |_{\epsilon=0}.
\end{equation*}
\end{definition}
The posterior influence function captures the infinitesimal effect of adding a new point $x_0$ to the random sample used to compute $\mposterior$. Note that unlike the standard definition of the influence function (see \cite{hampeletal1986,huber:ronchetti2009}), which is defined as a directional derivative for a population quantity, our influence function only makes sense for finite sample posteriors. While the mixture distribution considered in the definition can be defined for a central population distribution, the posterior distribution is degenerate in the limit. Indeed, when $n$ grows large the posterior contracts around a normal distribution with a shrinking variance as demonstrated by our BvM result. This suggests that the limiting object should be a mass-point at the M-functional $T(F)=\theta^*$, which is not very interesting.

We call an M-posterior \( \mposterior (\cdot \mid F_n) \) \emph{uniformly B-robust} if \( \sup_{\theta \in \Theta} \sup_{x_0 \in \mathcal{X}} |\textrm{PIF}(x_0;\theta,\loss, F_n)| < \infty \). Note that \cite{matsubara:knoblauch:briol:oates2022} refers to uniformly B-robust posteriors as \emph{globally bias-robust} posteriors. We stick to the  B-robust terminology common in robust statistics \citep{hampeletal1986}. Since the posterior influence function depends on $\theta$, we seek uniform boundedness over all $\theta\in \Theta$.

We begin by stating the following technical lemma that provides an upper bound on the pointwise posterior influence function given that the score function is bounded:
\begin{lemma}
\label{lemma::upper-bound-on-PIF}
    Let \( \mposterior (\cdot \mid F_n) \) be an M-posterior corresponding to a loss function \(\loss\), such that the score function \( \score \) is bounded. 
    Let \( B := \sup_{x \in \mathcal{X}} \sup_{\theta \in \Theta} \left | \score(x,\theta) \right | \). 
    We then have the following upper bound on the posterior influence function:
    \[
        |\PIF(x_0; \,\theta, \rho, F_n)|
        \leq 2B n \mposterior(\theta \mid F_n) \left ( | \theta | + \int_\Theta \mposterior(\theta' \mid F_n) |\theta'| \, d\theta' \right ).
    \]
\end{lemma}
While we do not need an exact expression of the posterior influence function to show B-robustness, we state it in the following remark. The derivation can be found in the proof of \Cref{lemma::upper-bound-on-PIF}. 
\begin{remark}
    Let \( \closs ( x, \theta) := \loss(x, \theta) - \loss(x, 0)\) be the re-centered loss and
    \(
        g(x,\theta) := \E_{F_n} [ \closs(X, \theta) ] - \closs(x, \theta).
    \)
    Then the posterior influence function can be written as
    \begin{equation}
    \label{eqn::exact-PIF}
        \PIF(x_0; \,\theta, \rho, F_n) = n \mposterior(\theta \mid F_n) \left ( g(x_0, \theta) - \int_\Theta \mposterior(\theta' \mid F_n) g(x_0, \theta') \, d\theta' \right ).
    \end{equation}
\end{remark}
We are now ready to state the main result of this section, which provides sufficient conditions on the prior and the loss function that guarantee uniform B-robustness of the M-posterior.
\begin{theorem}
\label{thm::global-bias-robust}
    Let \( \mposterior (\cdot \mid F_n) \) be an M-posterior corresponding to a loss function \(\loss\) that is bounded from below and such that the score function \( \score\) is bounded. 
    Furthermore, let \( \pi(\theta) \) be an upper-bounded prior over \( \Theta \) that is possibly improper, and assume one of the following:
    \begin{itemize}
        \item the prior \( \pi(\theta) \) has a finite first moment and is such that \( \sup_{\theta \in \Theta} \pi(\theta) | \theta | < \infty, \)
        \item or the loss function \(\loss\) is convex in \( \theta \) and coercive, meaning \( \lim_{| \theta |\rightarrow \infty} \rho(x, \theta) = \infty \).
    \end{itemize}
    Then the M-posterior \( \mposterior(\cdot \mid F_n) \) is uniformly B-robust.
\end{theorem}
In a nutshell, the above theorem says that a bounded score function ensures a bounded posterior influence function. In other words, infinitesimal perturbations of the data cannot significantly change the posterior distribution at any given point \(\theta\). The two conditions provide a good intuition about how one gets robustness in the Bayesian setting---there is a constant interplay between the loss function and the Bayesian prior. If we consider standard robust losses such as the Huber and check losses, which are both convex and coercive, we do not need to assume much on the prior to guarantee the boundedness of the posterior influence function. Moreover, the prior does not need to be proper, as long as it is upper bounded over its whole domain. On the other hand, without the convexity of the loss function, which is the case, for instance, for redescending losses like the Tukey loss, we require stronger conditions on the prior, mainly to guarantee that the M-posterior itself is well defined.

Our posterior influence function can be used to derive the influence function of functionals of the posteriors. We show how this can be done for posterior moments and quantiles in Section \ref{sec:IF_moments_quantiles}. We note that \cite{gustafson1996,gustafson2000} considered a notion of local sensitivity of posterior moments that resembles the influence function but where the sensitivity is measured with respect to the prior, not to the data. \cite{hooker:vidyashankar2014} introduced a notion of influence functions for posterior mean estimators that is slightly different from ours as they consider fixed contamination neighborhoods. \cite{ghosh:basu2016} and in particular \cite{matsubara:knoblauch:briol:oates2022,altamirano:briol:knoblauch2023} gave sufficient conditions that guarantee the posterior influence function is bounded for their estimators. Our results have the advantage of $(i)$ holding for general M-posteriors, $(ii)$ explicitly connecting the score function $\psi$ to the boundedness of the posterior influence function, as one would intuitively expect given the standard boundedness results for the frequentist M-estimator counterparts, and $(iii)$ highlighting the importance of the prior in the case of non-convex loss or equivalently in the case of  redescending score functions.

While \Cref{thm::global-bias-robust} gives sufficient conditions for obtaining  bounded posterior influence functions, we can also state the converse result. To be more precise, we show that the unboundedness of the score function \( \psi \) leads to non-robust M-posteriors.
\begin{proposition}
\label{prop::unbounded-score-implies-non-robust-posterior}
    Let \( \mposterior (\cdot \mid F_n) \) be an M-posterior corresponding to a loss function \(\loss(x, \theta)\) that is convex in \( \theta \) for every \( x \) and such that the score function \( \score(x, \theta) \) satisfies \(\lim_{x \rightarrow \pm \infty} \psi (x, \theta) = \pm \infty \) for all choices of \(\theta\). Assume \( \pi\) is not degenerate. Then the M-posterior is not uniformly B-robust.  
\end{proposition}
\subsection{Examples}
We begin this section by showing that the standard Gaussian model with a Gaussian prior on the mean parameter does not have a bounded posterior influence function, and so it is not robust in this sense. Our second example shows how working with Huber's loss fixes this issue. Finally, we derive the posterior influence function of the reweighted posteriors, which confirms that these posteriors can indeed be robust to outliers with natural choices of the prior on the weights.

\begin{example}[Gaussian likelihood]
\label{example::influence-function-gaussian}
Consider the Gaussian location likelihood model, i.e.\ let
\(
    \rho(x, \theta) = \tfrac{1}{2}(x-\theta )^2,
\)
for some non-degenerate prior \( \pi(\theta) \).  The corresponding score function \( \score(x, \theta) = \theta - x \) satisfies the assumptions of \Cref{prop::unbounded-score-implies-non-robust-posterior}.  Consequently, the posterior influence function under the Gaussian likelihood is unbounded.

This negative result is very intuitive since in the frequentist setting, using the squared loss leads to the sample mean as the estimator, and the influence function of the mean is unbounded -- one extreme outlier can move the mean by an arbitrarily large amount. The Bayesian analogue with a Gaussian likelihood and squared loss inherits the same problem since the posterior distribution is Gaussian with mean proportional to the sample mean. Thus, both the point estimator in the frequentist case and the full posterior in the Bayesian case fail to control the effect of outliers. 

\end{example}

The next example shows that, as in the frequentist setting, to mitigate the unbounded influence exhibited by the Gaussian likelihood posterior, one can replace the pure quadratic loss with a robust loss, like Huber loss.

\begin{example}[Huber loss]
Consider the M-posterior with Huber loss introduced in \Cref{subsection::huber-location-posterior},
for some \(c>0\), and for an upper-bounded prior \( \pi(\theta) \). 
The corresponding score function is
\[
  \score_c(x,\theta) = \frac{\partial}{\partial\theta}\loss_c(x,\theta)
  =
  \begin{cases}
    x-\theta,  & \quad |x-\theta|\le c,\\
    c\,\mathrm{sign}(x-\theta), & \quad |x-\theta |>c.
  \end{cases}
\]
Hence the score is bounded \(|\score_c(x,\theta)|\leq c\) for all \(x,\theta\). Furthermore, the loss function is convex and coercive; hence, it satisfies the second case of \Cref{thm::global-bias-robust}, which in combination with an upper-bounded prior, shows that the M-posterior \(\pi_n^{\loss_c} (\cdot \mid F_n)\) is uniformly B-robust. Clearly, this conclusion remains true for any convex loss with a bounded derivative.
\end{example}
\begin{example}[Reweighted posterior] 
\label{example::bias-robust-reweighted-posterior}
We continue with examining the reweighted posteriors from \cite{wang:kucukelbir:blei2017} introduced in \Cref{subsection::Bayesian-Data-Reweighting}, showing that this reweighting procedure does indeed robustify the posteriors in the sense of providing a bounded posterior influence function.  To that end, we again consider the setup from \Cref{example::influence-function-gaussian}, which we showed is not robust by default, but this time we also introduce the weights drawn from a Gamma prior. 
More precisely, suppose that \( X_i \mid \theta \sim N(\theta, 1) \) and let the prior on \( \theta \) be \( \pi(\theta) = N(\mu_0,\sigma_0^2) \). Furthermore, let the prior on the weights be \( \pi_\alpha(\alpha) = \Gamma(\kappa, \lambda) \).
A direct calculation (see \Cref{lemma::calculation-reweighted-posterior-normal}) shows that the corresponding loss for this M-posterior is
\[
    \loss(x, \theta) = \kappa \left [ \log \left ( \lambda + \frac{(x-\theta)^2}{2} + \frac{1}{2} \log(2\pi) \right ) - \log \lambda \right ],
\]
and score function
\[
    \psi(x, \theta) = \frac{\kappa(x-\theta)}{\lambda + \frac{(x-\theta )^2}{2} + \frac{1}{2} \log(2\pi)}.
\]
Note that this resulting loss is actually redescending, since \( | \psi | \rightarrow 0 \) as \(|x - \theta | \rightarrow \infty \). Now, we have that \( \loss \geq 0\) and that the score function \( \psi \) is uniformly bounded. Furthermore, the prior satisfies the requirements of the first case of \Cref{thm::global-bias-robust}; hence, we conclude that the reweighted posterior is uniformly B-robust. 
\end{example}

Another way to interpret the result of the previous example is to note that the gamma reweighting of the Gaussian likelihood turns it into a Cauchy-type likelihood tempered by the parameter $\kappa$. In the case $\kappa=1$ the M-posterior behaves exactly like a Cauchy model, which is well known to be robust \citep{clarke1983}.

\subsection{Influence function of posterior moments and quantiles}
\label{sec:IF_moments_quantiles}
We now turn to problem of deriving the influence function of functionals of the posterior distribution. We focus our attention on perhaps the most natural distribution functionals: moments and quantiles. 

\subsubsection{Posterior moments.}

We consider the $k$th-moment posterior functional 
\[
    T_k(F_n) := \int_\Theta \theta^k \mposterior(\theta \mid F_n) \, d\theta.
\]
 We are interested in uniformly bounding the influence function,
\begin{equation*}
    \textrm{IF}(x_0; T_k, F_n) = \frac{\partial}{\partial \epsilon} T_k ( F_{n,\epsilon,x_0}) \Big |_{\epsilon=0},
\end{equation*}
over all \( x_0 \in \mathcal{X}\). To that end, we see that
\begin{equation}
\label{eqn::influence-function-of-the-posterior-moment}
     \textrm{IF}(x_0; T_k, F_n) =  \frac{\partial}{\partial \epsilon} T_k ( F_{n,\epsilon,x_0}) \Big |_{\epsilon=0} =
    \int_\Theta  \frac{\partial}{\partial \epsilon} \theta^k \mposterior(\theta \mid F_{n,\epsilon,x_0}) \Big |_{\epsilon=0} \, d\theta =
    \int_\Theta \theta^k \textrm{PIF}(x_0; \theta,\rho, F_n) \, d\theta.
\end{equation}
Interestingly, this simple calculation reveals that the boundedness of the \( \textrm{IF}(x_0; T_k,F_n) \) does not immediately follow from the boundedness of the \( \textrm{PIF}(x_0; \theta, \rho, F_n) \), not even when taking \( k = 1\), i.e.\ the posterior mean.

It is insightful to contrast the influence functions of the posterior moments with the standard $k$-th moment functionals $\mu_k(\P):=\int_{\mathbb{R}} x^kd\P(x)$. The linearity of these functionals makes it straightforward to compute the influence function
$\textrm{IF}(x_0; \mu_k, F)=x_0^k-\mu_k(F).$
It follows that the standard moment functionals are never robust in the sense of the influence function. This is to be contrasted with the posterior moments which can inherit the robustness of the posterior distribution.

\subsubsection{Posterior quantiles.}

We consider the posterior (left) $\tau$-quantile functional 
\begin{equation}
\label{eqn::definiton-posterior-quantiles}
    T_\tau(F_n)
    :=
    \inf \Bigl \{ \theta : \int_{-\infty}^{\theta} \mposterior (\theta' \mid F_n)\, d\theta' \geq \tau\Bigr\}.
\end{equation}
In order to derive the influence function of $T_\tau(F_n)$ we introduce the functional
\[
    S(\theta, G)
    =
    \int_{-\infty}^{\theta}\mposterior(\theta'\mid G)\,d \theta'
    -
    \tau,
\]
so that \( S \bigl(T_\tau(F_n),\,F_n\bigr) = 0\). This last equation allows us to obtain the desired influence function as we can now invoke the implicit function theorem to get 
\[
    0
    =
    \frac{\partial}{\partial\epsilon}
     S\!\bigl(T_\tau(F_{n,\epsilon,x_0}),F_{n,\epsilon,x_0}\bigr)\Big|_{\epsilon=0}
    \!
    =
    \frac{\partial}{\partial\theta}S(\theta,F_n)\Big|_{\theta=T_\tau(F_n)}
     \frac{\partial}{\partial\epsilon}T_\tau(F_{n,\epsilon,x_0})\Big|_{\epsilon=0} \!
    + \! \frac{\partial}{\partial\epsilon}S\!\bigl(T_\tau(F_n),F_{n,\epsilon,x_0}\bigr)\Big|_{\epsilon=0}.
\]
Since
\[
    \frac{\partial}{\partial \theta} S\bigl(\theta,G\bigr)\Big |_{\theta = T_\tau(G)}
    =
    \mposterior(\theta\mid G)\Big |_{\theta = T_\tau(G)}  =
    \mposterior(T_\tau(G)\mid G),
\]
and
\[
    \frac{\partial}{\partial \epsilon} S\bigl(T_\tau(F_n),F_{n,\epsilon,x_0} \bigr) \Big |_{\epsilon = 0}
    =
    \int_{-\infty}^{\theta}\mposterior(\theta'\mid F_{n,\epsilon,x_0})  \Big |_{\epsilon = 0} \,d \theta'
    =
    \int_{-\infty}^{\theta}
    \PIF\bigl(x_0; \theta', \rho,F_n\bigr)\,d\theta',
\]
we obtain the influence function of the \(\tau\)-quantile,
\begin{equation}
\label{eqn::influence-function-of-posterior-quantile}
    \mathrm{IF}\bigl(x_0;\,T_\tau,F_n\bigr)
    =
    \frac{\partial}{\partial \epsilon} T_\tau(F_{n,\epsilon,x_0}) \Big \lvert_{\epsilon=0}
    =
    -
    \frac{\displaystyle
      \int_{-\infty}^{T_\tau(F_n)}
        \PIF\bigl(x_0;\, \theta', \rho,F_n\bigr)\,d\theta'
    }{
      \displaystyle \mposterior\bigl(T_\tau(F_n)\mid F_n\bigr)
    }.
\end{equation}

Once again, we can see that the boundedness of the posterior influence function is not enough to guarantee the boundedness of the influence function of the posterior quantiles. At the same time, we can see that to achieve the uniform bound on \(\mathrm{IF}\bigl(x_0;\,T_\tau,F_n\bigr)\), we require the integrability of the posterior influence function on \( (-\infty, T_\tau(F_n) )\).

It is again insightful to compare the influence functions of the posterior quantiles with those of standard quantiles $q_\tau(F):=
    \inf  \{ x : F(x)\geq \tau\}$. Assuming that $X\sim F$ has a non-zero density $f$ at $q_\tau(F)$, one can show that 
\begin{equation*}
   \textrm{IF}(x_0; q_\tau, F)=\frac{\tau-\mathbf{1}\{x_0\leq q_\tau(F)\}}{f(q_\tau(F))}. 
\end{equation*}
See \cite[Ch.\ 3.3.1.]{huber:ronchetti2009}. We conclude that the influence function of the standard quantile functional is always bounded provided that there exists a non-zero density at the population quantile. This is in sharp contrast with the posterior quantiles, which can easily be shown to not be robust for suitable non-robust posteriors as we illustrate in the discussion of next subsection. The intuition being that non-robust posterior distributions should not be expected to give robust posterior quantiles.

\subsubsection{Bounded-influence posterior moments and quantiles.}

While the calculations above show that there is no obvious connection between the boundedness of the influence function of the posterior mean and the boundedness of the posterior influence function, for example, the following result states sufficient conditions that guarantee that a bounded posterior influence function implies bounded-influence posterior moments and quantiles.  
\begin{proposition}
\label{prop::connections-to-influence-functions-of-posterior-moments-and-quantiles}
    Let \( \mposterior (\cdot \mid F_n) \) be an M-posterior corresponding to a loss function \(\loss\) that is positive and such that the score function \( \psi\) is bounded. Furthermore, let \( \pi(\theta) \) be a prior over \( \Theta \).
    \begin{enumerate}
        \item For any \(k \geq 1 \), if the prior \( \pi \) has a finite \( (k+1) \)-th moment, then \(k\)-th moment of the posterior  \( \int_\Theta \theta^k \mposterior(\theta \mid F_n) \, d\theta \) has a bounded influence function.
        \item If the prior \( \pi \) has a finite first moment, then the posterior quantiles have a bounded influence function.
    \end{enumerate}
\end{proposition}
The conditions in \Cref{prop::connections-to-influence-functions-of-posterior-moments-and-quantiles} are similar to those in \Cref{thm::global-bias-robust},  but this time requiring slightly stronger  conditions on the prior. Namely, we require \( (k+1) \) finite prior moments to show the boundedness of the influence function of the \(k\)-th posterior moment. At the same time, a finite first moment of the prior guarantees the bounded influence function of all posterior quantiles. 

\subsection{On the robustness of reweighted posteriors}
We revisit the reweighted posterior setting from Example~\ref{example::bias-robust-reweighted-posterior} in more generality. We will rigorously expand the result first mentioned in Theorem 2 in \cite{wang:kucukelbir:blei2017}, which states that the posterior mean of the reweighted posterior exhibits a bounded influence function under appropriate regularity conditions. 
\begin{proposition}
\label{prop::bounded-PIF-of-reweighted-posteriors}
    Let \( X^n = (X_1, \ldots, X_n)\) be an i.i.d.\ sample from the model
   $f(x \mid \theta) = \exp (-g(x, \theta)),$
    where a positive function \(g(x, \theta)\) is such that \( (x, \theta) \mapsto \log  [ g(x, \theta)  ]\) is \(L\)-Lipschitz in \( \theta\) for all \(X\).
    Furthermore, let the prior on the weights \( \pi_\alpha (\alpha)  \) be \( \Gamma(k, \lambda) \) and let \( \pi(\theta) \) be an upper bounded prior over \( \Theta \) with a finite first moment such that \( \sup_{\theta \in \Theta} \pi(\theta) | \theta | < \infty \). Then the reweighted posterior \( \pi_\alpha(\theta \mid F_n) \) defined in \eqref{def::alpha-posterior} is uniformly B-robust.
\end{proposition}

\Cref{prop::bounded-PIF-of-reweighted-posteriors} explains the observed robustness properties of reweighted posteriors introduced in \cite{wang:kucukelbir:blei2017}, but also imposes conditions on the working model. These conditions ensure that the reweighting procedure leads to a bounded posterior influence function.  The following counterexample shows that these conditions are necessary. 
Consider the Gumbel  likelihood model 
\[
f(x \mid \theta) = \exp \bigl ( - \exp \bigl ((x - \theta) ^2 \bigr) \bigr).
\]
This amounts to choosing the  function \(g \):
\(
g(x, \theta) = \exp \bigl ((x - \theta) ^2 \bigr).
\)
The absolute value of the score function of the corresponding reweighted posterior with \( \Gamma(\kappa, \lambda) \) prior on the weights will equal
\[
\left | \psi(x, \theta) \right | = \frac{2\kappa |x- \theta| \exp \bigl ((x - \theta) ^2 \bigr) }{\lambda + \exp \bigl ((x - \theta) ^2 \bigr)},
\]
with
\(
\left | \psi(x,\theta) \right | \rightarrow \infty,
\)
as \(| x - \theta | \rightarrow \infty\). Hence, this reweighted model will not exhibit a bounded score function, and \Cref{prop::unbounded-score-implies-non-robust-posterior} shows that the corresponding M-posterior will not have a bounded influence function.

\bigskip

\section{Posterior Breakdown point}
\label{section::posterior-BP}
In this section we extend the notion of finite sample breakdown point described in Section \ref{sec:robust_statistics} to the  Bayesian framework by introducing a natural definition of posterior breakdown. We will calculate the breakdown point of location M-posteriors defined by convex and non-convex losses, highlighting the importance of the loss and the prior. We connect our posterior breakdown point results to the breakdown point of the posterior mean and posterior quantiles. Contrary to their sample analogues, the posterior mean and quantiles will have a high breakdown point if the posterior breakdown is high, but could also have a breakdown point of $1/n$ if the posterior breakdown point is $1/n$. 
\subsection{Posterior Breakdown Point}
We use the Wasserstein distance on the space of probabilities over \( \Theta \) to define the breakdown point of the posterior distribution of M-posteriors evaluated at a dataset $X^n$.
\begin{definition}[Posterior breakdown point]
    For a given sample $X^n$ and prior distribution \( \pi \), the breakdown point of an M-posterior \( \mposterior(\cdot \mid F_n) \), is defined as
    \[
        \varepsilon^*_{W_2}(\mposterior, X^n) := 
        \min \biggl\{\frac{m}{n} : \sup_{F_{(n,m)} \,\in \, \mathcal{F}_{(n,m)}}
        W_2 \bigl(\mposterior(\cdot\mid F_{(n,m)}), \, \mposterior(\cdot\mid F_n)\bigr) = \infty \biggr\},
    \]
    where
    \( 
        \mathcal{F}_{(n,m)}=
        \{G \in \mathcal{F}_n:\,\sup_{x \in \R }|G (x)- F_n(x)| \leq \frac{m}{n} \}
    \)
    and we write $\mathcal{F}_n$ for the set of all distributions on $\mathcal{X}$ that can arise as empirical distributions of $n$ points in $\mathcal{X}$.  
\end{definition}

Contrasting our definition to the standard breakdown point, we replace the point estimator $T(X^n)$ with the M-posterior distribution $\mposterior(\cdot\mid F_n)$ and measure its stability using the $2$-Wasserstein distance between probability measures on $\Theta$. 
The contamination class $\mathcal{F}_{(n,m)}$ plays the same role as the set of contaminated samples in the classical definition: it contains all empirical distributions that differ from the observed empirical distribution $F_n$ in at most $m$ out of $n$ support points. 
The posterior breakdown point $\varepsilon^*_{W_2}(\pi_n^\rho,X^n)$ is then the smallest contamination fraction $m/n$  such that there exists a contaminated empirical distribution in $\mathcal{F}_{n,m}$ that sends the posterior arbitrarily far (in the $W_2$ sense) from the posterior based on the original data. While the choice of the Wasserstein distance  is somehow arbitrary, it is also a natural metric for probability measures. Furthermore, it allows us to still think about the breakdown as the fraction of data points that makes a distance go to infinity. This would not be the case if we worked with the total variation distance or the Prohorov distance,  which can be at most 1 by construction. Nonetheless, we will see in an example below that working with alternative distances and notions of breakdown point can lead to the same quantitative conclusions.

We proceed by presenting a technical lemma that provides upper and lower bounds for the 2-Wasserstein distance, expressed in terms of the means and variances of the measures. This will allow us to reduce the problem of finding the posterior breakdown point to that of  controlling the first two posterior moments. 
\begin{lemma}
\label{lemma::bounds-on-W2}
    Let $P,Q$ be probability measures on $\mathbb{R}$ with finite second moments. Denote
    \( \mu_P:=\mathbb{E}_P[X] \), \( \mu_Q:=\mathbb{E}_Q[Y] \), and \( \sigma_P^2:=\Var_P(X) \), \( \sigma_Q^2:=\Var_Q(Y) \).
    Then
    \[
    (\mu_P-\mu_Q)^2 \leq W_2^2(P,Q) \le (\mu_P-\mu_Q)^2 + \sigma_P^2 + \sigma_Q^2.
    \]
\end{lemma}
As a preliminary example, we demonstrate that the standard Gaussian posterior exhibits the lowest possible breakdown point of
\( 1/n \). This shows that by changing just one point in the sample, one can send the new posterior arbitrarily far from the original one. This is  analogous to the breakdown point of $1/n$ for the sample mean, which corresponds to the  maximum likelihood estimator of the location parameter for the Gaussian model in the frequentist setting.
\begin{example}
\label{example::BP-gaussian}
    Suppose \( X_1, \ldots, X_n  \overset{i.i.d.}{\sim} N(\theta, 1)\) and let \(  \pi(\theta) = N(0, 1) \).
    Furthermore, let \( \pi_n(\cdot \mid F_n)\) be the standard posterior. From conjugacy, we have
    \[
        \theta \mid X^n \sim N \left (\frac{1}{n+1}\sum_{i = 1}^n X_i, \frac{1}{n+1} \right ).
    \]
    Let \( \P_{(n,1)}\) be the empirical distribution of the contaminated sample \( X^{(n,1)} = (X_1', X_2, \ldots, X_n)\).
    By the lower bound in \Cref{lemma::bounds-on-W2}, we obtain
    \[
        \sup_{F_{(n,1)}} W_2^2 \left (\pi_n(\cdot \mid F_{(n,1)}), \, \pi_n (\cdot \mid F_n ) \right) \geq \sup_{X_1' \in \mathbb{R}} \left ( \frac{X_1' - X_1}{n+1} \right )^2 = \infty.
    \]
    By the definition of the posterior breakdown point, we conclude that \( \varepsilon^*_{W_2} \left (\pi_n, X^n \right ) = \frac{1}{n} \).
     Furthermore, it is easy to see that 
    $$\sup_{F_{(n,1)}} d_{\mathrm{TV}} \left (\pi_n(\cdot \mid F_{(n,1)}), \, \pi_n (\cdot \mid F_n ) \right) \geq 2\sup_{X_1' \in \mathbb{R}}\Phi\left(\frac{(n+1)|X_1'-X_1|}{2 n}\right)-1= 1,$$
    where $\Phi$ denotes the CDF of a standard normal random variable.
    So if we were to define the breakdown point as the number of contaminated points that make the total variation distance equal $1$, we arrive at the same conclusion as with the Wasserstein distance since \( \varepsilon^*_{d_{\mathrm{TV}}} (\pi_n, X^n ) = \frac{1}{n} \).
\end{example}

We proceed to examine the posterior breakdown point in the context of general location  M-posteriors. Our results generalize the ones obtained in \cite{donoho:huber1983} for location M-estimators. We will show that the posterior breakdown point is determined jointly by the selected robust loss function and the prior distribution. Interestingly, our analysis relies on extending the arguments used by \cite{huber1984} in the derivation of the breakdown point of the class of P-estimators or Pitman-type estimators introduced in \cite{johns1979}. While this class of estimators is rather exotic, they are intuitively closely connected to our problem as they can be viewed as M-posterior mean estimators based on uninformative priors.

\subsection{Convex loss for location M-posteriors}
We begin by studying M-posteriors induced by one-dimensional convex loss functions. We will see that, similarly to the frequentist framework, the boundedness of the score function leads to a high breakdown point. 
In the Bayesian setting, however, the prior also plays a crucial role in determining robustness properties.

We first state a technical lemma that generalizes \cite[Lemma 5.1]{huber1984}. Note that Huber considered Pitman-type estimators which in our setting correspond to M-posterior means with uninformative priors \(\pi = 1\).
\begin{lemma}
\label{lemma::posterior-moments-monotone}
   Assume the loss \( \loss \) is symmetric and convex and that the score \( \score \) is bounded. Under these assumptions, odd moments of the M-posterior are monotone increasing in all of its arguments (data points). On the other hand, even moments are decreasing to some point and then increasing in all of its arguments.
\end{lemma}
A useful consequence of the above lemma is the following: the largest bias of the corrupted odd moments of the M-posterior is achieved by taking all of the corrupted sample points equal to \( +\infty\). On the other hand, the largest bias for the even moments of the M-posterior is achieved by some combination of corrupted samples from \( \{ -\infty, +\infty \}\).

We now proceed by stating the result showing how different priors affect the robustness of the M-posterior. We say that  a density function $\pi$ has exponential-like tails if it is of the form \(\pi\propto \exp(-h) \), where \( h \) is convex, symmetric and has a bounded derivative \( h'\). We say that $\pi$ has lighter than exponential tails  if it is of the form \(\pi \propto \exp(-h) \) with a convex and symmetric \( h\), but unbounded derivative \( h'\).
\begin{theorem}
\label{thm::breakdown-point-convex-loss}
    Let $\rho$ be symmetric and convex with a score function $\psi = \rho'$ that is bounded. If the prior \( \pi \)
    \begin{enumerate}
        \item {Is uninformative}, then
        \(
            \varepsilon^*_{W_2}(\pi^{\rho}_n, X^n) = \frac{1}{2}.
        \)
        \item {Has exponential-like tails,} then
        \(
            \varepsilon^*_{W_2}(\pi^{\rho}_n, X^n) \ge \frac{1}{2}, 
        \)
        and
        \(
            \varepsilon^*_{W_2}(\pi^{\rho}_n, X^n) \downarrow \frac{1}{2}
        \)
        as \( n \rightarrow \infty.\)
        \item {Has lighter than exponential tails,} then the breakdown point does not exist, in the sense that no contamination level can drive the M-posterior arbitrarily far in $W_2$-distance.
    \end{enumerate}
\end{theorem}
\begin{remark}
    While the above statement only considers losses in one dimension, it can be extended to loss functions \( \loss : \R^d \to \R \) of the form \( \loss(x) = \widetilde{\loss}(\| x \|) \), where \( \widetilde{\loss} \) satisfies the assumptions of the above theorem. The corresponding multi-dimensional result is stated in \Cref{thm::convex-loss-higher-dim} in the \Cref{app:BP_general_d}.
\end{remark}

Theorem \ref{thm::breakdown-point-convex-loss} highlights the importance of the tails of the prior in determining the breakdown properties of the M-posterior. 
First, it shows that when the M-posteriors are constructed using flat improper priors, $\pi = c > 0$, a bounded score guarantees a breakdown point of $1/2$. Therefore, in this case, the M-posterior has the same breakdown point as its corresponding location M-estimator.  Second, it shows that when the prior has exponential-like tails, the posterior breakdown point is larger or equal to $1/2$, but asymptotically exactly $1/2$. 
Lastly, it shows that when the priors have lighter than exponential tails, the posterior cannot be broken. The interpretation of this seemingly surprising result is that lighter than exponential priors are so strong for robust convex losses that they prevent the posteriors from moving arbitrarily even if all $n$ data points are perturbed arbitrarily. A closer inspection of the proof makes it clear that when $n=m$ the posterior distribution remains lighter than exponential for all $n$, but the posterior mean becomes an increasing function of $n$. Hence the larger the $n$, the more the posterior can be moved in a $W_2$ sense.

In \Cref{fig:different-priors}, we illustrate the results of \Cref{thm::breakdown-point-convex-loss} in an empirical study.  We consider the loss \( \loss(x) = |x|\) with a bounded score function \(\psi(x) = \operatorname{sgn}(x)\). Furthermore, we consider three priors, each one representing one of the three groups of the priors considered in \Cref{thm::breakdown-point-convex-loss}: the flat prior \( \pi=1\) [uninformative], exponential prior [exponential-like tails], and Gaussian prior [lighter than exponential tails]. The blue curves show the M-posteriors fitted on the original non-corrupted sample, while red curves consider the M-posteriors after various levels of corruption.  As suggested by the first case of \Cref{thm::breakdown-point-convex-loss}, the breakdown point under the uninformative prior is equal to \( 1/2\), which can be seen by looking at the first row of \Cref{fig:different-priors} and noticing that the red curve in plot in the first column, with 50\% corruption, begins to move away from the blue curve, and moves farther away as the corruption grows in the second and third column. Furthermore, in the second row, the example shows that the breakdown point under the exponential prior is indeed at least \( 1/2\), where we see that in this example that the breakdown point is strictly greater than \( 1/2\). Lastly, considering the Gaussian prior in the third row, we see that the posterior can't be moved arbitrarily far even by corrupting all data points in the sample.

\begin{figure}[htb]
  \centering
  \includegraphics[width=\textwidth]{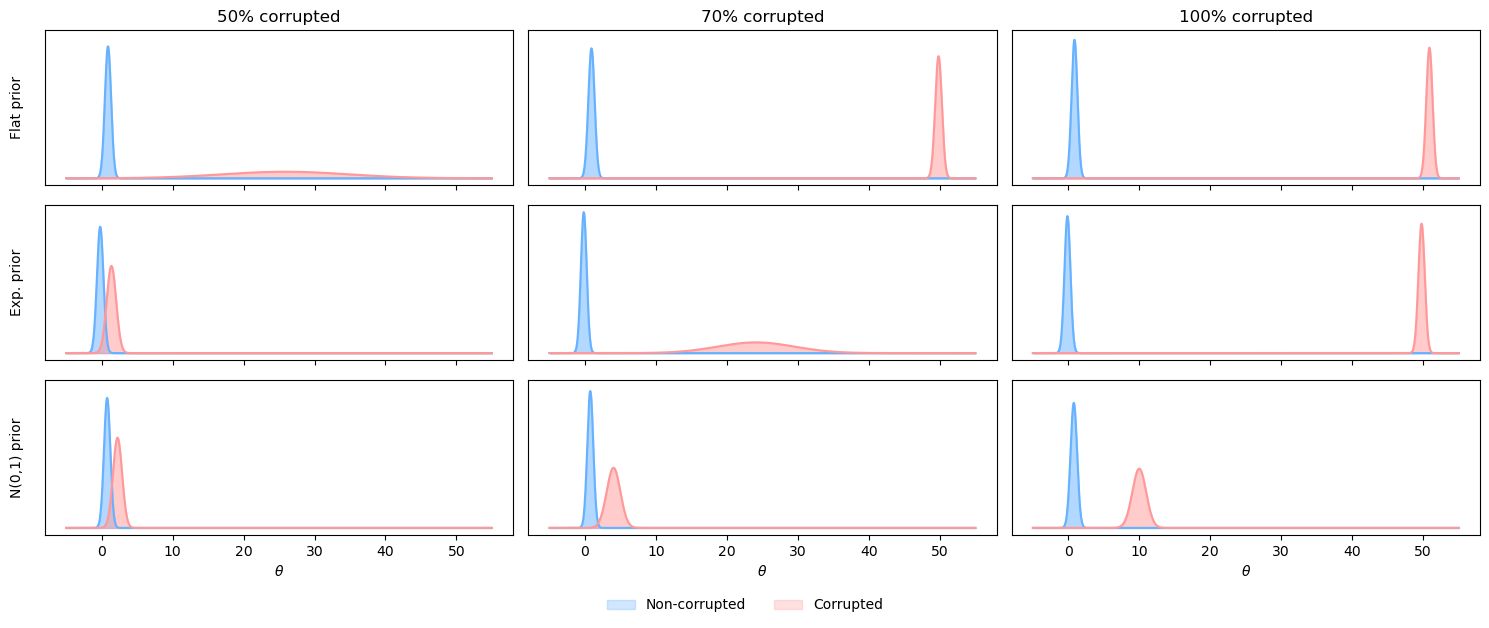}
  \caption{
    Density plots of the M-posterior for the location parameter~\(\theta\) using a Laplace likelihood, under three representative priors (rows: improper, exponential and Gaussian) and three contamination levels (columns: 50\%, 70\%, 100\%).
    The blue-shaded curves show the M-posteriors fitted on the original non-corrupted sample, while red-shaded curves correspond to the M-posteriors after shifting a fraction of observations by 50\% or more.
    This figure illustrates the implications of \Cref{thm::breakdown-point-convex-loss}: the posterior breakdown point  for uninformative priors is $1/2$, for  the exponential prior it can exceed \(1/2\) and for the Gaussian prior it does not exist.
}
  \label{fig:different-priors}
\end{figure}
\subsection{Nonconvex loss for location M-posteriors}
We continue by examining the posterior breakdown point of M-posteriors with redescending score functions.
Redescending M-estimators are characterized by score functions $\psi(x)$ that increase near the origin but eventually decrease toward zero as $|x|$ becomes large, effectively downweighting extreme observations.
This makes them particularly robust to outliers, as observations with very large deviations have diminishing influence on the estimator.
Common examples of redescending M-estimators include Tukey's biweight, the Hampel's loss, the Andrews' sine estimator and Cauchy-type M-estimators \citep{andrewsetal1972,tukeybook,hampeletal1986}.%

\subsubsection{Redescending M-posteriors with unbounded $\rho$.} \label{sec:redescend}
We now formally define the unbounded losses with redescending score functions, as in \cite{huber1984}. %
Essentially, in this section we consider cases where the loss still increases to infinity in the tails, but more slowly that linearly. We work under the following assumptions throughout this section. Assume that loss \( \loss \) is even, \( \loss(0) = 0\), and that \( \loss \) is increasing towards both sides. In addition, assume that
\(
    \lim_{|x| \rightarrow \infty} \loss (x) = \infty,
\)
but
\(
    \lim_{|x| \rightarrow \infty} {\loss (x)} / {| x |} = 0.
\)
Finally, we assume that \( \psi \) is continuous, and that there exists an \( x_0 \) such that \( \psi \) is non-decreasing for \( 0 < x < x_0\), and non-increasing for \( x_0 < x < \infty\). 
For a loss \( \loss \) satisfying these conditions, we say that \( \mposterior(\cdot \mid F_n)\) is a \emph{redescending M-posterior} with unbounded $\rho$.

Under the additional assumption about the finiteness of the first moment,
\cite{huber1984} showed that, using the improper prior \( \pi = 1\), the breakdown point of the posterior mean is equal to \( \frac{1}{2}. \) We extend these results to the M-posteriors. To that end, we first state the following technical lemma that will be used in deriving the main result.
\begin{lemma}
\label{lemma::huber-paper-redescending-loss}
    Assume \( \loss \) satisfies the assumptions given in the first paragraph of~\Cref{sec:redescend}. Let \( m \leq n\) and let \(\corruptedsample\) be a sample where we corrupted at most $m$ points. Define
    \[
        \Delta_{\corruptedsample} (\theta) := \sum_{x \in \corruptedsample} \bigl ( \loss (x - \theta) - \loss (x) \bigr ).
    \]
    Then there is a constant \(C\), which depends on \(X^n\) and on \(m\), but not on the actual corrupted values in \( \corruptedsample\), such that for all \( \theta\) we have
    $(n - 2m) \rho (\theta) - C \leq \Delta_{\corruptedsample} (\theta) \leq n \rho(\theta) + C.$
\end{lemma}

We proceed by stating the result that characterizes the posterior breakdown point of the M-posteriors arising from redescending losses:
\begin{theorem}
\label{thm::breakdown-point-redescending-loss}
    Let \( \loss \) be a loss satisfying the assumptions given in the first paragraph of~\Cref{sec:redescend}. 
    Let \( \pi \) be an arbitrary (potentially improper) prior. If \( \int_\R  \theta^2 \pi(\theta) \exp (- \loss(\theta) )  \, d\theta < \infty \), then the breakdown point of the M-posterior \( \mposterior(\cdot \mid F_n) \) is at least 1/2. Furthermore, if we assume that \( \pi = 1\), then the breakdown point is equal to 1/2.
\end{theorem}

Similar to the convex-loss case studied in the previous section, \Cref{thm::breakdown-point-redescending-loss} emphasizes that the prior controls the differences between the breakdown point in the frequentist and Bayesian setting. Since the prior is data independent, it can only help in making the posterior harder to break, resulting in the breakdown point of the corresponding M-posterior of at least 1/2. On the other hand, by taking an uninformative prior \( \pi = 1\), we retrieve the same result as \cite{huber1984}: the breakdown point of an estimator resulting from the redescending loss is equal to 1/2. 

\subsubsection{M-posteriors with bounded $\rho$.}
In this section, we demonstrate that when discussing the posterior breakdown point, there is an important distinction to be made between losses  that are unbounded and those that are bounded. In fact,
M-estimators with bounded loss functions such as the Tukey loss, the Hampel loss and the Huber-skip loss are more popular that their unbounded counterparts in the robust statistics literature.  However, \cite{huber1984} pointed out that bounded losses do not make sense for P-estimators i.e. M-posterior mean estimators based on uninformative priors. We similarly argue that for M-posteriors, bounded losses such that $|\rho|
\leq C<\infty$ can only lead to well defined posteriors if we use proper priors. 
Indeed, the normalizing constant will not be defined otherwise since 
\begin{equation*}
    \int_{\mathbb{R}} \pi(\theta) e^{-\sum_{i=1}^n\rho(X_i-\theta)}d\theta\geq e^{-nC}\int_{\mathbb{R}} \pi(\theta) d\theta.
\end{equation*}
A similar argument to that given above makes it clear that the M-posterior can only have two finite moments if the prior has two finite moments.

It is also not too hard to see that the breakdown point of M-posteriors with bounded losses does not exist. Indeed, M-posterior moments will be uniformly bounded over all corrupted samples: consider a corrupted sample \( \corruptedsample\), then
\[
    \frac{\int_\R |\theta| \pi(\theta)  \exp \bigl (-\sum_{x \in \corruptedsample} \loss(x - \theta) \bigr ) \, d \theta} {\int_\R \pi(\theta) \exp \bigl (-\sum_{x \in \corruptedsample} \loss(x - \theta) \bigr ) \, d \theta}
    \leq
    \frac{\int_\R |\theta| \pi(\theta)  e ^ {nC} \, d\theta}{\int_\R \pi(\theta) e ^ {-nC} \, d\theta} 
    =
    e ^ {2nC} \int_\R |\theta| \pi(\theta)  \, d\theta
\]
From the above, we can see that the posterior mean cannot be made infinite even if all the data points in the sample are corrupted, and the same conclusion can be reached for the posterior variance. Hence, the $W_2$ distance can never be made infinite and the breakdown point does not exist. This is an undesirable property that prevents the  M-posterior from reporting catastrophic failures and suggests that, in the context of M-posteriors, one should only consider robust unbounded losses that can be used to build a Gibbs measure that integrates to one. This is in contrast to the frequentist setting, where redescending M-estimators with bounded losses can have some optimality properties \citep{hampel:rousseeuw:ronchetti1981} or serve as the building blocks for high-breakdown point estimators in multivariate problems \citep{rousseeuw:yohai1984, yohai1987, davies1987,lopuhaa:rousseeuw1991}.

\subsection{Posterior moments and quantiles} \label{sec::BP-posterior-moment-quantiles}
In the preceding sections, we analyzed the robustness of M-posteriors through their posterior breakdown point. We now shift our focus to the breakdown properties of functionals of these posteriors, specifically the posterior mean and quantiles. Our first result establishes that the posterior mean inherits the robustness of the underlying M-posterior: the breakdown point of the M-posterior mean is bounded below by that of the M-posterior itself.
\begin{proposition}
\label{prop::bp-of-posterior-mean}
    Consider loss \( \loss \) and prior \( \pi \). If the mean of the M-posterior, which we label as \( T_1 \), is finite, then \( \varepsilon^*(T_1, X^n) \geq \varepsilon^*_{W_2}(\pi^\rho_n, X^n).\)
\end{proposition}

It is instructive to compare this result with the breakdown point of the sample mean, consistent with the discussion of posterior moment influence functions in \Cref{sec:IF_moments_quantiles}. Recall the standard mean functional $\mu_1(\P):=\int_{\mathbb{R}} x \, d\P(x)$. The breakdown point of the sample mean equals $1/n$. Consequently, the standard sample mean is not robust in the breakdown-point sense, in contrast to the M-posterior mean, which inherits robustness from the chosen loss.

We continue with examining the breakdown properties of the posterior quantiles. Recall that in \eqref{eqn::definiton-posterior-quantiles} we defined the posterior (left) $\tau$-quantile functional as
\begin{equation*}
    T_\tau(F_n)
    :=
    \inf \Bigl \{ \theta : \int_{-\infty}^{\theta} \mposterior (\theta' \mid F_n)\, d\theta' \geq \tau\Bigr\}.
\end{equation*}
The following technical lemma controls the distance of the distribution quantile to its mean in terms of the variance.

\begin{lemma}
\label{lemma::bound-on-quantiles}
    Let \( Q \) be a distribution with finite variance \( \sigma^2\). Let \( T_\tau\) be its (left) \( \tau\)-quantile and let \( \mu \) denote the mean. Then
    \[
        \bigl | \mu - T_\tau \bigr| \leq \sigma \sqrt{\max\Big \{ \frac{\tau}{1-\tau}, \frac{1 - \tau}{\tau} \Big \}}.
    \]
\end{lemma}

With this in mind, we can characterize the breakdown point of the posterior quantiles:
\begin{proposition}
\label{prop::bp-posterior-quantiles}
    Consider loss \( \loss \) and prior \( \pi \). Suppose that the M-posterior has finite variance. Then, for any \( \tau \in (0,1),\) we have
    \( \varepsilon^*(T_\tau, X^n) \geq \varepsilon^*_{W_2}(\pi^\rho_n, X^n). \)
\end{proposition}

We again compare this result with the breakdown point of the standard empirical quantiles. For the usual empirical $\tau$--quantile, the finite-sample breakdown point equals $\min\{\tau,1-\tau\}$. In contrast, the breakdown point of the M-posterior quantile can be even higher. For instance, taking a Huber location posterior with an improper prior, the posterior breakdown point is equal to \( 1/2\). Hence, by the above result, all posterior quantiles have a breakdown point of at least \( 1/2\).

\subsection{Examples}

We conclude this section with some additional illustrative examples.
\begin{example}[Laplace posterior]
In this example, we consider a Laplace  likelihood model
\(
    f(x\mid \theta) \propto \exp (-|x-\theta|),
\)
which arises from the 
loss
\(
    \loss(x)=|x|
\)
with score function
$\psi(x)= \mathrm{sign}(x)\in[-1,1].$
The Laplace model is intuitively robust since even its maximum likelihood estimator, the empirical median, is very robust. We can formalize this in our notion of posterior breakdown point since the Laplace likelihood is defined by a  convex and symmetric loss with a bounded score function $\psi$. Therefore, the conditions of \Cref{thm::breakdown-point-convex-loss} are met and we conclude that the Laplace posterior exhibits a breakdown point of at least \( \tfrac{1}{2} \).
\end{example}
\begin{example}[Bayesian quantile regression]
\label{example::qr‐regression‐bp}
    Recall the setup of \Cref{subsetion::bayesian-quantile-regression}.
    For a fixed design matrix $X=(X_{1},\dots,X_{n})\in\R^{n\times d}$ and responses $Y=(y_{1},\dots,y_{n})\in\R^n$, we have
    an M-posterior
    \begin{equation*}
        \pi_n^{\loss_\tau} (\beta \mid X^n, Y^n) \propto
        \pi(\theta) \exp \left ( - \sum_{i=1}^n \loss_\tau (y_i - X_i^\top \beta)  \right ),
    \end{equation*}
    where \( \loss_\tau(x)\) is the check loss. Note that the check loss is convex, but not symmetric. Hence, we cannot apply \Cref{thm::breakdown-point-convex-loss} directly. In this setting, an argument similar to the proof of \Cref{lemma::posterior-moments-monotone} still applies, but the maximum bias to the odd moments is now achieved by taking all corrupted values to be either \( +\infty\) or \( -\infty\), depending on whether \( \tau \) is bigger than \( 1/2\). For the uninformative prior \( \pi = 1\) we can follow the same logic as in the proof of \Cref{thm::breakdown-point-convex-loss}. From this we conclude that the M-posterior can be broken if and only if
    \(
        (n-m) \min\{ \tau, 1-\tau \} \leq m \max \{\tau, 1-\tau\}.
    \)
    This results in the breakdown point of the M-posterior of \( \min\{ \tau, 1 - \tau\}.\)
\end{example}
\begin{example}[Reweighted posterior] 
We revisit the setting of \Cref{{subsection::Bayesian-Data-Reweighting}} and \Cref{example::bias-robust-reweighted-posterior} where \( X_i \mid \theta \sim N(\theta, 1) \) and \( \theta \sim \pi(\theta) \). 
Let the prior on weights be \( \pi_\alpha = \Gamma(\kappa, \lambda) \) with \( \kappa > 2\), and let the prior \( \pi(\theta) \) be bounded. 
Then the reweighted posterior has a breakdown point greater than \( \frac{1}{2}. \)
As before, a short computation reveals that the reweighted posterior is actually an M-posterior with loss function
\[
    \rho (x) = \kappa \left [ \log \left ( \lambda + \frac{x ^ 2}{2} + \frac{1}{2} \log(2\pi) \right ) - \log \lambda \right ],
\]
and
\[
    \psi(x) = \frac{\kappa x}{\lambda + \frac{x^2}{2} + \frac{1}{2} \log(2\pi)}.
\]
Now, note that \( \rho \) is symmetric, can trivially be rescaled to \( \rho(0) = 0\) and \( \rho \) is increasing towards both sides. 
Furthermore, we have that \( \lim_{|x| \rightarrow \infty} \rho (x) = \infty\) and \( \lim_{|x| \rightarrow \infty} \rho (x) / |x| = 0. \) 
Also, \( \psi \) is continuous and writing \( \lambda' = \lambda + \tfrac{1}{2} \log(2\pi)\), we see that
\[
    \psi'(x) = \frac{\kappa}{\lambda' + \frac{x^2}{2}} - \frac{\kappa x^2}{(\lambda' + \frac{x^2}{2}) ^ 2} = \frac{\kappa}{\lambda' + \frac{x^2}{2}} \left ( 1 - \frac{x^2}{\lambda' + \frac{x^2}{2}}\right ).
\]
It follows that from the origin, \( \psi \) is first non-decreasing as \(x \) grows and then non-increasing. 
Hence the reweighted posterior is a redescending M-posterior. 
To apply \Cref{thm::breakdown-point-redescending-loss}, it remains to check the finite-moment condition:
\[
    \int \pi(\theta)\,e^{-\rho(\theta)}\,\theta^2\,d\theta
    = \int \pi(\theta) \lambda^\kappa \bigl(\lambda + \tfrac{\theta^2}{2} + \tfrac{1}{2} \log(2\pi)\bigr)^{-\kappa}\theta^2\,d\theta.
\]
Since for large $|\theta|$, we have $(\lambda + \tfrac{\theta^2}{2} + \frac{1}{2} \log(2\pi))^{-\kappa}\theta^2 = O(|\theta|^{2-2\kappa})$, it follows that when $\kappa>2$ and \( \pi(\theta)\) is upper-bounded, the integral is finite.  
Thus the second moment is finite.
From \Cref{thm::breakdown-point-redescending-loss}, we conclude that the posterior breakdown point of this reweighted posterior is at least \( 1/2\).
\end{example}
\section{Numerical Examples}
\label{sec::numerical-examples}

In this section, we present three experiments designed to complement and illustrate the theoretical results developed in the previous sections. 
For reproducibility, the complete code is openly available at \href{https://github.com/JurajMarusic/M-posteriors}{\texttt{https://github.com/JurajMarusic/M-posteriors}}.

\subsection{Normal Location Model}

As a first example, we illustrate the differences in the posterior influence function between the standard posterior and the M-posterior induced by Huber's loss in a simple setting. For this, we fit a normal location model \(\P_\theta = N(\theta, 1)\) to a dataset \(X^n\).
A similar example has been studied in the context of robust KSD-Bayes in \cite{matsubara:knoblauch:briol:oates2022}.
To that end, recall the definition of the posterior influence function
$\textrm{PIF}(x_0; \theta,\rho, F_n) := \frac{\mathrm{d}}{\mathrm{d} \epsilon} \pi_n^\rho (\theta \mid F_{n,\epsilon,x_0}) \Big |_{\epsilon=0}.$

In \Cref{fig::normal-location-model}, we illustrate the behavior of the mapping 
\(
(x_0,\theta) \mapsto \bigl|\mathrm{PIF}(x_0;\,\theta,\rho,F_n)\bigr|,
\)
under different choices of loss functions.
The left panel shows the function as a curve in \(x_0\) for a fixed value of \(\theta\), whereas the right panel reverses the perspective by fixing a contamination point \(x_0\) and examining how the PIF varies across the parameter space in \(\theta\).

 We observe that when the posterior is constructed using  \emph{non-robust} losses, such as the quadratic (Gaussian) loss, the magnitude of the PIF grows without bound as \(x_0\) moves farther away from the bulk of the data. This is consistent with the theoretical result in \Cref{thm::global-bias-robust} and is analogous to the classical non-robustness of least squares and Gaussian likelihood--based inference.
By contrast, when we replace the quadratic loss with robust alternatives---such as Huber’s loss---the influence of extreme contamination is effectively capped.
In this case, the PIF remains bounded even as \(x_0 \to \infty\).
This boundedness is precisely what \Cref{thm::global-bias-robust} guarantees: robust M-posteriors limit the global bias introduced by a single adversarial contamination.
\begin{figure}[htb]
  \centering
  \includegraphics[width=\textwidth]{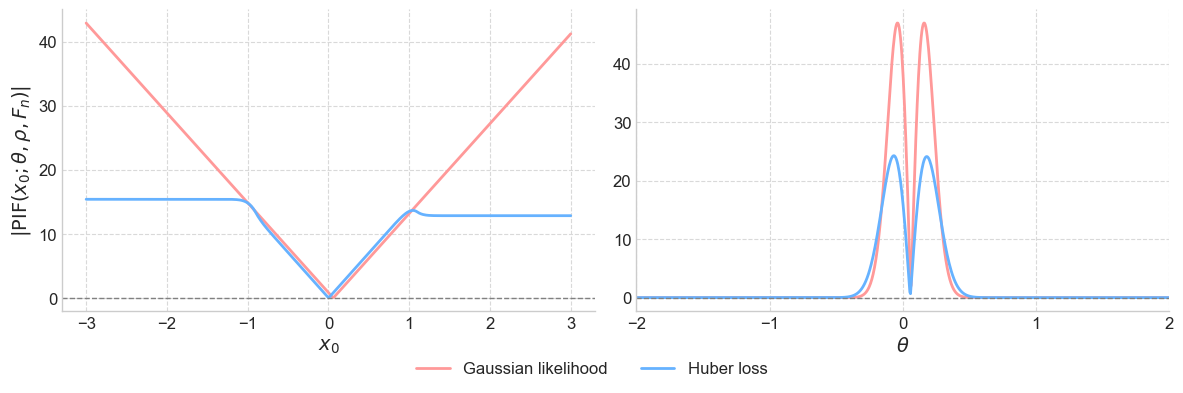}
  \caption{
    Comparison of PIF between standard Gaussian likelihood (red) and Huber loss (blue), computed on a sample of size $n=100$ with Huber threshold $c=1$. 
    \textbf{(Left)} PIF as a function of contamination point $x_0$, holding $\theta=0.1$ fixed; 
    \textbf{(Right)} PIF as a function of parameter $\theta$, holding $x_0=2.0$ fixed. 
    }

  \label{fig::normal-location-model}
\end{figure}

\subsection{Cluster Selection in a Mixture model}
We investigate the Dirichlet process mixture model (DPMM) \citep{murphy::mlbook} for clustering and density estimation, following a setup similar to \cite[Section 3.5]{wang:kucukelbir:blei2017}.
Specifically, we generated a two-dimensional dataset of \(N=2000\) observations from three skewed clusters with proportions \(\pi=(0.3,0.3,0.4)\), where each cluster was sampled from a skew-normal distribution with distinct shape, location, and scale parameters. 

The component means are \(\mu_1=(-2,-2)\), \(\mu_2=(3,0)\), \(\mu_3=(-5,7)\). 
The scale matrices are \(\Omega_1=\mathrm{diag}(4, 4)\), \(\Omega_2=\mathrm{diag}(4,16)\), \(\Omega_3=\mathrm{diag}(16,4)\). 
Lastly, the shape vectors are \(\alpha_1=(-5,0)^\top\), \(\alpha_2=(10,0)^\top\), \(\alpha_3=(15,0)^\top\). Each component has density \(f_j(x)=2 \phi(x;\mu_j,\Omega_j)\,\Phi (\alpha_j^\top \Omega_j^{-1/2}(x-\mu_j))\) \citep{skew-normal}, where \( \phi \) denotes the PDF of a standard normal random variable. Finally, the mixture density is \(f(x)=\sum_{j=1}^3 \pi_j f_j(x)\).

We then fit a Bayesian Gaussian mixture model with a Dirichlet process prior, allowing up to 30 diagonal components, using both the standard Gaussian likelihood and a robust Huber loss.
Posterior mixing proportions \(\{\omega_k\}\) were estimated and any components with \(\omega_k \leq  0.1\) were dissolved in order to filter out negligible ghost components. 

For a component with mean $\mu_k$ and diagonal covariance $\Sigma_k = \mathrm{diag}(\sigma_{k1}^2,\dots,\sigma_{kd}^2)$, the usual contribution in the Bayesian Gaussian mixture model is equal to
\begin{equation}    
\label{eq::numerical-examples-quadratic-part}
    -\tfrac{1}{2} \sum_{j=1}^d \Big( \frac{x_j - \mu_{kj}}{\sigma_{kj}} \Big)^2.
\end{equation}
We replace this squared residual by the Huber loss applied coordinate-wise to the standardized residuals 
$r_{kj} = (x_j - \mu_{kj})/\sigma_{kj}$. 
Thus, the quadratic part \eqref{eq::numerical-examples-quadratic-part} is replaced by 
\(
    - \sum_{j=1}^d \rho_c(r_{kj}).
\)
The results of both clustering methods are shown in \Cref{fig::mixture-model}.
The first panel on the left shows the true cluster assignments.
The middle panel shows the clustering assignments of the robust M-posterior induced by the Huber loss, as described above. 
We can see that this robust posterior recovers the true number of clusters, along with their respective mean locations. 
On the other hand, as seen in the right-most panel, the standard model using Gaussian likelihood incorrectly finds 5 clusters.

The findings of this experiment are consistent with the results in \cite{wang:kucukelbir:blei2017}, where they use the reweighting procedure to downweight the influence of outliers i.e. the data points that do not seem to match the Gaussianity assumption of the model.

\begin{figure}[htb]
  \centering
  \includegraphics[width=\textwidth]{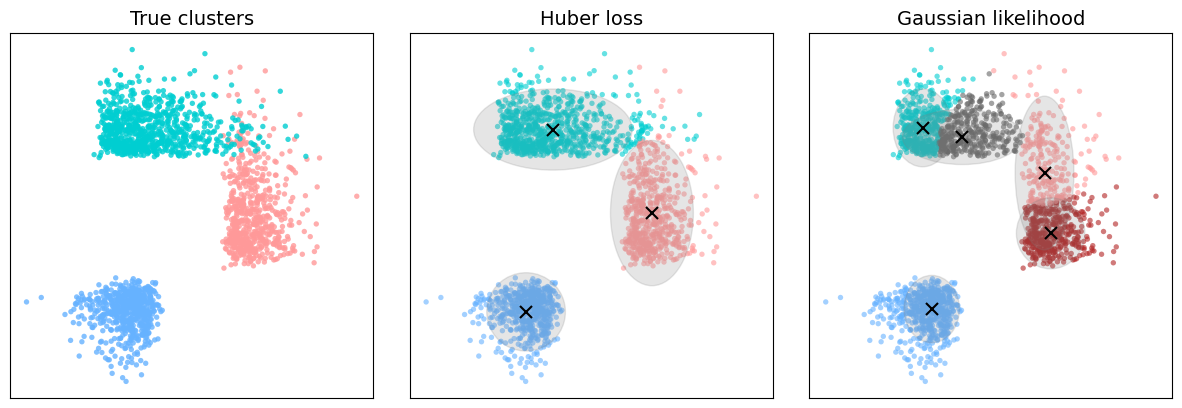}
    \caption{Side‐by‐side comparison of Dirichlet‐process Gaussian mixture fits on the same skewed, three‐cluster dataset. \textbf{(Left)} Data colored by true cluster labels. \textbf{(Center)} Posterior under a robust Huber loss ($c=1.0$), with active components outlined by shaded $2\sigma$ ellipsoids and their centers marked by “$\times$.” \textbf{(Right)} Posterior under the standard Gaussian likelihood, using the same visual conventions.}
  \label{fig::mixture-model}

\end{figure}

\subsection{Poisson Factorization for Recommendation Systems}
The MovieLens 1M dataset is a widely used benchmark in recommender‐systems research, containing one million ratings of 3,952 movies made by 6,040 users. Each rating is an integer from 1 (worst) to 5 (best).
In real‐world recommender‐system deployments, not all observed interactions faithfully reflect a single user’s tastes.
For example, friends or family often share streaming accounts, causing ratings and watch histories to mix multiple people’s preferences.  Similarly, “household” profiles on video platforms aggregate disparate viewing habits.
Such account sharing injects spurious signals—two users’ contrasting movie tastes end up conflated—which can mislead a pure collaborative‐filtering model into learning noisy or even contradictory latent factors. 
By deliberately corrupting a fraction of our users’ data—replacing their original movies and ratings with random values—we mimic this real‐world noise.

We model users and items in a Poisson factorization framework.
Let $U$ be the number of users, $M$ the number of movies, and $K$ the dimensionality of the latent factor space.
For each user $u\in\{1,\dots,U\}$ we introduce a nonnegative factor vector $\theta_u\in\R^K_{\geq 0}$ and for each item $m\in\{1,\dots,M\}$ a nonnegative factor vector $\beta_i\in\R^K_{\geq 0}$.
We place independent exponential priors over these latent variables
\(
    \theta_{u,k} \sim \textrm {Exp}(\lambda), 
    \,
    \beta_{m,k} \sim \textrm {Exp}(\lambda),
\)
for $k=1,\dots,K$.
Given these latent factors, the observed binary outcome $y_{u,m}$ is assumed to follow a Poisson distribution
$y_{u,m} \mid \theta_u, \beta_m \sim \textrm{Poisson}(\theta_u^\top\beta_m).$
Thus the model captures each user–item interaction by the inner product of their latent representations, while the exponential priors regularize the factor magnitudes.
Posterior inference proceeds by approximating or sampling from
\[
    \pi \Bigl (\{\theta_u\}_{u = 1}^U,\{\beta_m\}_{m=1}^M \mid \{y_{u, m}\}_{u=1, m=1}^{U, M} \Bigr) \propto
    \Bigl[\prod_{u = 1}^U e^{-\lambda \| \theta_{u} \|_1}\Bigr]
    \Bigl[\prod_{m = 1}^M e^{-\lambda \| \beta_{m} \|_1}\Bigr]
    \Bigl[\prod_{u = 1, m = 1}^{U, M} \frac{(\theta_u^\top \beta_m )^{y_{u,i}}}{y_{u,m}!}e^{-\theta_u^\top \beta_m }\Bigr].
\]

We continue by studying the reweighted posteriors, mimicking the example in \cite[Section 4]{wang:kucukelbir:blei2017}.
To that end, we introduce for each user \(u=1,\dots,U\) a latent weight \(\alpha_{u}\), drawn i.i.d. from a Gamma prior:
\(
    \pi_\alpha(\alpha_{u} ) = \Gamma(a, b).
\)
We then temper the likelihood of user \(u\) by raising it to the power \(\alpha_{u}\), as introduced in \Cref{subsection::Bayesian-Data-Reweighting}.  
Accordingly, the full joint posterior is
\[
\pi(\{\theta_u\},\{\beta_m\},\{\alpha_u\}\mid \{y_{u,m}\})
\;\propto\;
\prod_{u=1}^U e^{-\lambda\|\theta_u\|_1}\,\pi_\alpha(\alpha_u)
\prod_{m=1}^M e^{-\lambda\|\beta_m\|_1}
\prod_{u=1}^U\prod_{i=1}^M 
   \Bigl(\tfrac{(\theta_u^\top\beta_i)^{y_{u,i}}}{y_{u,i}!}\,
   e^{-\theta_u^\top \beta_i}\Bigr)^{\alpha_u}.
\]

For each user \(u\), let
\(
    A_u = \sum_{m=1}^M \log \pi (y_{u,m}\mid \theta_u,\beta_m).
\)
Then, since
\(
    \pi_\alpha(\alpha)
    =\tfrac{b^a}{\Gamma(a)}\alpha^{a-1} e^{-b\alpha},
\)
we have, by similar calculations as in \Cref{prop::bounded-PIF-of-reweighted-posteriors}, that
\begin{equation}
\label{eqn::poisson-factorization-form-of-m-posterior}
    \mposterior\Bigl (\{\theta_u\}_{u = 1}^U,\{\beta_m\}_{m=1}^M \mid \{y_{u, m}\}_{u=1, m=1}^{U, M} \Bigr) \propto
    \Bigl[\prod_{u = 1}^U e^{-\lambda \| \theta_{u} \|_1}\Bigr]
    \Bigl[\prod_{m = 1}^M e^{-\lambda \| \beta_{m} \|_1}\Bigr]
    \Bigl[ \prod_{u=1}^U (b - A_u)^{-a} \Bigr].
\end{equation}

Throughout the experiment, motivated by the choices of hyperparameters in \cite{wang:kucukelbir:blei2017}, we used \( \lambda = 10\), \(K = 10\), \( a = 1000\) and \(b= 3000\). Furthermore, we used the automatic differentiation variational inference (ADVI) \citep{ADVI} to perform the inference on the latent factors. 

We consider three corruption regimes—none, \(5\%\), and \(10\%\). For each model, we report the negative out-of-sample log-likelihood (NLL).
\begin{table}[H]
\centering
\caption{Negative log-likelihoods by corruption level and model type (lower is better).}
\label{tab:neg_ll}
\begin{tabular}{llll}
\toprule
{} & \multicolumn{3}{l}{Negative LL} \\
Model Type & M-posterior & Reweighted & Standard \\
Corruption Level & & & \\
\midrule
0.00 & 1.689 & 1.690 & 1.724 \\
0.05 & 1.727 & 1.725 & 1.739 \\
0.10 & 1.748 & 1.746 & 1.758 \\
\bottomrule
\end{tabular}
\end{table}

\Cref{tab:neg_ll} summarizes our results. The key observation is that the M-posteriors closely match the performance of the reweighted posterior while avoiding inference over latent weights, empirically confirming the calculation in \eqref{eqn::poisson-factorization-form-of-m-posterior} in a variational setting. Concretely, the reweighted approach requires inferring $U$ (users) $+$ $U$ (weights) $+$ $M$ (movies) $= 2U+M$ latent variables, whereas the M-posterior uses only $U+M$. Notably, the robust methods outperform the standard model even with no explicit corruption, suggesting mild contamination in the original data—a pattern also reported by \citet{wang:kucukelbir:blei2017}. That said, the gains are modest: the contamination magnitude is inherently capped because ratings lie in $\{1,\dots,5\}$, unlike settings where outliers can take arbitrarily extreme values.

We want to point out that the same experiment that was conducted in \cite{wang:kucukelbir:blei2017} used a Beta prior on the weights, instead of the Gamma which we use here. The reason for this is that, by using a Beta prior on the weights, there is no closed formula for the M-posterior in \eqref{eqn::poisson-factorization-form-of-m-posterior}.

\section{Discussion}
According to \cite{berger1994}, \emph{``Robust Bayesian analysis is the study of the sensitivity of Bayesian answers to uncertain inputs. The uncertain inputs are typically the model, prior distribution, or utility function, or some combination thereof''}, where by utility function, we can think of decision-theoretic loss function that yields a posterior functional of interest e.g.\ the posterior mean or quantiles.  In this work, we have emphasized what Huber described in the last chapter of \cite[page 327]{huber:ronchetti2009} as the prophylactic approach to robustness suggested by the Bayesian philosophy, \emph{``Make sure that uncertain parts of the evidence never have overriding influence on the final conclusions''}.  For our contribution, we adapted two classical quantitative measures of robustness, which gauge the influence that outlying evidence can have on final conclusions, for use in studying posterior distributions.  
Our results shed light into Berger's description of Bayesian robustness since they $(i)$ formalize the robust statistics intuition that the key input controlling the effect of outliers is the model or, more generally in our M-posterior framework, the M-estimation loss via the score function; $(ii)$ ellucitate the role of the prior for robustness,  which to some degree matches Huber's desiderata \emph{``robustness should prevent an uncertain prior from overwhelming the observational evidence''}; and, $(iii)$ perhaps surprisingly, also indicate that natural choices of utility functions do not play an important role for robustness. Indeed, the latter point downplays Huber's recommendation  that \emph{``the posterior distribution should be evaluated through utility functions that do not involve its extreme tails, for example in the one-dimensional case through a few selected quantiles, rather than through posterior expectations and variances''} \cite[page 329]{huber:ronchetti2009}. 
In Sections \ref{sec:IF_moments_quantiles} and \ref{sec::BP-posterior-moment-quantiles}, we saw that the robustness properties of posterior moments and quantiles stem directly from the robustness, or lack thereof, of the posterior distribution. This is in stark contrast to their standard frequentist counterparts -- sample moments and quantiles -- which behave very differently from a robustness standpoint. 

Much of the formal reasoning in the last chapter of \cite{huber:ronchetti2009} rests on linking the posterior mode and other posterior functionals to the classical maximum likelihood estimator (M-estimator) through the Bernstein–von Mises theorem, the same way we motivated the robustness properties of M-posteriors by first considering their concentration properties. This connection effectively washes out the influence of the prior, leading Huber's recommendations on Bayesian robust modeling to closely mirror standard M-estimator theory. 
We conceptually lean on this intuition further by recognizing that the same concentration idea also indicates that it makes little sense to study the robustness of these posteriors on the population level. Indeed, taking the number of samples to infinity, since in that case the posterior collapses to a point mass, defeats the whole purpose of Bayesian philosophy.  For this reason, we believe the focus should be on the finite sample analysis of the robustness properties of these posteriors.  Perhaps even more importantly, the posterior influence function and posterior breakdown point that we consider are model and paradigm agnostic; they make sense as mathematical quantifiers of the influence of small fractions of the data irrespective of the existence of a data generating process. 

Our work was greatly influenced by recent developments in Bayesian statistics that were published after Huber's chapter on Bayesian robustness. In particular, the work on generalized posteriors \citep{hooker:vidyashankar2014, bissiri:holmes:walker2016, ghosh:basu2016, miller2021,matsubara:knoblauch:briol:oates2022} relaxes the orthodox Bayesian update rule and naturally motivates a formal connection with the classical M-estimation theory.  We hope this work can serve as the basis for a more ambitious robustness study of modern Bayesian methods that are not covered by the theoretical framework considered in this paper. Probably some of the most natural and important methods to be studied in future work include high dimensional models, hierarchical latent variable models, and variational methods.
%


\appendix
\renewcommand{\thesection}{\Alph{section}}

\section{Proofs of main results} \label{appendix::proofs-of-main-results}
\subsection[Proofs of Sec. Asymptotics]{Proofs of \Cref{Sec::Asymptotics}}
\noindent \textbf{Proof of \Cref{thm::BvM}}

\begin{proof}
    This proof closely follows the proof of Theorem 1 in \cite{avellamedinaetal2021}. Our proof is adapted to the multiple $\mathbf{\alpha}$ framework and an arbitrary loss function \( \loss \). In addition to that, we use the Weighted M-LAN assumption instead of the (usual) LAN assumption on the model.
    Since $\theta^*$ belongs to the interior of $\Theta$, there exists $\delta > 0$ such that the open ball $B_{\theta^*}(\delta)$ around $\theta^*$ belongs to $\Theta$. Furthermore, we can choose $\delta$ such that $\pi$ is continuous and positive on $B_{\theta^*}(\delta)$.
    Next, for any compact set $K_0 \subset \mathbb{R}^p$, we have that $\theta^* + \frac{h}{\sqrt{n}} \in B_{\theta^*}(\delta)$ whenever $n \geq N_0$ for some $N_0$ depending on the set $K_0$ and parameter $\delta$.
    For vectors $g, h \in K_0$, the following random variable will be used throughout the proof:
    \begin{equation}
    \label{eq::BvM-def-of-fn}
        f_n(g,h) \equiv \left \{ 1 - \frac{\phi_n(h)}{\pi^{\mathrm{\mathrm{WMLAN}}}_{n}(h \mid \wempirical)} \frac{\pi^{\mathrm{\mathrm{WMLAN}}}_{n}(g \mid \wempirical)}{\phi_n(g)}\right \} ^ +,
    \end{equation}
    where $\phi_n(h) \equiv n ^ {- \frac{1}{2}} \phi (h \mid \Delta^{\boldsymbol{\alpha}}_{n, \theta*}, V_{\theta^*}^{-1} / \overline{\alpha}_n)$ 
    and
    $\pi^{\mathrm{WMLAN}}_{n}(h \mid \wempirical) \equiv n ^ {- \frac{1}{2}} \pi_{n}(\theta ^ * + \frac{h}{\sqrt{n}} \mid \wempirical)$,
    are scaled versions of desired densities. Furthermore, define 
    $\pi_n \equiv n ^ {- \frac{1}{2}} \pi (\theta ^ * + \frac{h}{\sqrt{n}})$,
    the density of the prior distribution of the transformation $\sqrt{n}(\theta - \theta ^ *)$.
    With this in mind, note that $f_n$ from above is well defined on $K_0 \times K_0$ for all $n > N_0$.
    Let $\overline{B}_{\textbf{0}}(r_n)$ denote a closed ball of radius $r_n$ around \textbf{0}. Since $d_{\mathrm{TV}} \leq 1$, for any sequence $r_n$ and $\eta > 0$, denoting 
    \[
    A_n = \Big \{ \underset{g,h \in \overline{B}_{\textbf{0}}(r_n) }{\sup} f_n(g, h) \leq \eta \Big \},
    \]
    we have
    \begin{equation}
    \label{eqn::BvM-proof-upper-bound-on-TV}
        \E_{P_0} \left [ d_{\mathrm{TV}} (\pi^{\mathrm{WMLAN}}_{n}(\cdot \mid \wempirical), \phi_n(\cdot) ) \right ]
        \leq \E_{P_0} [ d_{\mathrm{TV}} (\pi^{\mathrm{WMLAN}}_{n}(\cdot \mid \wempirical), \phi_n(\cdot) ) \mathbf{1} \{ A_n \} ] 
        + \P_{P_0} (A_n^c)
    \end{equation}
    Following the exactly same arguments as in \cite{avellamedinaetal2021}, we conclude that there exists $\Tilde{N}(\eta, \epsilon)$ such that for all $n > \Tilde{N}(\eta, \epsilon)$,
    \begin{equation}
    \label{eqn::BvM-proof-bound-first-term}
        \E_{P_0} [ d_{\mathrm{TV}} (\pi^{\mathrm{WMLAN}}_{n}(\cdot \mid \wempirical), \phi_n(\cdot) ) \mathbf{1} \{ A_n \} ] \leq \eta + 2 \epsilon
    \end{equation}
    Regarding the second term on RHS in \eqref{eqn::BvM-proof-upper-bound-on-TV}, \Cref{lemma:ballconvergence} shows that for a given $\eta, \epsilon > 0$, there exists a sequence $r_n \rightarrow + \infty$ and $N(\eta, \epsilon)$ such that, for all $n > N(\eta, \epsilon)$, 
    \begin{equation}
    \label{eqn::BvM-proof-bound-second-term}
        \P_{P_0} (A_n^c) = \P_{P_0} \Big ( \underset{g,h \in \overline{B}_{\textbf{0}}(r_n) }{\sup} f_n(g, h) > \eta \Big) \leq \epsilon.
    \end{equation}
    In the proof of \Cref{lemma:ballconvergence}, we use the stochastic Weighted M-LAN condition defined in Assumption \ref{assump::WMLAN}.
    Finally, from \eqref{eqn::BvM-proof-upper-bound-on-TV}, using bounds in \eqref{eqn::BvM-proof-bound-first-term} and \eqref{eqn::BvM-proof-bound-second-term}, for all $n > \max \{ N(\eta, \epsilon), \Tilde{N}(\eta, \epsilon) \}$,
    \begin{equation}
    \label{eqn::last-line-of-BvM-theorem}
    \E_{P_0} \left [ d_{\mathrm{TV}} \bigl (\pi^{\mathrm{WMLAN}}_{n,\boldsymbol{\alpha}}(\cdot \mid  \wempirical), \, \phi_n(\cdot) \bigr ) \right ] \leq (\eta + 2 \epsilon) + \epsilon = \eta + 3 \epsilon.
    \end{equation}
    Applying Markov's inequality gives the desired result.
\end{proof} 
\subsection{Proofs of Section \ref{sec::PIF}}
\noindent \textbf{Proof of \Cref{thm::global-bias-robust}}

\begin{proof}
    In both cases, we will show that the bound in \Cref{lemma::upper-bound-on-PIF} is finite, implying the uniform bound on the posterior influence function.

    \textbf{Case 1:} We assume the prior \( \pi(\theta) \) has a finite first moment and that \( \sup_{\theta \in \Theta} \pi(\theta) | \theta | < \infty \).
    First, since the loss \( \loss \) is bounded from below, we have that \( \loss_- := \inf_{\theta \in \Theta, x \in \mathcal{X}} \loss(x, \theta) > -\infty \). Therefore, we have
    \begin{equation*}
        \mposterior(\theta \mid F_n) \leq \frac{1}{Z} \exp(-n \loss_- ) \pi(\theta),
    \end{equation*}
    where \( Z = \int \exp \left ( - n \E_{F_n} [\loss (\theta, X)] \right ) \pi(\theta) \, d\theta \) is the normalizing constant. Therefore, since the prior \( \pi(\theta)\) is upper-bounded, we have that 
    \[ 
        \sup_{\theta \in \Theta} \mposterior(\theta \mid F_n) \leq \frac{1}{Z}\exp(-n \loss_- ) \sup_{\theta \in \Theta} \pi(\theta) < \infty. 
    \] 
    Moreover, since the prior \( \pi (\theta) \) exhibits a first moment, we conclude that
    \[
        \int \mposterior(\theta' \mid F_n) |\theta' | \, d\theta' \leq
        \frac{1}{Z} \exp(-n \loss_- ) \int \pi(\theta') |\theta' | \, d\theta'  < \infty.
    \]
    Finally, from the fact that \( \sup_{\theta \in \Theta} \pi(\theta) | \theta | < \infty \), we have
    \[
        \sup_{\theta \in \Theta} \mposterior(\theta \mid F_n) \left | \theta \right | \leq 
        \frac{1}{Z} \exp(-n \loss_- ) \left ( \sup_{\theta \in \Theta} \pi(\theta) | \theta | \right ) < \infty.
    \]
   Combining all this together, we conclude that the upper bound in \Cref{lemma::upper-bound-on-PIF} is finite:
    \[
    \sup_{\theta \in \Theta} 2B n \mposterior(\theta \mid F_n) \left ( | \theta | + \int \mposterior(\theta' \mid F_n) |\theta'| \, d\theta' \right ) < \infty.
    \]
    Hence, \( \left |\textrm{PIF}(x_0; \theta,\rho, F_n) \right | \) is uniformly bounded, and \( \mposterior(\cdot \mid F_n) \) is uniformly B-robust, as desired. We continue with studying the second case.

    \textbf{Case 2:} We assume that the loss function \(\loss\) is convex in \( \theta \) with \( \lim_{| \theta |\rightarrow \infty} \loss(x, \theta) = \infty \). We will bound the same three terms as in the previous case. 
    First, note that since the loss function \(\loss\) is convex in \( \theta \) with \( \lim_{| \theta |\rightarrow \infty} \loss(x, \theta) = \infty \), we have that the empirical average \( L_n(\theta) = \tfrac{1}{n} \sum_{i=1}^n \loss (X_i, \theta) \) is also convex and coercive. Hence, there exists constants \( a > 0\) and  \(b \in \R\) such that for all \(\theta\),
    \[
    L_n(\theta) \geq a |\theta|- b.
    \]
    As before, let \( Z = \int \exp \left ( - n \E_{X \sim F_n} [\loss (X,\theta)] \right ) \pi(\theta) \, d\theta \) be the normalizing constant and notice that we have \( Z \leq \sup_{\theta \in \Theta} \pi(\theta) \int  \exp(-n a | \theta' | +nb) \, d\theta'  < \infty\).
    Using the above and the fact that the prior is upper-bounded,
    \[
        \int \mposterior(\theta' \mid F_n) |\theta' | \, d\theta' \leq
        \frac{1}{Z} \sup_{\theta \in \Theta} \pi(\theta)  \int |\theta'| \exp(-n a | \theta' | +nb) \, d\theta'  < \infty.
    \]
    In the same fashion, we derive
    \[
        \sup_{\theta \in \Theta} \mposterior(\theta \mid F_n) \left | \theta \right | \leq
        \frac{1}{Z} \sup_{\theta' \in \Theta} \pi(\theta') \sup_{\theta \in \Theta} \bigl ( |\theta| \exp(-n a | \theta | +nb) \bigr ) < \infty,
    \]
    and
    \[
        \sup_{\theta \in \Theta} \mposterior(\theta \mid F_n) \leq
        \frac{1}{Z} \sup_{\theta' \in \Theta} \pi(\theta') \sup_{\theta \in \Theta} \bigl (  \exp(-n a | \theta | +nb) \bigr ) < \infty. 
    \]
    Finally, as before, combining this, we conclude that the upper bound in \Cref{lemma::upper-bound-on-PIF} is finite:
    \[
    \sup_{\theta \in \Theta} 2B n \mposterior(\theta \mid F_n) \left ( | \theta | + \int \mposterior(\theta' \mid F_n) |\theta'| \, d\theta' \right ) < \infty.
    \]
    Hence, \( \left |\textrm{PIF}(x_0; \theta,\rho, F_n) \right | \) is uniformly bounded, and \( \mposterior(\cdot \mid F_n) \) is uniformly B-robust, as desired. This completes the proof.
\end{proof}
\noindent \textbf{Proof of \Cref{prop::unbounded-score-implies-non-robust-posterior}}

\begin{proof}
    Recall that in \eqref{eqn::exact-PIF} we showed the posterior influence function can be written as
    \begin{equation*}
        \PIF(x_0; \,\theta, \rho, F_n) = n \mposterior(\theta \mid F_n) \left ( g(x_0, \theta) - \int \mposterior(\theta' \mid F_n) g(x_0, \theta') \, d\theta' \right ).
    \end{equation*}
    where 
    \(
        g(x,\theta) := \E_{F_n} [ \closs(X, \theta) ] - \closs(x, \theta)
    \)
    and \( \closs ( x, \theta) := \loss(x, \theta) - \loss(x, 0)\). We have
    \[
    \begin{split}
        h(x_0, \theta) 
        &=
        g(x_0, \theta) - \int \mposterior(\theta' \mid F_n) g(x_0, \theta') \, d\theta' \\
        &= 
        \E_{F_n} [ \closs(X, \theta) ] - \closs(x_0, \theta)
        -
        \int \mposterior(\theta' \mid F_n) \Bigl ( \E_{X \sim F_n} [ \closs(X, \theta') ] - \closs(x_0, \theta') \Bigr ) d\theta' \\
        &=
        \int \mposterior(\theta' \mid F_n) \Bigl ( \closs(x_0, \theta') - \closs(x_0, \theta) \Bigr ) \, d\theta' + C(\theta) \\
        &=
        \int \mposterior(\theta' \mid F_n) \Bigl ( \loss(x_0, \theta') - \loss(x_0, \theta) \Bigr ) \, d\theta' + C(\theta) \\
        &\geq
        \int \mposterior(\theta' \mid F_n) \score(x_0, \theta) (\theta' - \theta) \, d\theta' + C(\theta) \\
        &=
        \score(x_0, \theta) \int \mposterior(\theta' \mid F_n) (\theta' - \theta) \, d\theta' + C(\theta),
    \end{split}
    \]
    where we used the convexity of the \( \loss(x, \theta) \) to derive the inequality.
    Now pick \( \theta^* \neq \int \mposterior(\theta' \mid F_n) \theta' \, d\theta' \) with \( \pi(\theta^*) > 0\). Note that such \( \theta^* \) exists since the prior is non-degenerate. If \( \theta^* < \int \mposterior(\theta' \mid F_n) \theta' \, d\theta' \), we have
    \[
        \lim_{x \rightarrow +\infty} \score(x_0, \theta^*) \int \mposterior(\theta' \mid F_n) (\theta' - \theta^*) \, d\theta' = +\infty.
    \]
    In the case of \( \theta^* > \int \mposterior(\theta' \mid F_n) \theta' \, d\theta' \), by considering the other tail, we have
    \[
        \lim_{x \rightarrow - \infty} \score(x_0, \theta^*) \int \mposterior(\theta' \mid F_n) (\theta' - \theta^*) \, d\theta' = +\infty.
    \]
    Combining these two limits with the above inequality, we conclude
    \[
        \sup_{x_0} \bigl | h(x_0, \theta^*) \bigr |
        = 
        \infty.
    \]
    This in turn implies that \(\sup_{x_0} \bigl |\PIF(x_0; \,\theta, \rho, F_n) \bigr | = \infty\), showing the claim.
\end{proof}
\noindent \textbf{Proof of \Cref{prop::connections-to-influence-functions-of-posterior-moments-and-quantiles}}

\begin{proof}
    \textbf{Result (1).}  We first show that by the dominated convergence theorem, we can differentiate under the integral sign to obtain
    \begin{equation}
    \label{eq:IF_PIF}
        \textrm{IF}(x_0; T_k, F_n) =  \frac{\partial}{\partial \epsilon} T_k ( F_{n,\epsilon,x_0}) \Big |_{\epsilon=0} =
        \int_\Theta  \frac{\partial}{\partial \epsilon} \theta^k \mposterior(\theta \mid F_{n,\epsilon,x_0}) \Big |_{\epsilon=0} \, d\theta =
        \int_\Theta \theta^k \textrm{PIF}(x_0; \theta,\rho, F_n) \, d\theta.
    \end{equation}
    We will argue that for some constant $c>0$ and for all $\epsilon \in [0, \tfrac{1}{2}]$, we have
    \begin{equation}
    \label{eq:dom}
    \theta^k \mposterior(\theta \mid F_{n,\epsilon,x_0}) \leq c \, \theta^k \pi(\theta) \exp(-n \rho_-),
    \end{equation}
    where we have used the definition \( \loss_- := \inf_{\theta \in \Theta, X \in \mathcal{X}} \loss(X, \theta) > -\infty \), which follows as the loss \( \loss \) is assumed to be bounded from below. In particular, the right side of \eqref{eq:dom} is integrable since the prior $\pi$ is assumed to have a finite $k^{th}$ moment, justifying the interchange in \eqref{eq:IF_PIF}.
    
    The result in \eqref{eq:dom} can be justified as follows. First notice by definition,
    \begin{equation}
    \label{eq:dom1} 
    \mposterior(\theta \mid F_{n,\epsilon,x_0}) 
    =
    \frac{1}{Z_\epsilon } \pi(\theta) \exp \Bigl ( -n(1-\epsilon) \E_{F_n} \rho(X, \theta) - n\epsilon\rho(x_0, \theta) \Bigr ) 
    \leq
    \frac{1}{Z_\epsilon } \pi(\theta) \exp(-n \rho_-),
    \end{equation}
    where \( Z_\epsilon\) is the normalizing constant and the result in \eqref{eq:dom} follows if we show $Z_{\epsilon} \geq c > 0$ for any $\epsilon \in [0, \tfrac{1}{2}]$. Indeed,
    \[
    \begin{split}
    Z_\epsilon 
    =
    \int \pi(\theta) \exp  \Bigl( -n(1-\epsilon) &\E_{F_n} \rho(X, \theta) - n\epsilon\rho(x_0, \theta) \Bigr) \, d\theta  \\
    &\geq
    \int \pi(\theta) \exp \bigl ( -n \E_{ F_n} \rho(X, \theta) - \tfrac{1}{2}n \loss(x_0, \theta) \bigr) \, d\theta >0.
    \end{split}
    \]

    Recall the following bound on the posterior influence function in \eqref{eqn::bound-on-PIF-2}:
    \begin{equation}
    \begin{split}
    \label{eq:PIF_bound_1}
        |\textrm{PIF}(x_0; \theta, \rho,F_n)| &\leq 2B n \mposterior(\theta \mid F_n) \left ( | \theta | + \int \mposterior(\theta' \mid F_n) |\theta'| \, d\theta' \right ),
    \end{split}
    \end{equation}
    where \( B = \sup_{X \in \mathcal{X}} \sup_{\theta \in \Theta} \left | \psi(X,\theta) \right | < \infty\). Using the bound in \eqref{eq:PIF_bound_1} and the result in \eqref{eq:IF_PIF}, we have 
    \[
        \bigl | \mathrm{IF}(x_0; T_k,F_n) \bigr | 
        \leq
        \int |\theta| ^k |\textrm{PIF}(x_0; \,\theta,\rho, F_n)| \, d\theta 
        \leq
        2B n \int |\theta|^k \mposterior(\theta \mid F_n)  \bigl (  |\theta| +A \bigr ) \, d\theta 
        <
        \infty,
    \]
    where the final inequality uses that  the prior \( \pi \) has a finite \( (k+1) \)-th moment, and the following: by \eqref{eq:dom1} and the assumption that the prior admits the first moment, we have
    \begin{equation}
    \label{eq:A_def}
        A := \int \mposterior(\theta' \mid F_n) |\theta'| \, d\theta' \leq \frac{1}{Z_0}\exp(-n \rho_-)\int    |\theta'|\pi(\theta') \, d\theta' < \infty.
    \end{equation}
    This proves the first claim.
    
    \textbf{Result (2).} Regarding the second part, note that in \eqref{eqn::influence-function-of-posterior-quantile} we showed that, for any \( \tau \in (0, 1) \),
    \[
    \mathrm{IF}\bigl(x_0;\,T_\tau,F_n\bigr)
        =
        -
        \frac{\displaystyle
          \int_{-\infty}^{T_\tau(F_n)}
            \PIF\bigl(x_0; \,\theta', \rho,F_n\bigr)\,d\theta'
        }{
          \displaystyle \mposterior\bigl(T_\tau(F_n)\mid F_n\bigr)
        }.
    \]
    Therefore, using the bound in \eqref{eq:PIF_bound_1} and the definition of $A$ in \eqref{eq:A_def}, we have
    \[
        \bigl | \mathrm{IF}\bigl(x_0;\,T_\tau,F_n\bigr) \bigr | 
        \leq
        \frac{\displaystyle
          \int_{-\infty}^{\infty}
            \bigl | \PIF\bigl(x_0; \theta',\rho, F_n\bigr) \bigr | \,d\theta'
        }{
          \displaystyle \mposterior\bigl(T_\tau(F_n)\mid F_n\bigr)
        }
        \leq
        \frac{\displaystyle
          2B n \int_{-\infty}^{\infty} \mposterior(\theta' \mid F_n)  (  |\theta'| +A ) \, d\theta'
        }{
          \displaystyle \mposterior\bigl(T_\tau(F_n)\mid F_n\bigr)
        } <\infty,
    \]
    where in the last inequality we used the fact that the loss \(\loss\) is lower bounded and that the prior has a finite first moment, and hence the M-posterior \( \mposterior(\cdot \mid F_n)\) also admits the first moment. This finishes the proof.
\end{proof}
\noindent \textbf{Proof of \Cref{prop::bounded-PIF-of-reweighted-posteriors}}

\begin{proof}
    By integrating the latent weights, we have that the reweighted posterior is equal to
    \begin{align*}
    \pi_\alpha(\theta\mid F_n) \propto \pi(\theta) \prod_{i=1}^n \int \pi_\alpha(\alpha_i) f(X_i \mid \theta)^{\alpha_i} \, d\alpha_i
  & = 
    \pi(\theta)\prod_{i=1}^n
    \int_0^\infty
    e^{-\alpha_i\,g(X_i,\theta)}
    \frac{\lambda^k}{\Gamma(k)}\alpha_i^{k-1}e^{-\lambda\alpha_i}
    \, d\alpha_i \\
    &\propto
    \pi(\theta)\prod_{i=1}^n(\lambda + g(X_i,\theta))^{-k}.
    \end{align*}
    Therefore, the reweighted posterior is actually an M-posterior with loss function
    \[
    \loss(X, \theta) = k \,\log \bigl(\lambda + g(X,\theta) \bigr).
    \]
    The corresponding score function is 
    \[
    \psi(X, \theta) = \frac{k \, \nabla_\theta g(X,\theta)}{\lambda + g(X, \theta)}.
    \]
    Now, note that since by assumption \( (X, \theta) \mapsto \log  [ g(X, \theta)  ]\) is \(L\)-Lipschitz in \( \theta\) for all \(X\), we have that
    \[
    \sup_{\theta, x} \left |\psi(x,\theta)\right|\leq k\cdot\sup_{\theta, x} \left | \frac{\nabla_\theta g(x, \theta)}{g(x, \theta)}  \right | \leq kL.
    \]
    Finally, by assumptions on the prior, \Cref{thm::global-bias-robust} implies that the reweighted posterior \( \pi_\alpha(\cdot \mid F_n) \) is uniformly B-robust, as required.
\end{proof}

\subsection{Proofs of Section \ref{section::posterior-BP}}
\noindent \textbf{Proof of \Cref{thm::breakdown-point-convex-loss}}

\begin{proof}
    Note that by replacing \( \loss(X_i - \theta)\) with \( \loss(X_i - \theta) - \loss(X_i)\), the M-posterior remains the same:
    \begin{equation}
    \label{eq:diff}
        \frac{\pi(\theta) \exp \left ( - \sum_{i=1}^n [ \loss(X_i - \theta) - \loss(X_i)]  \right )}{\int \pi(\theta) \exp \left ( - \sum_{i=1}^n [ \loss(X_i - \theta) - \loss(X_i)]  \right ) \, d\theta}
        =
        \frac{\pi(\theta) \exp \left ( - \sum_{i=1}^n\loss(X_i - \theta)  \right )}{\int \pi(\theta) \exp \left ( - \sum_{i=1}^n \loss(X_i - \theta)   \right ) \, d\theta}.
    \end{equation}
    Furthermore, since \(\loss\) is convex, \( \loss(x - \theta) - \loss(x)\) is decreasing as a function of \(x\). Setting \( c = \sup \psi(x) < \infty\), by the mean value theorem:
    \[
        \lim_{x \rightarrow \pm \infty} \loss(x - \theta) - \loss(x) = \mp c \, \theta.
    \]
    At the same time, we have that
    \[
        \lim_{\theta \rightarrow \pm \infty} \frac{\loss (x - \theta) - \loss(x)}{| \theta |}
        =
        \lim_{\theta \rightarrow \pm \infty} - \frac{1}{| \theta |} \int_0^\theta \psi(x-u) \, du
        =
        c.
    \]
    Recall that we defined \( T_k\) to be the \(k\)-th moment of M-posterior \( \mposterior (\cdot \mid F_n) \) in \eqref{eq:Tk}. 
    In the same fashion, let  \( \widetilde T_k\) be the \(k\)-th moment of the corrupted M-posterior \( \mposterior (\cdot \mid \P_{(n, m)}) \), where \( \P_{(n, m)} \) is the empirical distribution of the corrupted sample \( \corruptedsample\). 
    From \Cref{lemma::posterior-moments-monotone} we conclude that the maximum bias to the mean that can be caused by contaminating \( m \) observations occurs if we move all contaminated points to \( + \infty\). 
    Hence the mean achieving the biggest bias is equal to
    \begin{equation}
        \widetilde T_1 = \frac{\int \theta \pi(\theta) \exp \left (-\sum_{i=1}^{n-m} \rho(X_i - \theta) + mc\theta \right) \, d\theta}{\int \pi(\theta) \exp \left (-\sum_{i=1}^{n-m} \rho(X_i - \theta) + mc\theta\right) \, d\theta}.
    \end{equation}
    At the same time, again from \Cref{lemma::posterior-moments-monotone}, we conclude that the maximum bias to the second moment can be caused if all contaminated observations are chosen from the set \( \{ -\infty, +\infty \}. \) Let \(m_{\infty} \leq m\) be the number of points chosen to be equal to \( +\infty\).
    Then the second moment achieving the biggest bias is of the form
    \begin{equation}
        \widetilde T_2 = \frac{\int \theta^2 \pi(\theta) \exp \left (-\sum_{i=1}^{n-m} \rho(X_i - \theta) + m_{\infty} c\theta - (m - m_{\infty})c \theta \right) \, d\theta}{\int \pi(\theta) \exp \left (-\sum_{i=1}^{n-m} \rho(X_i - \theta) + m_{\infty} c\theta - (m - m_{\infty})c \theta\right) \, d\theta}.
    \end{equation}
    Note that if we show that both \( \widetilde T_1\) and \(\widetilde T_2\) are finite for some \( m \leq n\), from the upper-bound in \Cref{lemma::bounds-on-W2}, we conclude that \( \sup_{\, F_{(n,m)}} W_2^2 (\mposterior(\cdot \mid F_{(n,m)}), \, \mposterior (\cdot \mid F_n ) ) < \infty, \) and hence \(\varepsilon^*_{W_2}(\pi^{\rho}_n, X^n) \geq \frac{m+1}{n}\). 
    On the other hand, if we show that \( \widetilde T_1\) is infinite, from the lower-bound in \Cref{lemma::bounds-on-W2}, we can conclude that \(\varepsilon^*_{W_2}(\pi^{\rho}_n, X^n) \leq \frac{m}{n}.\) 
    Note that we are not necessarily using the same distribution $F_{(n,m)}$ in finding the worst case bias for the first and second moments. We are allowed to do this because the upper-bound in \Cref{lemma::bounds-on-W2} can trivially be further bounded by using \( \sup(f+g) \leq \sup(f) + \sup(g).\) We continue with studying the behavior under different prior scenarios discussed above:
    
    \textbf{Uninformative prior.} In this setting $\pi=1$ so that we have 
    \[
        \widetilde T_1 = \frac{\int \theta \exp \left (-\sum_{i=1}^{n-m} \rho(X_i - \theta) + mc\theta \right) \, d\theta}{\int  \exp \left (-\sum_{i=1}^{n-m} \rho(X_i - \theta) + mc\theta\right) \, d\theta}.
    \]
    For large \( |\theta|\), the exponent in the above formula is \( c[ - (n-m)|\theta| (1 + o(1)) + m\theta] \); hence, the above is finite if and only if \( n - m > m\), or equivalently \( m < n / 2\).

By the same arguments, we conclude that \( \widetilde T_2 \) is finite whenever \( m < n / 2\), for all possible choices of \( m_\infty\). Indeed,  for large \( |\theta|\), the exponent in the above formula is \( c[ - (n-m)|\theta| (1 + o(1)) + m_{\infty}\theta - (m- m_{\infty})\theta] = c[ - (n-m)|\theta| (1 + o(1)) - (m- 2m_{\infty})\theta] \) and so \( \widetilde T_2 \) is finite whenever $n - m > 2 m_{\infty} - m$ or $m_{\infty} < \frac{n}{2}$.
We conclude that
    \( \varepsilon^*_{W_2}(\pi^{\rho}_n, X^n) = \frac{1}{2}.\)

    \textbf{Super-exponential prior.} In this case \(\pi\propto \exp(-h) \), where \( h \) is convex, symmetric and has a bounded derivative \( h'\) and
    \begin{equation}
    \label{eqn::convex-case-corrupted-mean}
        \widetilde T_1 = \frac{\int \theta \exp \left (-h(\theta) -\sum_{i=1}^{n-m} \rho(X_i - \theta) + mc\theta \right) \, d\theta}{\int  \exp \left (- h(\theta) -\sum_{i=1}^{n-m} \rho(X_i - \theta) + mc\theta\right) \, d\theta}.
    \end{equation}
    By denoting \( \sup h'(\theta) =:  c_h  < \infty\), we have, for large \( | \theta |\), that the exponent in the above formula is equal to 
    \[ - (c_h + c(n-m) ) |\theta| (1 +o(1)) + m c \theta . \]
    Hence, \( \widetilde T_1\) is finite if and only if \( c_h + c( n-m ) >c \, m \), equivalently \( m < n/2 + c_h/(2c) \). By the same arguments as in the previous case concerning the second moment, 
   \( \varepsilon^*_{W_2}(\pi^{\rho}_n, X^n) \ge \frac{1}{2}\) and \( \varepsilon^*_{W_2}(\pi^{\rho}_n, X^n) \downarrow \frac{1}{2} \) as \( n \rightarrow \infty \).

    \textbf{Sub-exponential prior.} Here we have  a convex and symmetric \( h\), but unbounded derivative \( h'\). Then for large \( | \theta | \), the exponent in \eqref{eqn::convex-case-corrupted-mean} is
    \[ -c_h(\theta) | \theta | + c \left [ - (n-m)|\theta| (1 + o(1)) + m\theta \right ]   \]
    for some function \( c_h(\theta)\) with \( \lim_{\theta \rightarrow \pm \infty} c_h(\theta) = \infty\). Since \( c_h(\theta) > m \, c\) eventually, we have that \( \widetilde T_1\) is finite for any choice of \( m \leq n\). Since the same arguments can be applied to \( \widetilde T_2 \), the M-posterior cannot be broken.
\end{proof}
\noindent \textbf{Proof of \Cref{thm::breakdown-point-redescending-loss}}

\begin{proof}
    Let \( X^n\) be the original sample, and \( \corruptedsample\) be the corrupted sample differing from the original one in at most \( m \) entries.
    Let  \( \widetilde T_k\) be the \(k\)-th moment of the corrupted M-posterior \( \mposterior (\cdot \mid \P_{(n, m)}) \): as in \eqref{eq:diff},
    \[
        \widetilde T_k = 
        \frac{\int \theta^k \pi(\theta) \exp \left (  - \Delta_{\corruptedsample} (\theta) \right )  \, d\theta}{\int \pi(\theta) \exp \left (  -\Delta_{\corruptedsample} (\theta) \right ) \, d\theta}.
    \]
    Now, for \( k \in \{ 1, 2\}\) and \( 2m < n\), it follows from \Cref{lemma::huber-paper-redescending-loss} and the integrability assumption that
    \begin{equation*}
    \begin{split}
        \left | \widetilde T_k \right | 
        &\leq 
        \frac{\int  |\theta|^k \pi(\theta) \exp \left (  - \Delta_{\corruptedsample} (\theta) \right ) \, d\theta}{\int \pi(\theta) \exp \left (  -\Delta_{\corruptedsample} (\theta) \right ) \, d\theta} 
        \leq 
        \frac{\int | \theta |^k \pi(\theta) \exp \left (  - ( n - 2m ) \rho (\theta) + C \right )  \, d\theta}{\int \pi(\theta) \exp \left (  -  n \rho (\theta) - C\right ) \, d\theta} < \infty.
    \end{split}
    \end{equation*}
    Finally, by considering the upper bound on the Wasserstein-2 distance in \Cref{lemma::bounds-on-W2}, we have the first part of the claim; namely, that the breakdown satisfies
    \begin{equation}
    \label{eqn::redescending-bp-upper-bound}
        \varepsilon^*_{W_2}(\pi^\rho_n, X^n) \geq \frac{1}{2}.
    \end{equation}
   Regarding the second claim, if we take \( \pi = 1 \), the posterior mean is translation equivariant:
    \[
        \frac{\int \theta \exp(-\sum_{i=1}^n \loss((X_i + c) - \theta) \, d \theta} {\int \exp(-\sum_{i=1}^n \loss((X_i + c) - \theta) \, d \theta}  
        =
        \frac{\int \theta \exp(-\sum_{i=1}^n \loss(X_i - \theta)) \, d \theta} {\int \exp(-\sum_{i=1}^n \loss(X_i - \theta)) \, d \theta}  + c.
    \]
    It follows that the breakdown point of the posterior mean is at most 1/2 \citep{donoho:huber1983}. Now, from the lower bound in \Cref{lemma::bounds-on-W2}, we have that the breakpoint of the posterior is upper bounded by the breakdown point of the mean; therefore, \( \varepsilon^*_{W_2}(\pi^\rho_n, X^n) \leq \frac{1}{2}\). Combining this with \eqref{eqn::redescending-bp-upper-bound} finishes the claim.
\end{proof}
\noindent \textbf{Proof of \Cref{prop::bp-of-posterior-mean}}

\begin{proof}
    Take \( m \leq n\) with \( m/n < \varepsilon^*_{W_2}(\pi^\rho_n, X^n).\) Let \( \corruptedsample\) be the corrupted sample, and denote with \( F_{(n,m)}\) the corresponding corrupted empirical distribution. Furthermore, let \( T_1\) be the mean of \(\mposterior(\cdot \mid F_n)\), and \( \widetilde T_1\) be the mean of \(\mposterior(\cdot \mid F_{(n,m)})\). From the lower bound in \Cref{lemma::bounds-on-W2}, we have
    \[
        \sup_{\corruptedsample} \bigl | T_1 - \widetilde T_1 \bigr | 
        \leq
        \sup_{F_{(n,m)}} W_2 \bigr (\mposterior(\cdot \mid F_{(n,m)}), \, \mposterior(\cdot \mid F_n) \bigl )
        \, <
        \infty.
    \]
    This shows \( \varepsilon^*(T_1, X^n) \geq \varepsilon^*_{W_2}(\pi^\rho_n, X^n),\) as required.
\end{proof}
\noindent \textbf{Proof of \Cref{prop::bp-posterior-quantiles}}

\begin{proof}
    First, note that for \( X \sim P\) and \( Y \sim Q\), we have
    \begin{equation}
    \label{eqn::bound-on-second-moment}
        \E[Y^2] \leq 2 W_2^2(P, Q) +2\E[X^2].
    \end{equation}
    Take \( m \leq n\) with \( m/n < \varepsilon^*_{W_2}(\pi^\rho_n, X^n)\); hence, we cannot break the M-posterior with \( m \) samples. Let \( \corruptedsample\) be the corrupted sample, and denote with \( F_{(n,m)}\) the corresponding corrupted empirical distribution. By setting \( P= \mposterior(\cdot \mid F_n)\) and \( Q= \mposterior(\cdot \mid F_{(n,m)})\) in \eqref{eqn::bound-on-second-moment}, we have that
    \begin{equation}
    \label{eqn::upper-bound-in-bp-of-posterior-quantiles}
        \sup_{F_{(n,m)}} \E_{Y \sim \mposterior(\cdot \mid F_{(n,m)})}  [ Y^2  ]
        \leq
        \sup_{F_{(n,m)}} W_2^2 \bigl ( \mposterior(\cdot \mid F_n) , \, \mposterior(\cdot \mid F_{(n,m)}) \bigr ) + \E_{X \sim \mposterior(\cdot \mid F_n)}  [ X^2 ] 
        <
        \infty.
    \end{equation}
    Denote by \( \widetilde T_\tau \) the (left) \( \tau \)-quantile of \( \mposterior(\cdot \mid F_{(n,m)})\). Then, from \Cref{lemma::bound-on-quantiles}, we have
    \[
        \bigl | \widetilde T_\tau \bigr | 
        \leq 
        \sqrt{  \E_{Y \sim \mposterior(\cdot \mid F_{(n,m)})}  [ Y^2 ] \max\left \{ \frac{\tau}{1-\tau}, \frac{1 - \tau}{\tau} \right \}} + \E_{Y \sim \mposterior(\cdot \mid F_{(n,m)})} \bigl [ |Y| \bigr ].
    \]
    Finally, from \eqref{eqn::upper-bound-in-bp-of-posterior-quantiles}, we have that
    \[
        \sup_{\corruptedsample} \bigl | \widetilde T_\tau \bigr | < \infty.
    \]
    In other words, we cannot break the \( \tau\)-quantile with corrupting \( m\) observations. This shows \( \varepsilon^*(T_\tau, X^n) \geq \varepsilon^*_{W_2}(\pi^\rho_n, X^n) \), as required.
\end{proof}
%
%
%
%

\section{Proofs of Supporting results}

\noindent \textbf{Proof of \Cref{lemma::upper-bound-on-PIF}}

\begin{proof}
    To that end, let \( \closs ( x,\theta) := \loss(x,\theta) - \loss(x,0)\) be the re-centered loss, and note that 
    \begin{equation*}
        \mposterior(\theta \mid G) =
        \frac{\exp \left ( -n \E_{G} [\loss (X,\theta)] \right) \pi(\theta)}{\int \exp \left ( -n \E_{G} [\loss (X,\theta')] \right) \pi(\theta') \, d\theta'} =
        \frac{\exp \left ( -n \E_{G} [\closs (X, \theta)] \right) \pi(\theta)}{\int \exp \left ( -n \E_{G} [\closs (X,\theta')] \right) \pi(\theta') \, d\theta'} = 
        \centeredmposterior(\theta \mid G).
    \end{equation*}
    For \( \epsilon > 0\), taking \( F_{n,\epsilon,x_0} = (1 - \epsilon) F_n + \epsilon \delta_{x_0} \), we let
    \begin{equation*}
        h(\epsilon; \theta) := \exp \left ( -n \E_{F_{n,\epsilon,x_0}} [\closs (X, \theta)] \right) \pi(\theta),
    \end{equation*}
    so that we can write the posterior as
    \begin{equation*}
        \mposterior(\theta \mid F_{n,\epsilon,x_0}) = \centeredmposterior(\theta \mid F_{n,\epsilon,x_0}) = \frac{h(\epsilon; \theta)}{\int h(\epsilon; \theta') \, d\theta'}.
    \end{equation*}
    We now find that
    \begin{equation*}
        h'(\epsilon; \theta) := \frac{\partial}{\partial \epsilon} h(\epsilon ; \theta) = n h(\epsilon; \theta) \left( \E_{F_n} [ \closs(X, \theta) ] - \closs(x_0, \theta) \right),
    \end{equation*}
    and
    \begin{equation*}
        \frac{\partial}{\partial \epsilon} \int h(\epsilon; \theta') \, d\theta' = \int \frac{\partial}{\partial \epsilon} h(\epsilon; \theta') \, d\theta' = \int h'(\epsilon; \theta') \, d\theta'.
    \end{equation*}
    Using these derivations, we have
    \begin{equation*}
    \begin{split}
        \frac{\partial}{\partial \epsilon} \mposterior(\theta \mid F_{n,\epsilon,x_0}) &= \frac{h' (\epsilon; \theta)}{\int h(\epsilon; \theta') \, d\theta'} - \frac{h(\epsilon; \theta) \int h'(\epsilon; \theta') \, d\theta'}{\left ( \int h(\epsilon; \theta') \, d\theta' \right ) ^ 2} \\
        &= \frac{h(\epsilon; \theta)}{\int h(\epsilon; \theta') \, d\theta'} \left ( \frac{h'(\epsilon; \theta)}{h(\epsilon; \theta)} - \frac{\int h'(\epsilon; \theta') \, d\theta'}{\int h(\epsilon; \theta') \, d\theta'} \right ) \\
        &= n \mposterior(\theta \mid F_{n,\epsilon,x_0}) \left ( \E_{F_n} [ \closs(X, \theta) ] - \closs(x_0, \theta) \right ) \\
        &\quad - n \mposterior(\theta \mid F_{n,\epsilon,x_0}) \left ( \int \mposterior(\theta' \mid F_{n,\epsilon,x_0}) ( \E_{F_n} [ \closs(X, \theta') ] - \closs(x_0, \theta') ) \, d\theta' \right ).
    \end{split}
    \end{equation*}
    By denoting
    \begin{equation}
    \label{eqn::defining-function-g}
        g(x_0, \theta) = \E_{F_n} [ \closs(X, \theta) ] - \closs(x_0, \theta),
    \end{equation}
    we have that the posterior influence function can be written as
    \begin{equation*}
        \textrm{PIF}(x_0; \,\theta, \rho, F_n) = n \mposterior(\theta \mid F_n) \left ( g(x_0, \theta) - \int \mposterior(\theta' \mid F_n) g(x_0, \theta') \, d\theta' \right ).
    \end{equation*}
    We can bound it as follows:
    \begin{equation}
    \label{eqn::bound-on-PIF} 
        |\textrm{PIF}(x_0; \,\theta, \rho, F_n)| \leq n \mposterior(\theta \mid F_n) \left ( | g(x_0, \theta) | + \int \mposterior(\theta' \mid F_n) |g(x_0, \theta')| \, d\theta' \right ).
    \end{equation}
    Furthermore, by noting that $\left | g(x_0, \theta) \right | \leq 2B |\theta |,$ as required we conclude that
    \begin{equation}
    \label{eqn::bound-on-PIF-2}
        |\textrm{PIF}(x_0; \,\theta, \rho, F_n)|
        \leq 2B n \mposterior(\theta \mid F_n) \left ( | \theta | + \int \mposterior(\theta' \mid F_n) |\theta'| \, d\theta' \right ).
   \end{equation}
\end{proof}
\noindent \textbf{Proof of \Cref{lemma::bounds-on-W2}}

\begin{proof}
    Write \( W_2^2(P,Q) = \inf_{\pi\in\Pi(P,Q)} \E_\pi  [(X-Y)^2 ] \), where \(\Pi(P,Q)\) is the set of couplings of \( (X,Y) \) with marginals \( P,Q \).
    
    \textbf{Lower bound:}
    For any coupling $\pi$, we have $\E_\pi [(X-Y)^2] \geq (\mathbb{E}_\pi[X-Y])^2 = (\mu_P-\mu_Q)^2,$
    by Jensen’s inequality.
    Taking the infimum over $\pi$ yields $W_2^2(P,Q)\ge (\mu_P-\mu_Q)^2$.
    
    \textbf{Upper bound:}
    Consider the independent coupling $\pi=P\otimes Q$. Then
    \[
    \E_{P\otimes Q} [(X-Y)^2] = \Var(X-Y) + (\mu_P-\mu_Q)^2 = \sigma_P^2+\sigma_Q^2 + (\mu_P-\mu_Q)^2,
    \]
    since $\Cov(X,Y)=0$ under independence. Hence
    \[
    W_2^2(P,Q) = \inf_{\pi\in\Pi(P,Q)} \mathbb{E}_\pi[(X-Y)^2] \leq \E_{P\otimes Q}[(X-Y)^2] = \sigma_P^2+\sigma_Q^2 + (\mu_P-\mu_Q)^2.
    \]
    Combining the two bounds gives the claim.
\end{proof}
\noindent \textbf{Proof of \Cref{lemma::posterior-moments-monotone}}

\begin{proof}
    Let \( T_k \) be the \( k\)-th moment of M-posterior \( \mposterior(\cdot \mid X^n)\):
    \begin{equation}
    \label{eq:Tk}
        T_k = \frac{\int \theta^k \pi(\theta) \exp \left (-\sum_{i=1}^n \rho(X_i - \theta) \right) \, d\theta}{\int \pi(\theta) \exp \left (-\sum_{i=1}^n \rho(X_i - \theta) \right) \, d\theta}.
    \end{equation}
    Denote the numerator and denominator of the above by $N$ and $D$ respectively, so that $T_k = N/D$.  By the dominated convergence theorem, we can differentiate under the integral sign in both the numerator and denominator with
    \[ \frac{\partial N}{\partial X_i} =  \int \theta^k \pi(\theta) \frac{\partial}{\partial X_i}\exp \left (-\sum_{i=1}^n \rho(X_i - \theta) \right) \, d\theta =  -\int \theta^k \pi(\theta) \exp \left (-\sum_{i=1}^n \rho(X_i - \theta) \right) \psi(X_i - \theta)\, d\theta,\]
    and
        \[ \frac{\partial D}{\partial X_i} = -\int \pi(\theta) \exp \left (-\sum_{i=1}^n \rho(X_i - \theta) \right) \psi(X_i - \theta)\, d\theta.\]
    Letting \(\exp(\ldots)\) denote \(\exp \left (-\sum_{i=1}^n \rho(X_i - \theta) \right)\), we find
    \[
    \begin{split}
        \frac{\partial T_k}{\partial X_i} 
        =
        \frac{\partial}{\partial X_i} \left ( \frac{N}{D}\right ) 
        &=
        \frac{1}{D^2} \left[D\frac{\partial N}{\partial X_i} - N\frac{\partial D}{\partial X_i} \right] \\
        &=
        \frac{- \int \theta^k \pi(\theta) \score (X_i - \theta) \exp(\ldots) \, d\theta + T_k \int \pi(\theta) \score(X_i - \theta) \exp(\ldots) \, d\theta }{D } \\
        &=
        \frac{\int (T_k - \theta^k)  \pi(\theta) \score(X_i - \theta) \exp(\ldots) \, d\theta }{D}.
    \end{split}
    \]
    Because
    \[
        \int (T_k - \theta^k) \pi(\theta) \exp(\ldots) \, d\theta = 0,
    \]
    we can write
    \[
        \frac{\partial T_k}{\partial X_i} 
        =
        \frac{\int (T_k - \theta^k) \left [ \score(X_i - \theta) - \score(X_i - T_k^{1/k}) \right ] \pi(\theta) \exp(\ldots) \, d\theta }{D}.
    \]
    Since \( \loss \) is convex, \( \score\) is monotone non-decreasing. Therefore, for odd \( k \), \( (T_k - \theta^k) [ \score(X_i - \theta) - \score(X_i - T_k^{1/k}) ] \geq 0. \) This in turn implies that the integrand in the above expression is positive; hence, \( {\partial T_k}/{\partial X_i} \geq 0 \), showing the first part of the claim.

    On the other hand, for even \( k \), \( (T_k - \theta^k)  [ \score(X_i - \theta) - \score(X_i - T_k^{1/k})  ] \) is positive for \( \theta \geq -T_k^{1/k}\) and negative for \( \theta < -T_k^{1/k}\). Hence, as a function of \( X_i\), \( T_k\) is decreasing to some point, and then increasing.
\end{proof}
\noindent \textbf{Proof of \Cref{lemma::huber-paper-redescending-loss}}

\begin{proof}
    See Lemma 4.3 in \cite{huber1984}.
\end{proof}
\noindent \textbf{Proof of \Cref{lemma::bound-on-quantiles}}

\begin{proof}
    For \( t > 0\), by Cantelli inequality,
    \[
        \P(X - \mu \geq t) \leq \frac{\sigma^2}{\sigma^2 + t^2}.
    \]
    If we pick \( t \) such that \( \sigma^2 /(\sigma^2 + t^2) = 1 - \tau\), then \( T_\tau \leq \mu +t.\) The condition solves to 
    \[
        t = \sigma \sqrt{\frac{\tau}{1-\tau}},
    \]
    and hence 
    \[
        T_\tau \leq \mu +\sigma \sqrt{\frac{\tau}{1-\tau}}.
    \]
    Similarly, by the same arguments applied to the left tail, we have that
    \[
        T_\tau \geq \mu - \sigma \sqrt{\frac{1-\tau}{\tau}}.
    \]
    Combining these two bounds yields the claim.
\end{proof}
%

\section{Supporting lemmas} \label{appendix::supporting-lemmas}

\begin{lemma}
\label{lemma::WMLAN_conditions}
    If the function $\theta \mapsto \loss(X_1, \theta)$ is differentiable at $\theta^*$ in $P_0$-probability with derivative $\psi_{\theta^*}(X_1)$ and:
    \begin{enumerate}[(i)]
        \item there is an open neighborhood $U$ of $\theta^*$ and a square-integrable function $m_{\theta^*}$ such that for all $\theta_1, \theta_2 \in U$:
        \[
        \left| \loss ( \cdot, \theta_1) - \loss( \cdot, \theta_2) \right |  \leq m_{\theta^*} \|\theta_1 - \theta_2\|_2, \quad (P_0 \text{-a.s.}).
        \]
        
        \item the map \( \theta \mapsto \E_{P_0} \rho(\cdot, \theta) \) admits a second-order Taylor expansion at \( \theta^* \), i.e.,
        \[
        \E_{P_0} [\loss( \cdot, \theta) - \loss( \cdot, \theta^*) ] = \frac{1}{2} (\theta - \theta^*)^\top V_{\theta^*} (\theta - \theta^*) + o(\|\theta - \theta^*\|^2), \quad (\theta \rightarrow \theta^*),
        \]
        where $V_{\theta^*}$ is a positive-definite $d \times d$-matrix,

        \item
        the parameter $\theta^*$ is the unique minimizer of $\theta \mapsto \E_{P_0}\rho(\cdot,\theta)$, and the weighted M-estimator $\mestimator$ exists and satisfies $\mestimator \overset{P_0}{\rightarrow} \theta^*$,
    \end{enumerate}
    then Assumption \ref{assump::WMLAN} (Weighted M-LAN) holds.
\end{lemma}

\begin{proof}
    We proceed using similar ideas to Lemma 19.31 in \cite{Vaart_1998}. Let \( (h_n) \subset \mathbb{R}^p \) be a bounded sequence and let \( \boldsymbol{\alpha} = (\alpha_n) \subset \R_{\geq 0} \) be an arbitrary sequence of weights with finite second moment, i.e.\
    $\limsup_{n \rightarrow \infty} \frac{1}{n}\sum_{i=1}^n \alpha_i^2 < \infty.$
    Define the weighted empirical process
    \[
    \textbf{G}_n^\alpha f := \sqrt{n} \left( \frac{1}{n} \sum_{i=1}^n \alpha_i f(X_i) - \frac{1}{n} \sum_{i=1}^n \alpha_i \E_{P_0} f \right)
    = \left( \frac{1}{n} \sum_{i=1}^n \alpha_i \right) \cdot \sqrt{n} \left( \frac{1}{\sum_{i=1}^n \alpha_i} \sum_{i=1}^n \alpha_i f(X_i) - \E_{P_0} f \right),
    \]
    and define
    \[
    \delta_n(x) := \sqrt{n} \left[ \rho(x, \theta^* + h_n/\sqrt{n}) - \rho(x, \theta^*) \right] - h_n^\top \psi_{\theta^*}(x).
    \]
    Now, note that random variables
    \[
    \textbf{G}_n^\alpha \bigl ( \delta_n \bigr) = \textbf{G}_n^\alpha \bigl ( \sqrt{n} \left[ \rho(\cdot, \theta^* + h_n/\sqrt{n}) - \rho(\cdot, \theta^*) \right] - h_n^\top \psi_{\theta^*} \bigr )
    \]
    have mean zero. Also, by condition \textit{(i)} and the fact that sequence $(h_n)$ is bounded, we can apply Dominated Convergence Theorem to get
    $
    \mathrm{Var} [ \delta_n (X) ] \rightarrow 0,
    $
    since loss function \( \loss \) is differentiable at \( \theta^* \) by assumption. Now,
    \[
    \mathrm{Var} \bigl[ \textbf{G}_n^\alpha( \delta_n ) \bigr ] =
    \left (\frac{\alpha_1 ^ 2 + \ldots + \alpha_n ^ 2}{n} \right ) \mathrm{Var} \bigl [ \delta_n (X) \bigr] \rightarrow 0
    \]
    since, by assumption, weight sequence $(\alpha_n)$ has a finite second moment. 
    Combining the facts that all of these variables have zero mean and variance converging to zero, we conclude
    \[
    \textbf{G}_n^\alpha( \delta_n ) = \textbf{G}_n^\alpha \bigl ( \sqrt{n} \left[ \rho(\cdot, \theta^* + h_n/\sqrt{n}) - \rho(\cdot, \theta^*) \right] - h_n^\top \psi_{\theta^*} \bigr ) \overset{P_0}{\rightarrow} 0.
    \]
    Expanding the above, we find
    \[
    \sum_{i=1}^n \alpha_i \bigl [ \rho(X_i, \theta^* + h_n/\sqrt{n}) - \rho(X_i, \theta^*) \bigr ] - \left( \sum_{i=1}^n \alpha_i \right) \E_{P_0} \left[ \rho(\cdot, \theta^* + h_n/\sqrt{n}) - \rho(\cdot, \theta^*) \right] - \textbf{G}_n^\alpha (h_n^\top \psi_{\theta^*})
    \]
    is equal to \(o_{P_0}(1)\).
    By the second-order Taylor expansion assumption \textit{(ii)}, we have:
    \[
    \E_{P_0} [\rho(\cdot, \theta^* + h_n/\sqrt{n}) - \rho(\cdot, \theta^*)] = \frac{1}{2n} h_n^\top V_{\theta^*} h_n + o_{P_0}(n^{-1}).
    \]
    Plugging this into the expansion above, we obtain:
    \[
    \sum_{i=1}^n \alpha_i [\rho(X_i, \theta^* + h_n/\sqrt{n}) - \rho(X_i, \theta^*)] - \textbf{G}_n^\alpha (h_n^\top \psi_{\theta^*}) - \frac{1}{2} \left( \frac{1}{n} \sum_{i=1}^n \alpha_i \right) h_n^\top V_{\theta^*} h_n = o_{P_0}(1).
    \]
    Hence, \textit{Weighted M-LAN} holds with centering sequence:
    \[
    \Delta_{n,\theta^*}^\alpha = n \left ( \sum_{i=1}^n \alpha_i \right )^{-1} V_{\theta^*}^{-1} \textbf{G}_n^\alpha (\psi_{\theta^*}) = V_{\theta^*}^{-1} \sqrt{n} \left( \frac{1}{\sum_{i=1}^n \alpha_i} \sum_{i=1}^n \alpha_i \psi_{\theta^*}(X_i) - \E_{P_0} \psi_{\theta^*} \right).
    \]
    Note that in Assumption ~\ref{assump::WMLAN}, we take a supremum in \( h \) over a compact set, while here we take a sequence \( h_n\) inside a compact set to prove the statement. These two are actually equivalent.  We finish the proof by showing that we can actually take slightly different centering sequence which is centered around the weighted M-estimator.
    Define the weighted empirical sum
    \[
    M_n^\alpha(\theta) = \frac{1}{n}\sum_{i=1}^n \alpha_i\,\rho(X_i,\theta),
    \qquad 
    \mestimator\in\arg\min_{\theta\in\Theta} M_n^\alpha(\theta).
    \]
    By (i)–(ii), the map $\theta\mapsto M_n^\alpha(\theta)$ is locally Lipschitz, differentiable at $\theta^*$ in $P_0$–probability with gradient
    \[
    \nabla M_n^\alpha(\theta^*) 
    = \frac{1}{n}\sum_{i=1}^n \alpha_i\Bigl\{\psi_{\theta^*}(X_i)-\E_{P_0}\psi_{\theta^*}(X)\Bigr\},
    \]
    and admits the quadratic expansion
    \[
    M_n^\alpha(\theta^*+h/\sqrt{n})
    = M_n^\alpha(\theta^*) 
    + \frac{1}{\sqrt{n}}\nabla M_n^\alpha(\theta^*)^\top h 
    + \frac{\bar\alpha_n}{2n}h^\top V_{\theta^*}h 
    + o_p(n^{-1}) ,
    \]
    uniformly for $\|h\|\le R$. The same arguments as in the proof of \cite[Thm.\ 5.23]{Vaart_1998}, using the assumption \( (iii) \), now give
    \begin{equation}
    \label{eq:weighted-M-linearization}
        \sqrt{n}(\mestimator - \theta^*) 
        = 
        \frac{1}{\bar\alpha_n} V_{\theta^*}^{-1} \frac{1}{\sqrt{n}}\sum_{i=1}^n \alpha_i\Bigl\{\psi_{\theta^*}(X_i)-\E_{P_0}\psi_{\theta^*}(X)\Bigr\}
        + o_p(1).
    \end{equation}
    Comparing \eqref{eq:weighted-M-linearization} with the Weighted M-LAN display derived above,
    \[
        \Delta_{n,\theta^*}^{\alpha}
        = 
        V_{\theta^*}^{-1}\sqrt{n} \left(
        \frac{1}{\sum_{i=1}^n\alpha_i}\sum_{i=1}^n \alpha_i \psi_{\theta^*}(X_i) - \E_{P_0}\psi_{\theta^*}(X)\right),
    \]
    and using $\sum_{i=1}^n\alpha_i=n\,\bar\alpha_n$, we obtain the identity
    \[
        \Delta_{n,\theta^*}^{\alpha}
        = 
        \sqrt{n}\,(\mestimator - \theta^*) + o_p(1).
    \]
    Therefore, as in Assumption~\ref{assump::WMLAN} we may \emph{take the centering sequence} to be the re-centered and re-scaled weighted M-estimator,
    \[
        \sqrt{n}\,(\mestimator - \theta^*),
    \]
    which completes the proof.
\end{proof}
\begin{lemma}  
    \label{lemma:ballconvergence}
    Assume there exists $\delta > 0$ such that the prior density $\pi$ is continuous and positive on $B_{\theta^*}(\delta)$ and that Assumption~\ref{assump::WMLAN} holds.
    For any $\eta, \epsilon > 0$, there exists a sequence $r_n \rightarrow + \infty$ and an integer $N(\eta, \epsilon) > 0$ such that for all $n > N(\eta, \epsilon)$, 
    \[
        \mathbb{P}_{P_0} \left ( \underset{g,h \in \overline{B}_{\textbf{0}}(r_n) }{\sup} f_n(g, h) > \eta \right) \leq \epsilon.
    \]
    where $f_n(g,h)$ is defined in \eqref{eq::BvM-def-of-fn} and $\overline{B}_{\textbf{0}}(r_n)$ denotes a closed ball of radius $r_n$ around $\mathbf{0}$.
\end{lemma}

\begin{proof}
We will follow the proof of Lemma 6 in Appendix B from \cite{avellamedinaetal2021}. 
We will first prove the claim for fixed $r > 0$.
First notice that for any $r > 0$, there exists an integer $N_0(r) > 0$ such that $\theta ^ * + \frac{h}{\sqrt{n}} \in B_{\theta ^ *} (\delta)$ whenever $h \in \overline{B}_{\textbf{0}}(r)$ and $n \geq N_0(r)$. 
Using the notation from \Cref{thm::BvM}, for any two sequences $\{ h_n \}, \{ g_n \} \in \overline{B}_{\textbf{0}}(r)$ and $n > N_0(r)$, we have
\begin{equation*}
\begin{split}
    \frac{\pi^{\mathrm{WMLAN}}_{n}(g_n \mid \wempirical)} {\pi^{\mathrm{WMLAN}}_{n}(h_n \mid \wempirical)} 
    &=
    \frac{\exp\left ( - \sum_{i = 1}^n \alpha_i \loss(X_i,  \theta ^ * + \frac{g_n}{\sqrt{n}}) \right ) \pi(\theta ^ * + \frac{g_n}{\sqrt{n}})}{\exp\left ( - \sum_{i = 1}^n \alpha_i \loss(X_i,  \theta ^ * + \frac{h_n}{\sqrt{n}}) \right ) \pi(\theta ^ * + \frac{h_n}{\sqrt{n}})} \\
    &=
    \frac{\exp\left ( - \sum_{i = 1}^n \alpha_i \loss(X_i,  \theta ^ * + \frac{g_n}{\sqrt{n}}) \right ) \pi_n(g_n)}{\exp\left ( - \sum_{i = 1}^n \alpha_i \loss(X_i,  \theta ^ * + \frac{h_n}{\sqrt{n}}) \right ) \pi_n(h_n)}
\end{split}
\end{equation*}
Defining
\[
    s_n(h_n) = \exp \left ( -\sum_{i=1}^n \alpha_i \bigl ( \loss (X_i, \theta^* + h_n/\sqrt{n}) - \loss (X_i, \theta^*) \bigr ) \right ),
\]
we have, following the definition \eqref{eq::BvM-def-of-fn}, 
\[
    f_n(g_n, h_n) \equiv \left \{ 1 - \frac{\phi_n(h_n)}{\pi^{\mathrm{WMLAN}}_{n}(h_n \mid \wempirical)} \frac{\pi^{\mathrm{WMLAN}}_{n}(g_n \mid \wempirical)}{\phi_n(g_n)}\right \} ^ +
    = \left \{ 1 - \frac{\phi_n(h_n)}{\phi_n(g_n)} \frac{s_n(g_n)}{s_n(h_n)}\frac{\pi_n(g_n)}{\pi_n(h_n)} \right \} ^ +.
\]
Now, for any sequence $h_n \in \overline{B}_{\textbf{0}}(r)$, Assumption~\ref{assump::WMLAN} implies
\[
\begin{split}
    \log (s_n (h_n)) &= -\sum_{i=1}^n \alpha_i \bigl ( \loss (X_i, \theta^* + h_n/\sqrt{n}) - \loss (X_i, \theta^*) \bigr ) \\
    &= 
    h_n^\top \left ( \frac{1}{n} \sum_{i = 1}^n \alpha_i \right ) V_{\theta^*} \Delta^{\boldsymbol{\alpha}}_{n, \theta^*} - \frac{1}{2} h_n^\top \left ( \frac{1}{n} \sum_{i = 1}^n \alpha_i \right ) V_{\theta^*} h_n + o_{P_0}(1),
\end{split}
\]
and we can see that
\[
    \log (\phi_n (h_n)) =
    - \frac{p}{2} \log (2 \pi) + \frac{1}{2} \log (\det ( \frac{V_{\theta*}}{n} \sum_{i = 1}^n \alpha_i )) - \frac{1}{2} (h_n - \Delta^{\boldsymbol{\alpha}}_{n, \theta^*}) ^ \top \left ( \frac{V_{\theta*}}{n} \sum_{i = 1}^n \alpha_i \right )  (h_n - \Delta^{\boldsymbol{\alpha}}_{n, \theta^*}),
\]
and therefore
\begin{equation*}
    \log \left ( \frac{s_n(h_n)}{\phi_n(h_n)} \right ) =
    \frac{p}{2} \log (2 \pi) - \frac{1}{2} \log (\det ( \frac{V_{\theta*}}{n} \sum_{i = 1}^n \alpha_i )) + \frac{1}{2} {(\Delta^{\boldsymbol{\alpha}}_{n, \theta^*})}^\top \left ( \frac{V_{\theta*}}{n} \sum_{i = 1}^n \alpha_i \right ) \Delta^{\boldsymbol{\alpha}}_{n, \theta^*} + o_{P_0}(1).
\end{equation*}
Furthermore, for a sequence $g_n \in \overline{B}_\mathbf{0}(r)$, let
\[
    b_n (g_n, h_n) \equiv \frac{\phi_n (h_n) s_n (g_n) \pi_n (g_n)}{\phi_n (g_n) s_n(h_n) \pi_n (h_n)}
\]
Now, by simple application of previous calculations and the fact that $\pi_n(g_n), \pi_n(h_n) \rightarrow \pi (\theta^*)$ as $n \rightarrow \infty$, we conclude that
\begin{equation*}
    \log (b_n (g_n, h_n)) = o_{f_{0, n}}(1).
\end{equation*}
Now, since $h_n$ and $g_n$ are arbitrary sequences in $B_{\mathbf{0}}(r)$, the above conclusion is equivalent to saying that for any fixed $r$, there exists $\Tilde{N}_0 (r, \epsilon, \eta)$ such that for all $n > \max \{ \Tilde{N}_0 (r, \epsilon, \eta), N_0(r) \}$ we have
\[
    \mathbb{P}_{P_0} \left ( \underset{g,h \in \overline{B}_{\textbf{0}}(r_n) }{\sup} | \log b_n(g_n, h_n) | > \eta  \right ) \leq \epsilon.
\]
Finally, for all $n > \max \{ \Tilde{N}_0 (r, \epsilon, \eta), N_0(r) \}$,
\begin{align*}
    \mathbb{P}_{P_0} \left ( \underset{g,h \in \overline{B}_{\textbf{0}}(r_n) }{\sup} f_n(g, h) > \eta \right) 
    & \leq \mathbb{P}_{P_0} \left ( \underset{g_n,h_n \in \overline{B}_{\textbf{0}}(r_n) }{\sup} f_n(g_n, h_n) > \eta \right) \\
    &\leq \mathbb{P}_{P_0} \left ( \underset{g,h \in \overline{B}_{\textbf{0}}(r_n) }{\sup} | \log b_n(g_n, h_n) | > \eta  \right ) 
    \leq \epsilon,
\end{align*}
where we used the fact that the mapping $x \mapsto |\log (1 - x)| - x$ is increasing for $x \in [0,1]$ and has value 0 at $x = 0$.

For a general sequence $r_n$, the result follows from exactly the same arguments as in Step 2 of Lemma 5 in \cite{avellamedinaetal2021}.
\end{proof}

\begin{lemma}[Loss function of the reweighted posterior: Gaussian case]
\label{lemma::calculation-reweighted-posterior-normal}
    Suppose \(X \mid \theta \sim N(\theta,1)\), and let the prior on the weights be \(\alpha \sim \Gamma(\kappa,\lambda)\) with shape parameter \(k\) and rate parameter \(\lambda>0\).
    Then the corresponding reweighted posterior coincides with an M-posterior with loss function
    \[
        \rho(X, \theta) 
        = \kappa\left[ \log \left(\lambda + \tfrac{1}{2}(X-\theta)^2 + \tfrac{1}{2}\log(2\pi)\right) - \log \lambda \right].
    \]
\end{lemma}
\begin{proof}
    The Gaussian likelihood for one observation is
    \[
        L(\theta;X) 
        = (2\pi)^{-1/2} \exp \left(-\tfrac{1}{2}(X-\theta)^2\right).
    \]
    Raising this likelihood to a power \(\alpha > 0\) yields
    \[
        L(\theta;X)^\alpha 
        = (2\pi)^{-\alpha/2} \exp \left(-\tfrac{\alpha}{2}(X-\theta)^2\right).
    \]
    The reweighted likelihood is defined by taking the expectation with respect to \(\alpha \sim \Gamma(\kappa,\lambda)\), whose density is
    \[
        \pi(\alpha) = \frac{\lambda^\kappa}{\Gamma(\kappa)} \alpha^{\kappa-1} e^{-\lambda \alpha}, 
        \qquad \alpha>0.
    \]
    Thus
    \[
        M(\theta;X) := \mathbb{E}_\alpha[L(\theta;X)^\alpha]
        = \int_0^\infty (2\pi)^{-\alpha/2} \exp\left(-\tfrac{\alpha}{2}(X-\theta)^2\right) \frac{\lambda^\kappa}{\Gamma(\kappa)} \alpha^{\kappa-1} e^{-\lambda \alpha}\, d\alpha.
    \]
    Combining terms in the exponential, we obtain
    \[
        M(\theta;X)
        = \frac{\lambda^\kappa}{\Gamma(\kappa)} \int_0^\infty \alpha^{\kappa-1} \exp \Bigl(-\alpha\bigl[\lambda + \tfrac{1}{2}(X-\theta)^2 + \tfrac{1}{2}\log(2\pi)\bigr]\Bigr)\, d\alpha.
    \]
    Set
    \[
        \lambda'(\theta;X) := \lambda + \tfrac{1}{2}(X-\theta)^2 + \tfrac{1}{2}\log(2\pi).
    \]
    The integral is then
    \[
        \frac{\lambda^\kappa}{\Gamma(\kappa)} \int_0^\infty \alpha^{\kappa-1} e^{-\lambda'\alpha}\, d\alpha.
    \]
    This integral is recognized as the normalizing constant of a Gamma distribution:
    \[
        \int_0^\infty \alpha^{\kappa-1} e^{-\lambda'\alpha}\, d\alpha = \frac{\Gamma(\kappa)}{(\lambda')^\kappa}.
    \]
    Hence,
    \[
        M(\theta;X) = \left(\frac{\lambda}{\lambda'}\right)^\kappa.
    \]
    Finally, the effective loss is defined as minus the logarithm of this marginal:
    \[
        \rho(X, \theta) := -\log M(\theta;X) = \kappa\left[\log(\lambda') - \log \lambda\right].
    \]
    Explicitly,
    \[
        \rho(X, \theta) = \kappa\left[ \log \left(\lambda + \tfrac{1}{2}(X-\theta)^2 + \tfrac{1}{2}\log(2\pi)\right) - \log \lambda \right].
    \]
    This completes the proof.
\end{proof}
\begin{lemma}[Loss function of the reweighted posterior: Exponential case]
\label{lemma::calculation-reweighted-posterior-exponential}
    Suppose \(X \mid \theta \sim \mathrm{Exp}(\theta)\) with rate parameter \(\theta>0\), 
    and let the prior on the weights be \(\alpha \sim \Gamma(\kappa,\lambda)\) 
    with shape parameter \(\kappa\) and rate parameter \(\lambda>0\).
    Then the corresponding reweighted posterior coincides with an M-posterior with loss function
    \[
        \rho(X, \theta) 
        = \kappa\left[ \log \left(\lambda + \theta X - \log \theta \right) - \log \lambda \right],
    \]
    provided that \(\lambda + \theta X - \log\theta > 0\).
\end{lemma}
\begin{proof}
    The exponential likelihood for one observation is $L(\theta;X) = \theta e^{-\theta X},$ for $X \geq 0$ and $\theta>0$.
    Raising this likelihood to a power \(\alpha > 0\) yields $
        L(\theta;X)^\alpha = \theta^\alpha \exp(-\alpha \theta X).$
    The reweighted likelihood is defined by taking the expectation with respect to 
    \(\alpha \sim \Gamma(k,\lambda)\), whose density is
    \[
        \pi(\alpha) = \frac{\lambda^\kappa}{\Gamma(\kappa)} \alpha^{\kappa-1} e^{-\lambda \alpha}, 
        \qquad \alpha>0.
    \]
    Hence,
    \[
        M(\theta;X) := \mathbb{E}_\alpha[L(\theta;X)^\alpha]
        = \int_0^\infty \theta^\alpha e^{-\alpha \theta X}\, 
          \frac{\lambda^\kappa}{\Gamma(\kappa)} \alpha^{\kappa-1} e^{-\lambda \alpha}\, d\alpha.
    \]
    Combining terms, note that \(\theta^\alpha = e^{\alpha \log \theta}\). 
    Therefore the exponential term can be written as
    $e^{-\alpha(\lambda + \theta X - \log \theta)}.$
    Hence
    \[
        M(\theta;X) 
        = \frac{\lambda^\kappa}{\Gamma(\kappa)} \int_0^\infty \alpha^{\kappa-1} 
          \exp\!\Bigl(-\alpha\bigl[\lambda + \theta X - \log \theta \bigr]\Bigr)\, d\alpha.
    \]
    Set
    $\lambda'(\theta;X) := \lambda + \theta X - \log \theta.$
    The integral is then
    \[
        \frac{\lambda^\kappa}{\Gamma(\kappa)} \int_0^\infty \alpha^{\kappa-1} e^{-\lambda' \alpha}\, d\alpha.
    \]
    This is recognized as the normalizing constant of a Gamma distribution:
    \[
        \int_0^\infty \alpha^{\kappa-1} e^{-\lambda' \alpha}\, d\alpha = \frac{\Gamma(\kappa)}{(\lambda')^\kappa}.
    \]
    Thus,
    \[
        M(\theta;X) = \left(\frac{\lambda}{\lambda'}\right)^\kappa.
    \]
    Finally, the effective loss is defined as minus the logarithm of this marginal:
    \[
        \rho(X, \theta) := -\log M(\theta;X) = \kappa\left[\log(\lambda') - \log \lambda\right].
    \]
    Explicitly,
    $\rho(X, \theta) 
        = \kappa\left[ \log \left(\lambda + \theta X - \log \theta \right) - \log \lambda \right].$
    This completes the proof.
\end{proof}
%
%
\section{Additional Results} \label{appendix::additional-results}

\subsection{Additional BvM results} \label{asec::additional-BvM-results}

In a standard application setting, the weights are sometimes defined as random quantities rather than positive constants. Therefore, we show that the BvM result also holds when the weights are drawn independently from a prior distribution with a finite second moment. We again find convergence in total variation in a product probability between the external probability and the prior probability on weights.
\begin{proposition}
\label{prop::BvMrandom}
    Suppose that the prior density $\pi$ is continuous and positive on a neighborhood around the true parameter $\theta^*$ and Assumptions \ref{assump::WMLAN} and \ref{assump::concentration} hold.
    Furthermore, let $\alpha_i$ be drawn i.i.d.\ from $\pi_\alpha$ which has a finite second moment. Then
    $$d_{\mathrm{TV}} \left( \mposterior(\cdot \mid \wempirical), \, \phi(\cdot \mid \mestimator, V_{\theta^*}^{-1} / (\overline{\alpha}_n n)) \right) \to 0,$$
    in $P_0 \times \pi_\alpha^{(n)}$-probability, where $d_{\mathrm{TV}}(\cdot, \cdot)$ denotes the total variation distance and $V_{\theta^*}$ is the positive definite matrix satisfying Assumption \ref{assump::WMLAN}.
\end{proposition}
\begin{proof}
    Let $C = \mathbb{E}_{\pi_\alpha} \left [ \alpha_1^2 \right ]$, which is finite by assumption.
    Pick $\delta > 0$ arbitrary. Define 
    \[
    B_n = \left \{ \frac{1}{n} \sum_{i = 1}^n \alpha_i^2 < C + \delta \right \}.
    \]
    As before, let $\phi_n(h) \equiv n ^ {- \frac{1}{2}} \phi (h \mid \mestimator, V_{\theta^*}^{-1} / \overline{\alpha}_n)$
    and $\pi^{\mathrm{WMLAN}}_{n}(h \mid \wempirical) \equiv n ^ {- \frac{1}{2}} \mposterior (\theta ^ * + \frac{h}{\sqrt{n}} \mid \wempirical)$. We now have
    \begin{align*}
        \E_{P_0 \times \pi_\alpha^{(n)}} [ d_{\mathrm{TV}} & ( \pi^{\mathrm{WMLAN}}_{n}(\cdot \mid \wempirical),  \, \phi_n(\cdot)  )  ] =  \E_{\pi_\alpha^{(n)}} [ \E_{P_0}  [ d_{\mathrm{TV}}  ( \pi^{\mathrm{WMLAN}}_{n}(\cdot \mid \wempirical),  \phi_n(\cdot) ) ]] \\
        &= \E_{\pi_\alpha^{(n)}}  [ \E_{P_0}  [ d_{\mathrm{TV}}  ( \pi^{\mathrm{WMLAN}}_{n}(\cdot \mid \wempirical), \phi_n(\cdot) )  \Big |  (\alpha_i)_{i = 1}^{n} ]] \\
        &\leq \E_{\pi_\alpha^{(n)}} \hspace{-1pt} [ \E_{P_0}  [ d_{\mathrm{TV}} (\pi^{\mathrm{WMLAN}}_{n}(\cdot \mid \wempirical),  \phi_n(\cdot)  ) \mathbf{1} \{ B_n \} \big | (\alpha_i)_{i = 1}^{n} \hspace{-1pt}] + \mathbb{P}_{\pi_\alpha^{(n)}} \hspace{-1pt}( B_n^c ).
    \end{align*}
    By assumptions, \Cref{thm::BvM} (see \eqref{eqn::last-line-of-BvM-theorem}) and Dominated Convergence, the first term on RHS goes to zero as $n \rightarrow \infty$.
    Furthermore, by the Weak Law of Large Numbers, the second term goes to zero as well.
    Finally, the standard application of Markov's inequality yields the desired result.
\end{proof}

The proposition establishes that the weighted M-posterior concentrates around the random weighted M-estimator. It is important to emphasize, however, that this result does not apply to the framework in \Cref{subsection::Bayesian-Data-Reweighting}. Here, weights are drawn independently and never inferred. In this setting, the weights alone do not provide robustness, so any robustification of the posterior must come from the choice of a robust loss function. Recall that in \eqref{def::alpha-posterior} we defined the marginalized posteriors of the form
$
    \pi_\alpha(\theta \mid X^n) := \int \pi(\theta, \alpha^n \mid X^n) \, d\alpha^n.
$
We then showed that this type of posterior actually corresponds to the M-posterior for a certain loss that depends on the likelihood and the prior on the weights. Hence, the asymptotic properties of this object can be described by \Cref{thm::BvM}. To conclude this section, we state a different BvM result for the above posteriors, showing the concentration around a mixture of Gaussian random variables:

\begin{lemma}
\label{lemma::BvM-for-Reweighted-Posterior}
Assume conditions from \Cref{prop::BvMrandom}. Then, with denoting the weighted MLE with $\hat{\theta}_{n, \boldsymbol{\alpha}} = \operatorname{argmax}_{\theta \in \Theta} \sum_{k=1}^n {\alpha_k} \log f(X_k \mid \theta) $, we have
    $$d_{TV} ( \pi_\alpha(\theta \mid X^n), \, \mathbb{E}_{\alpha^n \sim \pi(\alpha^n \mid X^n)} \phi(\cdot \mid \hat{\theta}_{n, \boldsymbol{\alpha}}, V_{\theta^*}^{-1} / (\overline{\alpha}_n n) ) \to 0,$$
    in $P_0$-probability.
\end{lemma}

\begin{proof}
    By definition, the marginalized posterior is obtained by integrating over the random weights:
    \[
        \pi_\alpha(\theta \mid X^n) 
        = \int \mposterior(\theta \mid \wempirical)\, \pi(\alpha^n \mid X^n)\, d\alpha^n.
    \]
    Now, for any family of distributions $(Q_\alpha)_{\alpha}$ and $(P_\alpha)_{\alpha}$ and any probability measure $\mu$ on the index set,
    \[
        d_{\mathrm{TV}}\!\left(\int Q_\alpha \, d\mu(\alpha), \, \int P_\alpha \, d\mu(\alpha)\right) 
        \leq \int d_{\mathrm{TV}}(Q_\alpha, P_\alpha) \, d\mu(\alpha).
    \]
    This inequality follows directly from the definition of total variation distance and Jensen’s inequality.  

    Applying this inequality with $Q_\alpha = \mposterior(\cdot \mid \wempirical)$, $P_\alpha = \phi(\cdot \mid \hat{\theta}_{n,\boldsymbol{\alpha}}, V_{\theta^*}^{-1}/(\overline{\alpha}_n n))$, and $\mu = \pi(\alpha^n \mid X^n)$, we obtain
    \[
    \begin{split}
        d_{\mathrm{TV}}\!\bigl ( \pi_\alpha(\theta \mid X^n), \,
        \E_{\alpha^n \sim \pi(\alpha^n \mid X^n)} & \phi(\cdot \mid \hat{\theta}_{n, \boldsymbol{\alpha}}, V_{\theta^*}^{-1}/(\overline{\alpha}_n n)) \bigr) \\
        & \leq \E_{\alpha^n \sim \pi(\alpha^n \mid X^n)} 
        d_{\mathrm{TV}}\!\left(\mposterior(\cdot \mid \wempirical), \phi(\cdot \mid \hat{\theta}_{n,\boldsymbol{\alpha}}, V_{\theta^*}^{-1}/(\overline{\alpha}_n n))\right).
    \end{split}
    \]
    The expectation on the right  side converges to zero by Proposition~\ref{prop::BvMrandom} and the Dominated Convergence Theorem.  

    As claimed, we therefore conclude that in $P_0$-probability
    \[
        d_{\mathrm{TV}}\!\left(\pi_\alpha(\theta \mid X^n), \,
        \E_{\alpha^n \sim \pi(\alpha^n \mid X^n)} \phi(\cdot \mid \hat{\theta}_{n, \boldsymbol{\alpha}}, V_{\theta^*}^{-1}/(\overline{\alpha}_n n)) \right) \to 0.
    \]
\end{proof}

Note that we expect robustness to arise in the mixture of Gaussians setting because the weights are sampled from the posterior distribution conditional on the data, rather than being drawn independently. This dependence on the observed data differentiates the setup from the i.i.d.\ weighting scheme and is precisely what enables robustification of the posterior.

\subsection{Breakdown point in higher dimensions}
\label{app:BP_general_d}
Here we state the result in the multivariate setting, showing how different classes of priors on $\mathbb{R}^d$ affect the robustness of the M-posterior. 
We say that a prior density $\pi$ on $\mathbb{R}^d$ has \emph{exponential-like tails} if it is of the form
\(
    \pi(\theta) \propto \exp(-h(\|\theta\|)),
\)
where $h:[0,\infty)\to\mathbb{R}_+$ is convex, symmetric in $\|\theta\|$, and has a bounded derivative $h'$. 
We say that $\pi$ has \emph{lighter-than-exponential tails} if it is of the form
\(
    \pi(\theta) \propto \exp(-h(\|\theta\|)),
\)
with $h$ convex and symmetric in $\|\theta\|$, but with unbounded derivative $h'$.

\begin{theorem} \label{thm::convex-loss-higher-dim}
    Let $\widetilde \rho$ be symmetric and convex with a score function $\widetilde\psi = {\widetilde\rho}'$ that is bounded. Let \( \loss \) be the radial loss of the form \( \loss(x) := \widetilde \loss (\| x\|)\). Assume an uninformative prior \( \pi = 1\), then the breakdown point of the M-posterior induced by loss \(\rho\) is equal to \( 1/2\). If the prior has exponential-like tails, then the breakdown point is at least \( 1/2\), and decreases down to \( 1/2\) as sample size grows. Finally, if the prior has tails lighter than those of the exponential distribution, the breakdown point of the posterior does not exist.
\end{theorem}
\begin{proof}
    The proof closely follows claims from the proof of \Cref{thm::breakdown-point-convex-loss}.
    Write $\theta = r_\theta\,v$ with $r_\theta=\|\theta\|$ and $v=\theta/\|\theta\|$ when $\theta\neq 0$.
    For any fixed $\theta$ and any unit vector $u$, writing \( c = \sup \widetilde \psi (x) < \infty\), and using the fact that $\|Ru-\theta\|=R-\langle u,\theta\rangle+o(1)$ as $R\to\infty$, we have
    \begin{equation} \label{eqn::multiple-d-limit-in-x}
        \lim_{R\to\infty}\big\{\rho(Ru-\theta)-\rho(Ru)\big\} = - c \langle u,\theta\rangle.
    \end{equation}
    Also, for any fixed sample point $x$ and $\theta=r_\theta v$ with $r_\theta\to\infty$,
    \begin{equation} \label{eqn::multiple-d-limit-in-theta}
        \rho(x-\theta)-\rho(x)
        =
        c r_\theta + o(r_\theta).
    \end{equation}
    Similar as before, let 
    \begin{equation*}
        \widetilde T_1 = \frac{\int_{\R^d} \| \theta \| \pi(\theta) \exp \left \{-\sum_{i=1}^{n-m} \bigl ( \loss(X_i - \theta) - \loss(X_i) \bigr ) - \sum_{i=n - m +1}^{n} \bigl ( \loss(X_i' - \theta) - \loss(X_i') \bigr ) \right \} \, d\theta}{\int_{\R^d} \pi(\theta) \exp \left \{-\sum_{i=1}^{n-m} \bigl ( \loss(X_i - \theta) - \loss(X_i) \bigr ) - \sum_{i=n - m +1}^{n} \bigl ( \loss(X_i' - \theta) - \loss(X_i') \bigr ) \right \} \, d\theta}.
    \end{equation*}
    be the expected norm under the contaminated M-posterior, where we contaminated \( m \) samples. Also, in the same fashion, let
        \begin{equation*}
        \widetilde T_2 = \frac{\int_{\R^d} \| \theta \|^2 \pi(\theta) \exp \left \{-\sum_{i=1}^{n-m} \bigl ( \loss(X_i - \theta) - \loss(X_i) \bigr ) - \sum_{i=n - m +1}^{n} \bigl ( \loss(X_i' - \theta) - \loss(X_i') \bigr ) \right \} \, d\theta}{\int_{\R^d} \pi(\theta) \exp \left \{-\sum_{i=1}^{n-m} \bigl ( \loss(X_i - \theta) - \loss(X_i) \bigr ) - \sum_{i=n - m +1}^{n} \bigl ( \loss(X_i' - \theta) - \loss(X_i') \bigr ) \right \} \, d\theta}.
    \end{equation*}
    denote the second moment.
    Exactly as in the one-dimensional monotonicity argument (\Cref{lemma::posterior-moments-monotone}), we see that the biggest bias to these two quantities is achieved by corrupting samples such that \( \| X_i' \| = \infty\).  
    
    \textbf{Improper prior.} For those corrupted points, let \( u_i' \) denote the direction vector (of unit norm) of those points.  Now, recall that we have $\theta = r_\theta\,v$ with $r_\theta=\|\theta\|$ and $v=\theta/\|\theta\|$. Therefore, from \eqref{eqn::multiple-d-limit-in-x} and \eqref{eqn::multiple-d-limit-in-theta}, for large \( r_\theta\), we have that the exponent in the first two moments from the above is equal to
    \[
        - \left ((n - m ) \cdot c + o(1) \right ) \cdot r_\theta + \big ( c \sum_{i=n - m +1}^{n} \langle v, u_i \rangle \big ) \cdot r_\theta.
    \]
    Now, we see that \( \widetilde T_1 \) and \( \widetilde T_2\) are finite as long as \( n -m > \sum_{i=n - m +1}^{n} \langle v, u_i' \rangle \). Now, since \( v \) and \( u_i'\) are unit vectors, we have that \( \left | \sum_{i=n - m +1}^{n} \langle v, u_i' \rangle \right | \leq m \) and hence if \( n - m >m\), \( \widetilde T_1 \) and \( \widetilde T_2\) are finite. By the multidimensional equivalent of \Cref{lemma::bounds-on-W2}, we conclude that \( \varepsilon^*_{W_2}(\pi^\rho_n, X^n) \geq \frac{1}{2}\).

    To get the lower bound on the breakdown point, notice that the posterior mean is translation equivariant (as shown in the proof of \Cref{thm::breakdown-point-redescending-loss}).
    It follows that the breakdown point of the posterior mean is at most 1/2 \citep{donoho:huber1983}. Now, from the multidimensional lower bound similar to \Cref{lemma::bounds-on-W2}, we have that the breakdown point of the posterior is upper bounded by the breakdown point of the mean; therefore, \( \varepsilon^*_{W_2}(\pi^\rho_n, X^n) \leq \frac{1}{2}\). Combining this with the upper bound, we conclude that the breakdown point of the M-posteriors is equal to \(1/2\).

    \textbf{Super-exponential and Sub-exponential prior.} Same reasoning as in the proof of \Cref{thm::breakdown-point-convex-loss} and using the adaptation to higher dimension from above.
\end{proof}

\subsection{Additional examples} \label{asec::additional-examples}
\begin{example}[Reweighted posterior: Gaussian model]
\label{example:gamma_weights}
    Consider the setup of \Cref{subsection::Bayesian-Data-Reweighting} and suppose that we have a model where \( X_i \mid \theta \overset{i.i.d.}{\sim} N(\theta, 1) \) and let the prior on \( \theta \) be \( \pi(\theta) = N(\mu_0,\sigma_0^2) \).
    Furthermore, let the prior on the weights be \( \pi_\alpha = \Gamma(\kappa, \lambda) \).
    A direct calculation (see \Cref{lemma::calculation-reweighted-posterior-normal}) shows that the corresponding loss for this M-posterior is
    \[
        \loss(x, \theta) = \kappa \left [ \log \left ( \lambda + \frac{1}{2}(\theta-x)^2 + \frac{1}{2}\log(2\pi) \right ) - \log \lambda \right ],
    \]
    with score function
    \[
        \psi(x, \theta) = \frac{\kappa(\theta - x)}{\lambda + \frac{1}{2} (\theta - x)^2 + \frac{1}{2}\log(2\pi)}.
    \]
    Furthermore, if the data is actually drawn from \( X_i \overset{i.i.d.}{\sim} N(\theta^*, 1) \), due to the symmetry of the score function \( \psi \), the loss \( \loss \) is Fisher consistent:
    \[
        \E_{X \sim N(\theta^*, 1)} [ \psi(X, \theta^*)  ] = 0.
    \]
    Hence, as in the previous example, the (reweighted-)M-posterior \( \mposterior(\cdot \mid F_n)\) will concentrate around the true model parameter.
\end{example}
\begin{example}[Huber-skip loss] \label{example::Huber-skip-loss}
    We take a look at an example where we have a bounded posterior influence function, but the posterior mean does not exist; hence, we cannot even define the influence function of the posterior mean.  Consider the Huber-skip loss
    \[ \loss(x,\theta) = \min ((x-\theta)^2, 1), \] 
    so that \( 0 \leq \rho(x,\theta) \leq 1. \) Furthermore, take the Cauchy prior over the parameter space: 
    \( \pi(\theta) = (1 + \theta^2)^{-1}. \)
    Hence, by examining the formula of the posterior influence function in \eqref{eqn::exact-PIF}, we immediately get
    \begin{equation*}
        \bigl|\mathrm{PIF}(x_0;\,\theta,\rho,F_n)\bigr| \leq 4n\sup_{\theta\in\Theta}\pi_n^\rho(\theta\mid F_n)<\infty,
    \end{equation*}
    because the loss is lower-bounded and the Cauchy prior is upper-bounded. In other words,  the M-posterior \( \mposterior (\cdot \mid F_n)\) is uniformly B-robust. Now, note that for a sufficiently large \( \theta \), we have that \( \loss(\theta, X_i) = 1 \) for all \( i \in [n]\). Hence \( \mposterior(\theta \mid X^n) \propto \pi(\theta)\) for large enough \( \theta\). Since we chose a Cauchy prior, this implies that the posterior mean does not exist. 
    
    This example is puzzling in the following sense: \Cref{thm::BvM} shows that this M-posterior \( \mposterior(\cdot \mid F_n)\) converges in total variation distance to a Gaussian centered at the M-estimator \(\hat{\theta}_{\loss}\), which is robust since the score function \(\psi\) is bounded. At the same time, the posterior mean of \( \mposterior(\cdot \mid F_n)\) does not exist. This once again shows the important interplay between the prior distribution \(\pi(\theta)\) and the robust loss \( \loss \).

\end{example}


\end{document}